\documentclass[a4paper, reqno, 12pt]{amsart}

\usepackage{fourier}	
\usepackage{amssymb, amsmath, amsthm, bm}

\usepackage[T1]{fontenc}
\usepackage[utf8]{inputenc}
\usepackage{graphicx}
\usepackage{comment}
\usepackage[linkbordercolor={0.8 0.9 0.9}, citebordercolor={0.85 0.85 0.85}]{hyperref}
\usepackage[normalem]{ulem}

\usepackage{amsmath, amsthm, amssymb} 
\usepackage[all]{xy}
\usepackage[usenames,dvipsnames]{color}
\usepackage{tikz}
\usepackage{tikz-cd}

\usepackage{cancel} 

\usepackage[capitalize, nameinlink]{cleveref}

\numberwithin{equation}{section}

\theoremstyle{plain}
\newtheorem{theorem}{Theorem}[section]
\newtheorem{proposition}[theorem]{Proposition}
\newtheorem{lemma}[theorem]{Lemma}
\newtheorem{corollary}[theorem]{Corollary}
\theoremstyle{definition}
\newtheorem{definition}[theorem]{Definition}
\newtheorem{example}[theorem]{Example}

\theoremstyle{remark}
\newtheorem{remark}[theorem]{Remark}

  {\par\begin{tabular}{rcl}}%
  {\end{tabular}\par}

\newcommand{\ds}{\displaystyle}

\newcommand{\defeq}{:=}

\newcommand{\ip}[1]{\langle #1 \rangle}

\newcommand{\into}{\hookrightarrow}
\newcommand{\onto}{\twoheadrightarrow}
\newcommand{\thus}{\Rightarrow}

\makeatletter
\newcommand*\slot{\mathpalette\slot@{.5}}
\newcommand*\slot@[2]{\mathbin{\vcenter{\hbox{\scalebox{#2}{$\m@th#1\bullet$}}}}}
\makeatother

\newcommand{\lie}{\mathfrak}

\newcommand{\CC}{\mathbb{C}}

\newcommand{\NN}{\mathbb{N}}

\newcommand{\TT}{\mathbb{T}}
\newcommand{\QQ}{\mathbb{Q}}
\newcommand{\ZZ}{\mathbb{Z}}

\newcommand{\cA}{\mathcal{A}}
\newcommand{\cB}{\mathcal{B}}

\newcommand{\cF}{\mathcal{F}}

\newcommand{\cL}{\mathcal{L}}
\newcommand{\cO}{\mathcal{O}}

\newcommand{\cU}{\mathcal{U}}

\newcommand{\bfA}{\mathbf{A}}

\newcommand{\SL}{\mathrm{SL}}
\newcommand{\SU}{\mathrm{SU}}

\newcommand{\bP}{\mathbf{P}}		
\newcommand{\bQ}{\mathbf{Q}}		

\DeclareMathOperator{\Hom}{Hom}

\DeclareMathOperator{\res}{res}
\DeclareMathOperator{\Res}{Res}

\DeclareMathOperator{\id}{id}

\DeclareMathOperator{\ev}{ev}

\DeclareMathOperator{\Vect}{span}

\newcommand{\param}{s}
\newcommand{\Uqg}{\cU_q(\lie{g})}
\newcommand{\Uqk}{\cU_q(\lie{k})}
\newcommand{\Uqh}{\cU_q(\lie{h})}
\DeclareMathOperator{\wt}{wt}

\newcommand{\Rhat}{\hat{R}}
\newcommand{\braid}{\sigma} 
\newcommand{\Cartan}{\eta}  
\newcommand{\rightend}{\mathsf{R}}
\newcommand{\leftend}{\mathsf{L}}

\newcommand{\OqG}{\cO_q[G]}

\newcommand{\OqAOG}{\cO_q^{\bfA_0}[G]}

\newcommand{\OqK}{\cO_q[K]}
\newcommand{\OT}{\cO[T]}

\newcommand{\OqAOK}{\cO_q^{\bfA_0}[K]}
\newcommand{\OGq}{\cO[G_q]}
\newcommand{\OKq}{\cO[K_q]}
\newcommand{\OKO}{\cO[K_0]}

\newcommand{\OYSq}{\cO[Y_{\rootset,q}]}
\newcommand{\OXSq}{\cO[X_{\rootset,q}]}
\newcommand{\OYSO}{\cO[Y_{\rootset,0}]}
\newcommand{\OXSO}{\cO[X_{\rootset,0}]}

\newcommand{\mult}{\mathsf{M}}

\newcommand{\simpleroots}{{\boldsymbol{\Delta}}}
\newcommand{\dominant}{\bP^+}
\newcommand{\roots}{{\boldsymbol{\Phi}}}
\newcommand{\fundweights}{{\boldsymbol{\Pi}}}
\newcommand{\rootset}{S}

\newcommand{\pid}{{\mathbf{A}_0}}
\newcommand{\pidinfty}{\mathbf{A}_\infty}
\newcommand{\laurent}{\mathbf{A}}
\newcommand{\lattice}{\cL}

\newcommand{\kasE}{\tilde{E}}
\newcommand{\kasF}{\tilde{F}}

\newcommand{\genf}{\mathsf{f}}
\newcommand{\genv}{\mathsf{v}}

\newcommand{\cpa}{\mathsf{f}}
\newcommand{\cpb}{\mathsf{g}}

\newcommand{\edge}{S}

\newcommand{\rank}{r}

\newcommand{\Cset}{\mathsf{C}}
\newcommand{\Sset}{\mathsf{S}}

\newcommand{\hgraph}{\Lambda_{\lie{g}, \Cset}}
\newcommand{\dmap}{\mathsf{d}}
\newcommand{\NC}{N}
\newcommand{\Vset}{\hgraph^0}

\newcommand{\KPalg}{\mathrm{KP}_\CC(\hgraph)}
\newcommand{\hgraphstd}{\Lambda_{\lie{g}}}
\newcommand{\Vsetstd}{\hgraphstd^0}
\newcommand{\KPalgstd}{\mathrm{KP}_\CC(\hgraphstd)}
\newcommand{\KPiso}{\Phi}

\newcommand{\CalgG}{\cA_{\lie{g}, \Cset}}
\newcommand{\CalgK}{\cA_{\lie{k}, \Cset}}
\newcommand{\CalgGpi}{\cA_{\lie{g}, \fundweights}}

\newcommand{\CalgGstd}{\cA_{\lie{g}}}
\newcommand{\CalgKstd}{\cA_{\lie{k}}}

\newcommand{\rmap}{\mathsf{r}}
\newcommand{\smap}{\mathsf{s}}

\newcommand{\counit}{\epsilon}
\newcommand{\hit}{\triangleright}
\newcommand{\hitby}{\triangleleft}

\newcommand{\SoibH}{\mathsf{H}}
\newcommand{\Soib}{\tilde\pi}
\newcommand{\SoibT}{\pi}
\newcommand{\repT}{\chi}
\newcommand{\rept}{\chi_t}
\newcommand{\inclusion}{\iota}
\newcommand{\projection}{\phi}

\newcommand{\Ktilde}{\widetilde{K}}

\newcommand{\OqKtilde}{\cO_q[\Ktilde]}

\newcommand{\qrange}{(0, \infty) \setminus\{1\}}

\newcommand{\OqGP}{\cO_q[G/N^+]}
\newcommand{\OqGM}{\cO_q[G/N^-]}
\newcommand{\OqGAOP}{\cO^{\mathbf{A}_0}_q[G/N^+]}
\newcommand{\OqGAOM}{\cO^{\mathbf{A}_0}_q[G/N^-]}
\newcommand{\OqKAO}{\cO^{\mathbf{A}_0}_q[K]}

\begin{document}

\title[Crystal limits of quantum groups as higher-rank graph algebras]{Crystal limits of compact semisimple quantum groups as higher-rank graph algebras}

\author{Marco Matassa}
\address{OsloMet – Oslo Metropolitan University, Oslo, Norway}
\email{marco.matassa@oslomet.no}

\author{Robert Yuncken}
\address{Institut Elie Cartan de Lorraine -- Universit\'e de Lorraine
 -- 3 rue Augustin Fresnel, 57000 Metz, France}
\email{robert.yuncken@univ-lorraine.fr}

\date{}

\begin{abstract}
Let $\OqK$ denote the quantized coordinate ring over the field $\CC(q)$ of rational functions corresponding to a compact semisimple Lie group $K$, equipped with its $*$-structure.  
Let $\pid\subset\CC(q)$ denote the subring of regular functions at $q = 0$.
We introduce an $\pid$-subalgebra $\OqAOK \subset \OqK$ which is stable with respect to the $*$-structure, and which has the following properties with respect to the crystal limit $q \to 0$.

The specialization of $\OqK$ at each $q\in\qrange$ admits a faithful $*$-repre\-sentation $\pi_q$ on a fixed Hilbert space, a result due to Soibelman.  We show that for every element $a \in \OqAOK$, the family of operators $\pi_q(a)$ admits a norm-limit as $q\to0$.  These limits define a $*$-representation $\pi_0$ of $\OqAOK$. We show that the resulting $*$-algebra $\OKO=\pi_0(\OqAOK)$ is a Kumjian-Pask algebra, in the sense of Aranda Pino, Clark, an Huef and Raeburn.  We give an explicit description of the underlying higher-rank graph in terms of crystal basis theory.  As a consequence, we obtain a continuous field of $C^*$-algebras $(C(K_q))_{q\in[0,\infty]}$, where the fibres at $q = 0$ and $\infty$ are explicitly defined higher-rank graph algebras.
\end{abstract}


\maketitle

\tableofcontents

\section{Introduction}

\subsection{Algebraic and analytic crystal limits}
Let $K$ be a compact semisimple Lie group, $G$  its complex form.
The quantized coordinate ring on $G$ can be constructed in two different flavours.
In the algebraic approach we define a $\CC(q)$-algebra $\OqG$ with $q$ an indeterminate, while in the analytic approach we define a $\CC$-algebra $\OGq$, or a $*$-algebra $\OKq$ if we are interested in the compact form, with $q \in \qrange$ being a numerical value.
In this paper, we study these algebras in the crystal limit $q \to 0$ (or $q \to \infty$).  This is defined quite differently in the two communities.

In the algebraic world, the $q=0$ limit is realized by the theory of crystal bases, due to Kashiwara \cite{Kashiwara:crystal1, Kashiwara:crystal2} and Lusztig \cite{Lusztig:canonical}.
Beginning with the quantized enveloping algebra $\Uqg$ of $\lie{g}=\lie{k}_\CC$ over $\CC(q)$ or $\QQ(q)$, one replaces the quantized Lie algebra generators $E_i$, $F_i$ by algebraic renormalizations $\kasE_i$, $\kasF_i$, called the Kashiwara operators, and then localizes the finite-dimensional modules at $q=0$.  The resulting localized modules admit nice ``crystal'' bases with an elegant combinatorial structure.  
This approach was used by \cite{Kashiwara:global, LeclercThibon, Iglesias:bitableaux} to define a crystal basis for the quantized coordinate rings $\OqG$.

In the analytical world, one is interested in the continuous fields of quantized coordinate rings, in the $C^*$-algebraic sense of Woronowicz \linebreak \cite{Woronowicz:SUq2, Woronowicz:pseudogroups} and Vaksman-Soibelman \cite{VakSoi:SUq2, Soibelman}, as well as the subalgebras of continuous functions on quantized flag varieties.  Here, we have a handful of results, in particular due to Hong and Szymanski \cite{HonSzy:spheres, HonSzy:lens}, which show that in certain examples ($\mathrm{SU}_q(2)$, quantum projective spaces, quantum spheres and quantum lens spaces), the continuous field of $C^*$-algebras over $(0,\infty)$ extends to $[0,\infty]$, with fibres at the boundaries given by graph $C^*$-algebras.\footnote{In fact, their results are stronger, showing that their quantized function algebras are isomorphic for all values of $q\neq1$, although only in the $q=0$ limit do the graph algebra generators coincide with matrix coefficients for the quantum group.}
A recent study of the $q=0$ limit of $C(\SU_q(n))$ has also been undertaken in \cite{GirPal}, although without comparing the results to graph algebras.
In fact, it is known that groups with rank larger than one do not admit a description as a graph algebra, see the remarks at the end of the introduction of \cite{HonSzy:spheres}.

To overcome this difficulty, one can look at the more general concept of higher-rank graph algebras, due to Kumjian and Pask \cite{KumPas}.
Roughly speaking, a higher-rank graph is a graph whose edges are classed into $\NC$ colours, and which is equipped with an equivalence relation on paths such that, for instance, a path of two edges of colours red-blue can always be replaced by an equivalent path of colours blue-red with the same start and end points.    In this way, paths $e$ in the graph are equipped with a \emph{degree}, or \emph{coloured length}, $\dmap(e)\in\NN^\NC$.  For the precise definition, see \cref{sec:k-graph_algebra}.

In a recent advance, Olof Giselsson \cite{Giselsson:SU3} has shown that $C(SU_q(3))$ is isomorphic to the $C^*$-algebra of a $2$-graph, which is reproduced in \cref{fig:graph-SU3}.

\begin{figure}[h]
\centering
    
\begin{tikzpicture}[
vertex/.style = {align=center, inner sep=2pt},
Rarr/.style = {->, red},
Barr/.style = {->, blue, dotted},
shadow/.style = {white, line width=3pt},
Rloop/.style = {->, red, out=165, in=195, loop},
Bloop/.style = {->, blue, out=15, in=-15, loop, dotted}
]
\node (v1) at ( 0, 0) [vertex] {$\bullet$};
\node (v2) at (-2,-1) [vertex] {$\bullet$};
\node (v3) at ( 2,-1) [vertex] {$\bullet$};
\node (v4) at (-2,-2) [vertex] {$\bullet$};
\node (v5) at ( 2,-2) [vertex] {$\bullet$};
\node (v6) at ( 0,-3) [vertex] {$\bullet$};

\draw [Rloop] (v1) edge (v1);
\draw [Bloop] (v1) edge (v1);
\draw [Barr] (v2)--(v1);
\draw [Rarr] (v3)--(v1);

\draw [Rloop] (v2) edge (v2);
\draw [Bloop] (v2) edge (v2);
\draw [Rarr, transform canvas={xshift=-0.5em}] (v4)--(v2);
\draw [Barr] (v4)--(v2);
\draw [Rarr] (v5)--(v2);
\draw [Rarr] (v6)--(v2);

\draw [Rloop] (v3) edge (v3);
\draw [Bloop] (v3) edge (v3);
\draw [shadow] (v4)--(v3);
\draw [Barr] (v4)--(v3);
\draw [shadow] (v6)--(v3);
\draw [Barr] (v6)--(v3);
\draw [Rarr] (v5)--(v3);
\draw [Barr, transform canvas={xshift=0.5em}] (v5)--(v3);

\draw [Rloop] (v4) edge (v4);
\draw [Bloop] (v4) edge (v4);
\draw [Rarr] (v6)--(v4);
\draw [Barr] (v6)--(v5);

\draw [Rloop] (v5) edge (v5);
\draw [Bloop] (v5) edge (v5);

\draw [Rloop] (v6) edge (v6);
\draw [Bloop] (v6) edge (v6);
\end{tikzpicture}
    
\caption{Giselsson's $2$-graph for $SU(3)$. Note that this is not just the Bruhat graph for $SU(3)$---there are additional edges.}
\label{fig:graph-SU3}
\end{figure}
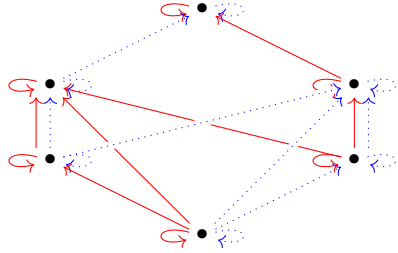

The main result of this article is that the coordinate ring of any quantized compact semisimple Lie group admits a $q = 0$ limit which is a higher-rank graph algebra.
The construction of the higher-rank graph, which is summarized below and makes use of the theory of crystal bases, will make clear the link between the algebraic and analytic approaches to the crystal limit.

The generators of the $C^*$-algebra of our higher-rank graphs can be represented by tensor products of shift operators. 
They have a much simpler algebra structure than the algebras $C(K_q)$, even at $q=1$, since giving an explicit formula for a product of matrix coefficients requires the use of Clebsch-Gordan formulas, whereas the product rules for graph algebras are purely combinatorial.  In this way, our results can be seen as a non-commutative geometric counterpart of the simplifications afforded by crystal basis theory in the algebraic context.

\subsection{Statement of main results}

The quantum groups $\OGq$ with $q$ specialized in $\qrange$ can be represented faithfully on a Hilbert space using the work of Soibelman \cite{Soibelman}.  By composing this with the specialization map\footnote{The specialization map is only partially defined, since not all elements of $\OqG$ can be specialized at any $q\in\qrange$.  We will ignore this detail in the introduction.} $\OqG \to \OGq$, we obtain for every value $q\in\qrange$ a representation $\SoibT_q$ of $\OqG$ on the Hilbert space $\SoibH=\ell^2(\NN)^{\otimes l} \otimes L^2(T)$, where $l$ is the length of the longest word in the Weyl group of $K$ and $T$ is the maximal torus of $K$.

Using the theory of crystal bases, we can easily define a subring $\OqAOG$ of the $\CC(q)$-algebra $\OqG$, consisting of elements that can be "specialized at $q = 0$" (see for instance \cite[\S6]{Iglesias:bitableaux} for this construction, where it is denoted $\widetilde{\cF}$).
We begin by proving the following analytic result, which is probably known to experts.

\begin{theorem}
For any $u \in \OqAOG$, the one-parameter family of operators $\SoibT_q(u)$ admits a well-defined norm limit $\SoibT_0(u)$ as $q \to 0$.
\end{theorem}

However, an important subtlety arises when we introduce the $*$-struc\-ture.  One can define a $*$-structure on the $\CC(q)$ algebra $\OqG$ which specializes to the usual $*$-structure on $\OGq$ when $q \in \qrange$, and which defines the real form $\OKq$.  However, the $\pid$-subalgebra $\OqAOG$ is not stable under the $*$, see \cref{ex:SU2-star}, so the algebra $\OqAOG/q\OqAOG$ at $q=0$ doesn't inherit a $*$-structure.

To resolve this issue, we define a new $\pid$-algebra $\OqAOK \subset \OqG$ which is stable under the $*$.  Essentially, $\OqAOK$ is just the $\pid$-algebra generated by $\OqAOG$ and $\OqAOG^*$, although for technical reasons we define it slightly differently, see \cref{def:OqAOK}.  
The limit $\SoibT_0$ of the Soibelman representations extends to $\OqAOK$, and this
allows us to define a $*$-subalgebra $\OKO=\pi_0(\OqAOK)$ of $\cB(\SoibH)$.
Our main task is to investigate the structure of this $*$-algebra.  

For this, a major role is played by crystal basis theory, as follows.  Given a dominant integral weight $\lambda \in \dominant$, we write $V(\lambda)$ for the simple $\Uqg$-module of highest weight $\lambda$, and $\cB(\lambda)$ for the associated crystal.  Recall that if $\mu, \mu' \in \dominant$ are dominant integral weights with $\mu + \mu'= \lambda$, then there is a unique non-trivial morphism of crystals $\inclusion_{\mu, \mu'}: \cB(\lambda) \to \cB(\mu) \otimes \cB(\mu')$.   The image of this inclusion is called the \emph{Cartan component} of the tensor product $\cB(\mu) \otimes \cB(\mu')$, namely the unique irreducible component of highest weight $\mu + \mu'$.

\begin{definition}
Let $\lambda, \mu \in \dominant$ with $\mu \leq \lambda$. The \emph{$\mu$-right end} of a crystal element $b \in \cB(\lambda)$ is the element $\rightend(b) = b'' \in \cB(\mu)$ determined by $\inclusion_{\lambda-\mu,\mu}: b \mapsto b'\otimes b''$.  

If $\fundweights = (\varpi_1, \cdots, \varpi_\rank)$ is the family of fundamental weights of $K$ and $\lambda\geq\rho =\sum_i\varpi_i$, then we write
\[
 \rightend_\fundweights(b) = \big(\rightend_{\varpi_1}(b), \cdots, \rightend_{\varpi_\rank}(b) \big)
\]
for the family of \emph{fundamental right ends} of $b \in \cB(\lambda)$.
\end{definition}

We can now define the higher-rank graph associated to the crystal limit of $\OKq$.  We identify the set of dominant integral weights $\dominant$ with the monoid $\NN^\rank$ by identifying $(n_i) \in \NN^\rank$ with $\sum_i n_i \varpi_i$.  

\begin{theorem}
\label{prop:intro-hgraph}
We can define an $\rank$-graph $\hgraphstd$ as follows.  The vertex set is
\[
 \Vsetstd = \{ \rightend_\fundweights(b) \mid b \in \cB(\rho) \}
\]
namely the fundamental right ends of elements of $\cB(\rho)$.
The paths are given by pairs $(v, b) \in \Vsetstd \times \cB(\lambda)$ where $v = \rightend_\fundweights(c)$ is such that $c \otimes b$ is in the Cartan component of $\cB(\rho) \otimes \cB(\lambda)$. The range and source maps are given by
\begin{align*}
    & \smap(v,b) = v,
    && \rmap(v,b) = \rightend_\fundweights(c\otimes b).
\end{align*}
\end{theorem}

\begin{remark}
It is part of the proposition that the above definitions depend only on the vertex $v\in\Vsetstd$ and not on the choice of crystal element $c\in\cB(\rho)$ which represents it.
\end{remark}

Our main theorem is the following.  We write $\KPalgstd$ for the Kumjian-Pask algebra of $\hgraphstd$ in the sense of \cite{Aranda:Kumjian-Pask}, and $C^*(\hgraphstd)$ for the higher rank graph $C^*$-algebra in the sense of \cite{KumPas}.

\begin{theorem}
  Let $K$ be a compact, connected, simply connected semisimple Lie group, with complexified Lie algebra $\lie{g}=\lie{k}_\CC$, and let $\hgraphstd$ be the higher rank graph from \cref{prop:intro-hgraph}.  There is an isomorphism of $*$-algebras $\OKO\cong\KPalgstd$.  Its $C^*$-closure is $C(K_0)\cong C^*(\hgraphstd)$.
\end{theorem}

In fact, we obtain a much more general result, see \cref{thm:Soibelman_faithful}.  Firstly, we may allow $K$ to be non-simply connected.  Secondly, we consider the quantized coordinate ring of the canonical torus bundle $Y_\rootset \to X_\rootset$ over any generalized flag variety for $K$ associated to a set $\rootset\subseteq\simpleroots$ of simple roots, see \cref{sec:flag_manifolds} for the notation.  The coordinate ring $\cO[Y_\rootset]$ is generated by matrix coefficients for simple $\lie{g}$-modules with highest weights in a submonoid $\dominant_{K,\rootset} \subseteq \dominant$.  We write  $\Cset = (\vartheta_1, \cdots, \vartheta_\NC)$ for the set of dominant weights which generate $\bP_{K,\rootset}$.  The quantized coordinate rings of $Y_\rootset$ and $X_\rootset$ admit analytic limits at $q=0$.

We prove that by replacing $\fundweights$ by $\Cset$ in the definition of $\hgraphstd$, we obtain another higher-rank graph $\hgraph$ and show that we have an isomorphism $\cO[Y_{\rootset,0}] \cong \KPalg$.  The gauge-invariant subalgebra is then $\cO[X_{\rootset,0}]\cong \KPalg_0$.  These isomorphisms extend naturally to the $C^*$-algebras.

The results can be summarized in the following way.

\begin{theorem}
\label{thm:continuous_field}
  Let $K$ be a compact connected semisimple Lie group, and let $Y_\rootset$ be the canonical torus bundle over the generalized flag variety $X_\rootset$ associated to a set $\rootset\subset\simpleroots$ of simple roots.
  There is a continuous field of $C^*$-algebras $C(Y_{\rootset,\bullet})$ over $[0,\infty]$ whose fibres are
  \begin{equation*}
      C(Y_{\rootset,q}) \cong 
      \begin{cases}
       C(Y_{\rootset}), & q=1,\\
       C(Y_{\rootset,q}), & q\in\qrange\\
       C^*(\hgraph), &q\in\{0,\infty\}.
      \end{cases}
  \end{equation*}
  where $\hgraph$ is the  higher rank graph described above (and spelled out in \cref{thm:hgraph}).
\end{theorem}

\begin{remark}
The $C^*$-algebras $C(K_q)$ are known to be abstractly isomorphic for all $q\in\qrange$, see \cite{Giselsson}.  It is not yet known if they are also isomorphic to the graph algebra $C(K_0)$, other than for the cases $K=\SU_q(2)$ and $\SU_q(3)$, the latter due again to Giselsson \cite{Giselsson}, and a few examples in the case of homogeneous spaces, due to Hong and Szymanski \cite{HonSzy:spheres}.
\end{remark}

\subsection{Structure of the paper}

In \cref{sec:background} we review some background material on quantized enveloping algebras and crystal bases.
In \cref{sec:coordinate_ring_def} we review some standard material on quantized coordinate rings, as well as defining the $\pid$-form $\OqAOK$ of $\OqK$.
In \cref{sec:coordinate_rings} we show that the elements of $\OqAOK$ admit analytic limits in terms of the Soibelman representation. This allows us to define the $*$-algebra $\OKO$, which is our main object of study.
In \cref{sec:Cartan-braiding} we introduce the notion of Cartan braiding, which is the crystal limit of a certain rescaled version of the braiding in the category of finite-dimensional $\Uqg$-modules.
This allows us to obtain various relations holding in $\OKO$.
In \cref{sec:properties-Cartan} we study the Cartan braiding in more detail: we show that it satisfies the hexagon and braid relations, and gives a partial action of the symmetric group on tensor products.
We also introduce the notion of right end of a crystal.
In \cref{sec:k-graph_algebra} we discuss higher-rank graph and their corresponding algebras. We prove, using crystal bases, that one can associate a higher-rank graph $\hgraphstd$ of rank $r$ to any complex semisimple Lie algebra $\lie{g}$ of rank $r$ (in fact, we prove this in a more general setting).
In \cref{sec:crystal-algebra} we introduce the crystal algebra $\CalgKstd$ as a useful tool to study the relation between $\OKO$ and $\KPalgstd$, the $*$-algebra associated to the higher-rank graph $\hgraphstd$.
The main result here is that we have surjective $*$-homomorphisms $\KPalgstd \to \CalgKstd$ and  $\CalgKstd \to \OKO$.
In \cref{sec:crystal-limit} we show that the three $*$-algebras
$\KPalgstd$, $\CalgKstd$ and $\OKO$ are all $*$-isomorphic, which gives our main result about the structure of the crystal limit $\OKO$.
Finally in \cref{sec:further-properties} we discuss some further properties of the higher-rank graphs $\hgraphstd$, mainly related to the role of the Weyl groups, as well as discussing some explicit examples.

\subsection*{Acknowledgements}

We would like to thank Olof Giselsson for sharing with us the example of $\SU_q(3)$ as a $2$-graph algebra, as well as comparing his own approach with ours.
We would also like to thank Sergey Neshveyev for discussions concerning continuous fields of $C^*$-algebras in the context of compact quantum groups.

\section{Background material}
\label{sec:background}

\subsection{Ground fields and specialization}

As mentioned before, we will consider the quantized enveloping algebra $\Uqg$ and quantized coordinate ring $\OqG$ in two settings: 1) over the field of rational functions $\CC(q)$ with $q$ an indeterminate; 2) over $\CC$ with $q \in \qrange$.

We write $\CC[q]$ for the polynomial ring in $q$, so that $\CC(q)$ is its field of fractions. We consider the three subrings
\begin{align}
\nonumber
\pid &:= \{ g / h : g, h \in \CC[q], \ h(0) \neq 0 \}, \\
\label{eq:base_rings}
\pidinfty &:= \{ g / h : g, h \in \CC[q], \ h(\infty) \neq 0 \},\\
\nonumber 
\laurent &:= \CC[q,q^{-1}].
\end{align}

For every fixed $q\in[0,\infty]$, there is a partially defined \emph{specialization map}
\(
  \ev_{q} : \CC(q) \to \CC
\)
given by evaluation at $q$.  We apologize for the reuse of $q$ as both formal parameter and specialized value $q\in[0,\infty]$.  The rings $\pid$, $\pidinfty$, $\laurent$ contain the elements which can be evaluated at $0$, at $\infty$ and at all $q\in\CC^\times$, respectively.

\subsection{Quantized enveloping algebras}
\label{sec:Uqg}

Let $\lie{g}$ be a complex semisimple Lie algebra and fix a Cartan subalgebra $\lie{h}$. Write $\lie{k}$ for the compact real form of $\lie{g}$ and put $\lie{t} = \lie{k} \cap \lie{h}$. The connected, simply connected Lie groups associated to $\lie{g}$ and $\lie{k}$ are denoted by $G$ and $K$.  The maximal torus of $K$ is $T \cong \TT^\rank$, where $\rank$ denotes the rank.

We write $\roots$ for the set of roots of $\lie{g}$ and 
$\simpleroots = \{\alpha_1, \cdots, \alpha_\rank\}$ for a choice of simple roots.
The root lattice will be denoted by $\bQ$ and the lattice of integral weights by $\bP$, generated by the fundamental weights $\fundweights = \{\varpi_1, \cdots, \varpi_\rank\}$.  The abelian semigroup $\dominant = \NN \cdot \fundweights \cong \NN^\rank$ of dominant weights will play an important role in the higher-rank graph to be defined later.
We put $\rho=\sum_{i=1}^\rank \varpi_i$ as usual.
We write $(\slot,\slot)$ for the invariant bilinear form on $\lie{h}^*$ such that the short roots $\alpha$ satisfy $(\alpha,\alpha)=2$.

The quantized enveloping algebra $\Uqg$ is the Hopf algebra with generators $E_i$, $F_i$ and $K_i^{\pm1}$ for $i = 1, \cdots, \rank$, and with the standard algebra relations that can be found for instance in \cite[\S6.1.2]{KliSch}, or \cite[\S{}4.3]{Jantzen} %
with $K_i=K_{\alpha_i}$.
Again, we may consider this either over the field $\CC(q)$ or over $\CC$ with $q\in(0,\infty)\setminus\{1\}$ a fixed parameter.
There are different possible choices for the coproduct.
The one we consider is more or less standard in the context of crystal bases\footnote{For the operator algebraist, it is equivalent to coproduct from \cite{NesTus:book}, but the opposite of the coproduct from \cite{KliSch, VoiYun:CQG}.}
and is given by
\begin{equation}
\label{eq:coproduct}
\Delta(K_i) = K_i \otimes K_i, \quad
\Delta(E_i) = E_i \otimes K_i^{-1} + 1 \otimes E_i, \quad
\Delta(F_i) = F_i \otimes 1 + K_i \otimes F_i.
\end{equation}
We write $\Uqh$ for the Hopf subalgebra generated by the $K_i^{\pm1}$.

We also consider a $*$-structure on $\Uqg$ which corresponds to the compact real form $\lie{k}$ of $\lie{g}$.
It acts on the generators by
\begin{align}
\label{eq:star-structure}
K_i^* = K_i, &&
E_i^* = q_i F_i K_i^{-1}, &&
F_i^* = q_i^{-1} K_i E_i,
\end{align}
and $q^* = q$ when we work over $\CC(q)$.  Here $q_i = q^{(\alpha_i, \alpha_i) / 2}$.  
We write $\Uqk$ for $\Uqg$ equipped with this $*$-structure.
We note that, since we are putting $q^* = q$, the specialization map at $q\in\CC^\times$ is a $*$-morphism only when $q$ is real.

\begin{remark}
We warn the reader that this is not the most common choice for the $*$-structure in the operator algebraic picture. It is a convenient choice when working with crystal bases because the formulas \eqref{eq:star-structure} coincide with those of the standard anti-automorphism of $\Uqg$ called $\tau_1$ in \cite[\S9.20 (3)]{Jantzen}, except that our involution $*$ is extended $\CC$-antilinearly.
\end{remark}

We write $V(\lambda)$ for the irreducible finite-dimensional integrable (\emph{i.e.}, type 1) $\Uqg$-module of highest weight $\lambda \in \dominant$.   Again, this can be constructed either in the algebraic setting as a module over $\CC(q)$, or in the analytic setting as a $\CC$-vector space.  If necessary, we will denote the latter by $V(\lambda)_q$ with $q\in\qrange$  a fixed parameter, but generally we will suppress the subscript and interpret the $\Uqg$-module $V(\lambda)$ according to context.  Once again, there is a partially defined specialization map $\ev_q:V(\lambda) \to V(\lambda)_q$ for each $q\in\qrange$.

There is a unique $\CC(q)$-valued inner product on $V(\lambda)$ such that the $\Uqg$-action is a $*$-representation and the highest weight vector $v_\lambda$ has norm $1$, see \cite[Lemma 9.20 c)]{Jantzen}.  For convenience, we will refer to this as the \emph{standard inner product} on $V(\lambda)$.

\subsection{Crystal bases}

In this section we briefly recall various facts associated with Kashiwara's theory of crystal bases \cite{Kashiwara:crystal1, Kashiwara:crystal2}.
Our main references are the textbooks \cite{Hong-Kang} and \cite{Jantzen}.

We need the Kashiwara operators $\kasE_i$ and $\kasF_i$, which are algebraic renormalizations of $E_i$ and $F_i$ with well-defined limits at $q = 0$.  
They play a similar role to the operator phases of $E_i$ and $F_i$ in the $C^*$-algebraic world.
To define them, we first introduce the divided powers
\begin{align*}
  E_i^{(k)} & := E_i^k/[k]_{q_i}!, &
  F_i^{(k)} & := F_i^k/[k]_{q_i}!,
\end{align*}
where $[k]_q := \frac{q^k - q^{-k}}{q - q^{-1}}$ and $[k]_q! := [k]_q [k-1]_q \cdots [1]_q$. 
The Kashiwara operators are defined by imposing that, for any weight vector $u_0$ in an integrable $\Uqg$-module $V$ with $E_iu_0=0$, we have
\begin{align*}
  \kasE_i :F_i^{(k)}u_0 &\mapsto F_i^{(k-1)}u_0, &
  \kasF_i :F_i^{(k)}u_0 &\mapsto F_i^{(k+1)}u_0,
\end{align*}
with the convention $F_i^{(-1)} u_0 = 0$.

We can now define crystal bases.

\begin{definition}
A \emph{crystal lattice} in $V$ is an $\pid$-submodule $\cL$ of $V$ such that:
\begin{enumerate}
\item $\cL$ is finitely generated over $\pid$ and generates $V$ as a vector space over $\CC(q)$,
\item $\cL = \bigoplus_{\mu \in \bP} \cL_\mu$, where $\cL_\mu = \cL \cap V_\mu$,
\item $\kasE_i \cL \subset \cL$ and $\kasF_i \cL \subset \cL$ for all $i$.
\end{enumerate}
\end{definition}

\begin{definition}
A \emph{crystal basis} of $V$ is a pair $(\cL, \cB)$, where $\cL$ is a crystal lattice and $\cB$ is a $\CC$-basis of the quotient $\cL / q \cL$ such that:
\begin{enumerate}
\item $\cB = \bigcup_\mu \cB_\mu$, where $\cB_\mu = \cB \cap (\cL_\mu / q \cL_\mu)$,
\item $\tilde{E}_i \cB \subset \cB \sqcup \{0\}$ and $\tilde{F}_i \cB \subset \cB \sqcup \{0\}$ for all $i$,
\item for any $b, b^\prime \in \cB$ and $i$ we have $b = \tilde{E}_i b^\prime \Longleftrightarrow b^\prime = \tilde{F}_i b$.
\end{enumerate}
\end{definition}

In this way, we replace the action of $\Uqg$ on a finite-dimensional module $V$ by a purely combinatorial action of the Kashiwara operators $\kasE_i,\kasF_i$ on the finite set $\cB\sqcup\{0\}$.
Kashiwara proved that every finite-dimen\-sion\-al integrable $\Uqg$-module admits a crystal basis, which is unique up to isomorphism.  We denote the crystal basis of the simple module $V(\lambda)$ by $(\cL(\lambda), \cB(\lambda))$.

The crystal basis is orthonormal in the following sense, compare \cite[Lemma 5.1.6]{Hong-Kang}, \cite[Theorem 9.25]{Jantzen}.

\begin{lemma}
\label{lem:orthogonality}
  Fix $\lambda \in \dominant$. Let $\{v_1, \cdots, v_d\} \subset \cL(\lambda)$ be a family of lifts for the crystal basis $\cB(\lambda) = \{b_1, \cdots, b_d\}$ of $V(\lambda)$. Let $\ip{\slot,\slot}$ be the standard inner product on $V(\lambda)$.  Then $\ip{v_a,v_b} \in \delta_{ab} + q \pid$.
  
  As a consequence, for any $u,v\in\cL(\lambda)$ we have $\ip{u,v} \in \pid$.
\end{lemma}

From an analytic point of view, \cref{lem:orthogonality} means that for any $u, v \in \cL(\lambda)$, the limit of the specialized inner products
\[
  \lim_{q\to 0} \ev_q\ip{u,v}
\]
exists.  Note that $\ev_q\ip{u,v}$ does not necessarily exist for all $q>0$, but since $\ip{u,v}\in\pid$, it at least exists in some neighbourhood of $0$.
We will use similar observations frequently in this work, and as above will use the notation $\lim_{q\to0}$ when strictly speaking we mean $\lim_{q \to 0^+}$.

\subsection{Tensor products}

A major feature of crystal bases is the simplicity of their tensor products, and this will play a major role in what follows.

Let $V_1$ and $V_2$ be finite-dimensional integrable $\Uqg$-modules with crystal bases $(\cL_1, \cB_1)$ and $(\cL_2, \cB_2)$ respectively.
Then $(\cL_1 \otimes_{\pid} \cL_2, \cB_1 \times \cB_2)$ is a crystal basis of $V_1 \otimes V_2$, see \cite[Theorem 4.4.1]{Hong-Kang}.
We will be lazy and write $\cL_1\otimes\cL_2$ to mean the crystal lattice $\cL_1\otimes_\pid\cL_2$.

To describe the action of the Kashiwara operators we first define
\[
\varepsilon_i(b) := \max \left\{ k \in \mathbb{N}_0 : \kasE_i^k b \neq 0 \right\}, \quad
\varphi_i(b) := \max \left\{ k \in \mathbb{N}_0 : \kasF_i^k b \neq 0 \right\}.
\]
Then the action on the tensor product is given by
\begin{align}
\label{eq:tensor_E}
\kasE_i (b \otimes b') &= \begin{cases}
\kasE_i b \otimes b' & \varphi_i(b) \geq \varepsilon_i(b') \\
b \otimes \kasE_i b' & \varphi_i(b) < \varepsilon_i(b')
\end{cases}, 
\\
\label{eq:tensor_F}
\kasF_i (b \otimes b') &= \begin{cases}
\kasF_i b \otimes b' & \varphi_i(b) > \varepsilon_i(b') \\
b \otimes \kasF_i b' & \varphi_i(b) \leq \varepsilon_i(b')
\end{cases}.
\end{align}
Here, as is customary, we denote the element $(b_1, b_2) \in \cB_1 \times \cB_2$ by $b_1 \otimes b_2$, and we use the conventions $b_1 \otimes 0 = 0 \otimes b_2 = 0$. 


\section{The quantized coordinate ring}
\label{sec:coordinate_ring_def}

\subsection{Definitions}

Let $V$ be a finite-dimensional integrable $\Uqg$-module over $\CC(q)$ and let $V^*=\Hom_{\CC(q)}(V,\CC(q))$ be the dual module.
A \emph{matrix coefficient} for the module $V$ is a $\CC(q)$-linear map $c^V_{f, v}$ on $\Uqg$ of the form
\begin{equation}
\label{eq:matrix_coefficient}
 c^V_{f, v}(X) := f(X v), \quad v \in V,\ f \in V^*.
\end{equation}
The \emph{quantized coordinate ring} $\OqG$ is the subspace of $\Hom_{\CC(q)}(\Uqg,\CC(q))$ spanned by the matrix coefficients of finite-dimensional integrable modules.
This is a Hopf algebra with operations obtained from those of $\Uqg$ by duality.  Here, we use the algebraists' convention of a non-skew pairing:
\begin{align*}
    &(\Delta(X), a\otimes b) = (X,ab),
    &(X\otimes Y, \Delta(a)) = (XY, a).
\end{align*}
Explicitly, this corresponds to the matrix coproduct on matrix coefficients:
\[
 \Delta(c^V_{f,v}) = \sum_i c^V_{f,v_i} \otimes c^V_{f^i,v},
\]
where $\{v_i\}$ and $\{f_i\}$ are any basis and dual basis for $V$ and $V^*$, respectively.
We denote the left and right regular representations of $\Uqg$ on $\OqG$ by
\begin{align}
    \label{eq:regular_repns}
    X\hit c^V_{f,v} &= c^V_{f,Xv}, &
    c^V_{f,v} \hitby X &= c^V_{f\circ X,v}
\end{align}

We also have a version of the quantized coordinate ring defined over $\CC$ for a fixed complex parameter $q \in \qrange$. This algebra will be denoted $\OGq$, with apologies for the subtle difference in notation.  For every $q \in \qrange$ we have a partially defined specialization map
\[
 \ev_q : \OqG \to \OGq,
\]
induced by evaluation at $q$.

In \cite{Kashiwara:global} Kashiwara introduced a crystal lattice for $\OqG$ consisting of finite sums of matrix coefficients $c^V_{f,v}$ where $v$ belongs to a crystal lattice $\cL$ for $V$ and $f$ belongs to $\cL^* = \Hom_\pid(\cL,\pid) \subset V^*$.
We denote this $\pid$-form by $\OqAOG$.
This algebra was also studied by Iglesias in \cite{Iglesias:bitableaux}, where he uses the notation $\widetilde{\cF}$.

\subsection{Star structure and compact integral form}
\label{sec:compact_form}

The $*$-structure \eqref{eq:star-structure} on $\Uqg$ induces a corresponding $*$-structure on $\OqG$ and on each $\OGq$ by the formula $(X,a^*) = \overline{(S(X)^*,a)}$.
When equipped with this $*$-structure, we denote the $*$-algebras $\OqG$ and $\OGq$ by $\OqK$ and $\OKq$, respectively.   Again, we caution that $q^* = q$, and so the specialization map \linebreak $\OqK\to\OKq$ is a $*$-morphism only when $q$ is real.  

As mentioned in the introduction, the $\pid$-form $\OqAOG$ is not stable under the $*$.  This can be seen in the following example.

\begin{example}
\label{ex:SU2-star}
Let $V$ be the fundamental representation of $\cU_q(\lie{sl}_2)$ over $\CC(q)$. Let $v_1$ be the highest weight vector for $V$ and $v_2 = F v_1$.  Let $\{f^1,f^2\}\in V^*$ be the dual basis for $V^*$.  The subspace of $\OqAOG$ corresponding to the fundamental representation is spanned by the matrix coefficients $u^V_{ij} = c^V_{f^i,v_j}$ with $i,j \in \{1,2\}$.

In this particular case, $\{v_1,v_2\}$ is an orthonormal basis for $V$, so upon specialization at $q\in\qrange$, the matrix coefficients $u^V_{ij}$ coincide with the traditional generators for the operator algebraists' $*$-algebra $\OKq$, which in Woronowicz's notation \cite{Woronowicz:SUq2} are denoted
\begin{align*}
    \alpha &= u^V_{11} , &
    -q\gamma^* &= u^V_{12} , \\
    \gamma &= u^V_{21} , &
    \alpha^* &= u^V_{22} .
\end{align*}
We see that $u^{V*}_{21} = -q^{-1}u^V_{12}$, which does not belong to $\OqAOG$.
\end{example}

Because of this, 
we cannot simply equip the algebra $\OqAOG$ with the $*$ operation.  Instead, we proceed as follows.

For each simple module $V(\lambda)$ fix a lift of the crystal basis to a basis of weight vectors $\{v_i\}_i \subset \cL(\lambda)$, with dual basis $\{f^i\}_i \subset V(\lambda)^*$.
We impose that $v_1 = v_\lambda$ is the highest weight vector of $V(\lambda)$, and so $f^1 = f^{-\lambda}$ is the lowest weight element.

With this notation, we define the matrix coefficients
\begin{equation}
\label{eq:genf_genv}
\genf^\lambda_i := c^{V(\lambda)}_{f^i, v_\lambda}, \quad
\genv^\lambda_i := S(c^{V(\lambda)}_{f^{-\lambda}, v_i}).
\end{equation}
We also note that $\genv^\lambda_i = c^{V(\lambda)^*}_{\tilde{v}_i, f^{-\lambda}}$, where $\{\tilde{v_i}\}_i \subset V(\lambda)^{**}$ is the dual basis to $\{f^i\}_i$.

Following Joseph \cite[\S9.1.6]{Joseph:book}, let $\OqGP$ and $\OqGM$ be the $\CC(q)$-subalgebras of $\OqG$ generated by the matrix coefficients $\genf^\lambda_i$ and $\genv^\lambda_i$, respectively (for all $\lambda\in\dominant$ and $i=1, \cdots, \dim V(\lambda)$).  Joseph uses the notation $R_q[G/N^+]$ and $R_q[G/N^-]$.
The multiplication map $\OqGP \otimes_{\CC(q)} \OqGM \to \OqG$ is surjective, see \cite[Proposition 9.2.2]{Joseph:book}. 

We now define $\pid$-forms of these subalgebras.

\begin{definition}
\label{def:OqAOK}
We define:
\begin{itemize}
    \item $\OqGAOP$ as the $\pid$-subalgebra generated by the elements $\genf^\lambda_i$,
    \item $\OqGAOM$ as the $\pid$-subalgebra generated by the elements $\genv^\lambda_i$,
    \item $\OqKAO$ as the $\pid$-subalgebra generated by the elements $\genf^\lambda_i$ and $\genv^\lambda_i$.
\end{itemize}
\end{definition}

Thus, while $\OqK$ and $\OqG$ denote the same $\CC(q)$-algebra with or without the $*$-structure, we stress that $\OqAOK$ and $\OqAOG$ are not the same $\pid$-subalgebra.  They represent different $q=0$ limits of the families of algebras $\OKq=\OGq$.
The notation $\OqKAO$ suggests that this algebra is closed under the $*$-structure, which is not immediately obvious from its definition.
However this is true, as shown in the next result.

\begin{proposition}
\label{prop:v_f_adjoints}
We have $(\OqGAOP)^* = \OqGAOM$. Moreover
\[
\begin{split}
(\genf^\lambda_i)^* & \equiv \genv^\lambda_i \mod q \OqGAOM, \\
(\genv^\lambda_i)^* & \equiv \genf^\lambda_i \mod q \OqGAOP.
\end{split}
\]
\end{proposition}

\begin{proof}
Let $\{v_i\}_i$ be the lift of the crystal basis of $V(\lambda)$ as above.
Let $(\cdot, \cdot)$ be the unique inner product on $V(\lambda)$ invariant under $*$ such that $(v_1, v_1) = 1$.
Let $G^i_j = (v_i, v_j)$ be the Gram matrix for the basis $\{v_i\}_i$.
Some linear algebra shows that we have the identity
\[
(c^{V(\lambda)}_{f^i, v_j})^* = \sum_{k, l} (G^{-1})^l_i G^j_k S(c^{V(\lambda)}_{f^k, v_l}).
\]
Now consider the case of $\genf^\lambda_i = c^{V(\lambda)}_{f^i, v_1}$.
We have $G^1_k = (v_1, v_k) = \delta_{1 k}$, since different weight spaces are orthogonal under the inner product. Then we get
\[
(\genf^\lambda_i)^* = \sum_l (G^{-1})^l_i \genv^\lambda_l.
\]
Recall from \cref{lem:orthogonality} that the basis $\{v_i\}_i$ becomes orthonormal at $q = 0$, which implies that $G^i_j$ and $(G^{-1})^i_j$ are equal to $\delta_{i j}$ modulo terms in $q \pid$.
Thus we obtain the desired result for $(\genf^\lambda_i)^*$.  The result for $(\genv^\lambda_i)^*$ is proven similarly, and it follows that $(\OqGAOP)^* = \OqGAOM$.
\end{proof}

\begin{remark}
It is not clear from the definition above whether $\OqAOG \subset \OqAOK$, although we expect this to be the case.  Answering this question would require a more careful study of the crystal lattice structure.  Since we don't need this property here, we won't discuss it further.
\end{remark}

\subsection{The gauge action}
\label{sec:gauge_action}

For the moment, we continue to suppose that $K$ is simply connected, with rank $\rank$. That is, we have $T \cong \TT^\rank$ for the maximal torus and $\bP \cong \hat{T} \cong \ZZ^\rank$ for the lattice of integral weights.
The quantized coordinate ring $\OqK$ is $\ZZ^\rank$-graded as follows.  

\begin{definition}
\label{def:Zr-grading}
\label{def:gauge_action}
For any $\mu \in \bP$, we write $\OqK_{\mu}$ for the span of all matrix coefficients $c^V_{f, v}$ where $v \in V$ is a vector of weight $\mu$, so that $\OqK = \bigoplus_{\mu \in \bP} \OqK_\mu$ is a $\bP$-graded algebra.  
\end{definition}

The $\bP$-grading corresponds to the gauge action of the torus subgroup $T$,
\begin{equation}
  z\cdot c^V_{f,v} = z^\mu c^V_{f,v} \quad \text{for } c^V_{f,v}\in\OqK_\mu, 
\end{equation}
where $z^\mu = e^{(\mu,\log(z))}$ denotes the evaluation at $z\in T$ of the unitary character associated to $\mu\in\bP$.  Geometrically, this corresponds to the realization of $K$ as a principal $T$-bundle over the full flag variety $X=K/T$, and this language is extended by analogy to the quantum case.  The gauge-invariant subalgebra of $\OqK$ is then the quantized coordinate ring of the flag variety.

\subsection{Flag manifolds and non-simply connected groups}
\label{sec:flag_manifolds}

We finish this preliminary section with the structure theory of the quantized coordinate rings of general flag varieties.   This material could be skipped on a first reading.  We follow the conventions of Stokman \cite{Stokman:quantum_orbit_method}.

Let $\lie{b}_+$ denote the Borel subalgebra of $\lie{g}$ generated by $\lie{h}$ and the elements $E_1,\cdots,E_\rank$.  The standard parabolic subalgebras are indexed by subsets of simple roots $\rootset\subseteq\simpleroots$.  Specifically, we denote by $\lie{p}_\rootset \supseteq \lie{b}_+$ the standard parabolic which contains $F_i$ if and only if $\alpha_i\in\rootset$. The Levi factor $\lie{l}_\rootset$ of $\lie{p}_\rootset$ is generated by $\lie{h}$ and those $E_i,F_i$ with $\alpha_i\in\rootset$.
The associated Lie subgroups of $G$ will always be denoted by the corresponding uppercase letters.  We write $X_\rootset = G/P_\rootset$ for the corresponding flag variety.

Passing to the compact form, we have $X_\rootset = K/K_\rootset$ where $\lie{k}_\rootset = \lie{k} \cap \lie{p}_\rootset$.  The Lie subalgebra $\lie{k}_\rootset$ decomposes as 
\[
 \lie{k}_\rootset = \lie{k}^0_\rootset \oplus \lie{z}_\rootset,
\]
where $\lie{z}_\rootset$ is the centre of $\lie{k}_\rootset$ and $\lie{k}^0_\rootset$ is semisimple.  Then $Y_\rootset = K/K^0_\rootset$ is a principal torus bundle over the flag variety $X_\rootset$ with fibres $Z_\rootset \cong \TT^\NC$, where $\NC=\rank-|\rootset|$.

For instance, when $\rootset=\emptyset$ we have $K_\emptyset = T$ and $K^0_\emptyset=\{1\}$, so the torus bundle $Y_\emptyset\to X_\emptyset$ is the principal bundle $K\to K/T$ with fibres $Z_\emptyset = T\cong\TT^\rank$.  More generally, the torus group $Z_\rootset$ identifies canonically with the dual torus of the abelian group $\bP_{\rootset^c}\subseteq \bP$ generated by the fundamental weights $\varpi_i$ with $i\notin\rootset$.

All of the above can be quantized.  
We shall not spell out all the details here, but refer directly to  \cite{Stokman:quantum_orbit_method}.  Following our above notation, we write $\OXSq$ and $\OYSq$ for what Stokman \cite{Stokman:quantum_orbit_method} calls $\CC_q[U/K_\rootset]$ and $\CC_q[U/K^0_\rootset]$ in Definitions 2.3(a) and Section 4, namely
\begin{align*}
    \OYSq 
    & = \{ a\in\OKq \mid T \hit a = \epsilon(T)a \text{ for all } T\in \cU_q(\lie{k}^0_\rootset)\},\\
    \OXSq 
    & = \{ a\in\OKq \mid T \hit a = \epsilon(T)a \text{ for all } T\in \cU_q(\lie{k}_\rootset)\}.
\end{align*}

The following is due to Stokman \cite{Stokman:quantum_orbit_method}.  
We retain the notation $\genf^\lambda_i = c^{V(\lambda)}_{f, v_1^\lambda}$, $\genv^\lambda_i = c^{V(\lambda)^*}_{\tilde{v}, f^1_\lambda}$ for the generators of $\OqGP$ and $\OqGM$ from Equation \eqref{eq:genf_genv}.
We denote the specialized versions of these elements by the same symbols, for simplicity.

\begin{theorem}
\label{thm:flag_generators}
  Let $\bP_\rootset^+$ denote the set of dominant weights which are non-negative integral linear combinations of the fundamental weights $\varpi_k$ with $\alpha_k \in \simpleroots \setminus \rootset$.
  The quantized coordinate ring $\OYSq$ of the torus bundle $Y_\rootset\to X_\rootset$ is generated as an algebra by the matrix coefficients
  \[
   \left\{ \genf^\lambda_i, \genv^\lambda_i \mid \lambda\in\dominant_\rootset \right\}.
  \]
  This algebra is stable under the gauge action of $T$ defined in \cref{sec:gauge_action}, and the gauge-invariant subalgebra is $\OXSq$.
\end{theorem}

\begin{proof}
 In fact, Stokman \cite[Theorem 4.1]{Stokman:quantum_orbit_method} proves that $\OYSq$ is generated by $c^{V(\lambda)}_{f, v_1^\lambda}$ and $c^{V(\lambda)^*}_{\tilde{v}, f^1_\lambda}$ with $\lambda = \varpi_k$ for $\alpha_k \notin \rootset$, so our first claim is in fact weaker.  The second claim follows immediately upon observing that $\OXSq$ consists of those elements of $\OYSq$ which are invariant under the action of $\Uqh$.
\end{proof}

\begin{remark}
The algebra $\OYSq$ is already invariant under the action of $K_i$ for $\alpha_i\in\rootset$, so the gauge action in \cref{thm:flag_generators} is determined entirely by the action of the generators $K_i$ with $\alpha_i\notin\rootset$.  Thus the relevant gauge action on $\OYSq$ reduces to an action of $\TT^\NC$ with $\NC=\rank-|\rootset|$.
\end{remark}

Next we consider the case of a connected but not simply connected compact semisimple Lie group $K$.
Let $\lie{k}$ be the Lie algebra of $K$.  Let $\Ktilde$ be the universal cover of $K$, and $\bQ$, $\bP$ be the root and weight lattices of $\lie{k}$.  Then there is a lattice $\bP_K\subset\lie{h}^*$ with $\bQ \subseteq \bP_K \subseteq \bP$ such that the irreducible representations of $K$ are precisely those of the simply connected group $\Ktilde$ with highest weights in $\bP_K^+ = \bP_K \cap \dominant$.  As a consequence we have 
\begin{equation}
\label{eq:OqKtilde}
 \OqK = \Vect \left\{\left. c^{V(\lambda)}_{f, v} \right| \lambda\in\bP_K^+,~ f\in V(\lambda)^*, ~ v\in V(\lambda) \right\} \quad \subseteq \OqKtilde,
\end{equation}
as a $*$-subalgebra.

We can then define generalized flag varieties $X=K/K_\rootset$ and their torus bundles $Y=K/K^0_\rootset$ for the non simply connected group $K$.  Combining the above with Stokman's result, \cref{thm:flag_generators}, we get the following.

\begin{theorem}
\label{thm:flag_generators2}
  Let $K$ be a compact semisimple Lie group, not necessarily simply connected.  The quantized coordinate ring $\OYSq$ of the torus bundle $Y_\rootset\to X_\rootset$ is generated by the matrix coefficients 
  \[
   \left\{ \genf^\lambda_i, \genv^\lambda_i \mid \lambda\in\dominant_{K,\rootset} \right\}.
  \]
  where $\bP_{K,\rootset}^+ := \bP_K^+ \cap \bP_\rootset^+$.
  The gauge-invariant subalgebra of $\OYSq$ is $\OXSq$.
\end{theorem}

\section{The analytic limit of the quantized coordinate ring}
\label{sec:coordinate_rings}

\subsection{The analytic limit for \texorpdfstring{$SU_q(2)$}{SUq(2)}} 
\label{sec:SUq2_limit}

We identify the integral weights of $\lie{sl}(2)$ with the integers, $\bP=\ZZ$, so that $V(m)$ is the irreducible integral representation of dimension $m+1$.  Let%
\footnote{Unfortunately, this notation clashes with the  common convention of using $v^\lambda_1$ for the highest weight vector in $V(\lambda)$, which we have followed elsewhere in this article when $K\neq\mathrm{SU}(2)$.}
$v^m_0$ denote a highest weight vector of $V(m)$, and put $v^m_k = F^{(k)}v^m_0$.  
The $\pid$-span of these vectors is a crystal lattice $\cL(m)$ and their images $b^m_k\in\cL(m)/q\cL(m)$ define a crystal basis with crystal $\cB(m) = \{b^m_k \mid k = 0, \cdots, m\}$.

Let $\{f_m^k \mid k = 0, \cdots, m\}$ denote the basis of $V(m)^*$ dual to $\{v^m_k \mid k = 0, \cdots, m\}$.
We write $u^m_{ij} = c^{V(m)}_{f^i,v_j}$ for the associated matrix coefficients.  For the matrix coefficients of the fundamental representation we will use the operator algebraists' notation from \cite{Woronowicz:SUq2}, that is
\begin{align*}
    \alpha &= u^1_{00} , &
    -q\gamma^* &= u^1_{01} , \\
    \gamma &= u^1_{10} , &
    \alpha^* &= u^1_{11} .
\end{align*}
Note that the basis $\{v^m_k\}$ for $V(m)$ is generally only orthogonal, but in the case $m=1$ the basis $\{v^1_0,v^1_1\}$ is also orthonormal, so we do indeed recover Woronowicz's generators.
They satisfy the following relations
\[
\alpha \gamma = q \gamma \alpha, \quad
\alpha \gamma^* = q \gamma^* \alpha, \quad
\gamma \gamma^* = \gamma^* \gamma, \quad
\alpha^* \alpha + \gamma^* \gamma = 1, \quad
\alpha \alpha^* + q^2 \gamma \gamma^* = 1.
\]

\begin{definition}[Woronowicz \cite{Woronowicz:SUq2}, Vaksman-Soibelman \cite{VakSoi:SUq2}]
\label{def:SU2-rep}
  Fix $q \in \qrange$.  The \emph{standard representation} of the quantized coordinate ring $\cO[\SU_q(2)]$ is the $*$-representation $\Soib_q$ on $\ell^2(\NN) = \Vect\{e_0, e_1, \cdots\}$ determined by
  \begin{align*}
    \Soib_q(\alpha)e_n &= \sqrt{1-q^{2n}} e_{n-1}, &
    \Soib_q(\gamma)e_n &= q^ne_n, &
    &(n\geq0).
  \end{align*}
  We will use the same notation $\Soib_q$ for the partially defined representation of the algebraists' quantized coordinate ring $\Soib_q\circ\ev_q:\cO_q[\SU(2)] \to \cB(\ell^2(\NN))$.
\end{definition}

If $V$ is any finite-dimensional integrable $\Uqg$-module over the ground field $\CC(q)$ then any $v\in V$ can be specialized at all but finitely many $q\in(0,\infty)$, and likewise for any $f\in V^*$.  Therefore, 
for any given element $u\in\OqK$, the representation $\Soib_q(u)$ is well-defined for all sufficiently small $q > 0$, so it makes sense to ask about the existence of the limit of $\Soib_q(u)$ as $q\to0$.  The following result is essentially known, compare \cite{Woronowicz:SUq2, VakSoi:SUq2, HonSzy:spheres}.

\begin{theorem}
\label{thm:SL2_limit}
  For any $u\in \cO_q^{\bfA_0}[\SL(2)]$, the one-parameter family of operators $\Soib_q(u)\in\cB(\ell^2\NN)$ admits a norm limit as $q\to0$, which we denote by $\Soib_0(u)$.  
  Explicitly, if $m\in\NN$ and $i,j \in \{0, \cdots, m\}$, then 
  \begin{align*}
     \Soib_0(u^m_{ij}) = 
 \begin{cases}
  T^{j}  P_0 T^{*m-i}, & \text{if } i>j \\
  T^{i} T^{*m-i} , & \text{if } i=j \\
  0, & \text{if } i<j
 \end{cases}
  \end{align*}
  where $T$ is the right-shift operator and $P_0$ is the orthogonal projection onto $e_0$.
\end{theorem}

\begin{proof}
  For $m=0$, we have $\Soib_0(u^0_{00})=1$, and for $m=1$, the formulas in Definition \ref{def:SU2-rep} give the desired limits as $q\to0$:
  \begin{align*}
    \Soib_0(u^1_{00}) & = T^*, &
    \Soib_0(u^1_{01}) & = 0,\\
    \Soib_0(u^1_{10}) & = P_0, &
    \Soib_0(u^1_{11}) & = T. 
  \end{align*}
  Since the fundamental matrix coefficients generate $\cO_q^{\bfA_0}[\mathrm{SL}(2)]$ as an $\bfA_0$-algebra, the existence of the limit for all elements of $\cO_q^{\pid}[\mathrm{SL}(2)]$ follows immediately.  
  
  It remains to check the stated formula when $m>1$.
  There is a unique inclusion of $\Uqg$-modules $\inclusion_m:V(m) \into  V(1)^{\otimes m}$ which sends the highest weight vector $v^m_0$ to $v^1_0\otimes\cdots\otimes v^1_0$.  Equation \eqref{eq:tensor_F} shows that the Kashiwara operator $\kasF$ acts on $\cB(1)^{\otimes m}$ by
  \[
    \kasF: (b^1_1)^{\otimes k} \otimes (b^1_0)^{\otimes (m-k)} \mapsto  
    (b^1_1)^{\otimes (k+1)} \otimes (b^1_0)^{\otimes (m-k-1)},
  \]
  for $k = 0, \cdots, m - 1$, so it follows that
  \[
   \inclusion_m(v^m_k) = (v^1_1)^{\otimes k} \otimes (v^1_0)^{\otimes (m-k)} \mod q\cL(1)^{\otimes m}.
  \]
  Note that for every $i,j \in \{0,\ldots, m\}$ and every $a,b\in\NN$ we have
  \[
    \big( (f^1_1)^{\otimes i} \otimes (f^1_0)^{\otimes m-i} , \kasE^a\kasF^b (v^1_1)^{\otimes j} \otimes (v^1_0)^{\otimes (m-j)} \big)
    = \big( f^i_m , \kasE^a\kasF^b v^m_j) \mod q\pid,
  \]
  which implies that 
  \begin{align*}
  u^m_{ij} 
    & = \begin{cases}
            (u^1_{11})^{j} (u^1_{10})^{i-j} (u^1_{00})^{m-i} , & \text{if } i > j \\
           (u^1_{11})^{i} (u^1_{00})^{m-i} , & \text{if } i = j \\            
           (u^1_{11})^{i} (u^1_{01})^{j-i} (u^1_{00})^{m-j}  , & \text{if } i < j  
          \end{cases}
    \quad\mod q\cO_q^{\bfA_0}[\SL(2)].
\end{align*}
Now we can use the result for $m=1$ to deduce the general formula.
\end{proof}

This result immediately extends to our compact $\pid$-form $\cO_q^{\bfA_0}[\SU(2)]$ from \cref{def:OqAOK}, as follows.

\begin{corollary}
\label{cor:SU2_limit}
  For any $u\in \cO_q^{\bfA_0}[\SU(2)]$, the one-parameter family of operators $\Soib_q(u)\in\cB(\ell^2\NN)$ admits a norm limit as $q\to0$, which we denote by $\Soib_0(u)$.  
\end{corollary}

\begin{proof}
By \cref{thm:SL2_limit}, for every generator $\genf^\lambda_i$ of $\cO_q^\pid[\mathrm{SL}(2)/N^+]$, the operators $\Soib_q(\genv^\lambda_i)$ admit a norm-limit as $q\to0$. Since the $\Soib_q$ are $*$-re\-pre\-sent\-ations, \cref{prop:v_f_adjoints} shows that the same is true of the generators $\genv^\lambda_i$ of $\cO_q^\pid[\mathrm{SL}(2)/N^-]$.  The result follows.
\end{proof}

Note, for instance, that the matrix coefficient $\gamma^*=(u^1_{10})^* = -q^{-1}u^1_{01}$ does not belong to the $\pid$-form $\widetilde{\cF}=\cO_q^\pid[\mathrm{SL}(2)]$, but nonetheless it does admit an analytic limit at $q=0$, namely $\lim_{q\to0} \Soib_q(\gamma^*) = P_0$.

\subsection{The analytic limit in higher rank}

Next, we generalize the previous result to higher rank.
Let $K$ be a connected, simply connected compact semisimple Lie group.  Let $W$ denote its Weyl group.

For every $\alpha_i \in \simpleroots$ there is a Hopf $*$-morphism $\iota_i:\cU_{q_i}(\lie{su}(2)) \into \Uqk$ sending $E$ to $E_i$.  The dual map $\Res_{S_i}^{K_q}:=\iota_i^*:\OqK \onto \cO_{q_i}[\SU(2)] $ is called the \emph{restriction to the quantum subgroup $S_i\cong \SU_{q_i}(2)$} associated to the simple root $\alpha_i$.  It specializes at $q\in\qrange$ to a $*$-morphism $\Res_{S_i}^{K_q}:\OKq \onto \cO[\SU_{q_i}(2)] $.
This allows us to define a $*$-representation for each  $\alpha_i\in\simpleroots$ and each $q \in \qrange$ by
\begin{equation}
\label{eq:pi_i}
  \Soib_{q, i} := \Soib_q \circ \Res_{S_i}^{K_q}: \OKq \to \cB(\ell^2(\NN)).
\end{equation}
As previously, we will use the same notation for the partially defined representation of the $\CC(q)$-algebra $\OqK$ obtained by first specializing at $q$.

We also have a restriction morphism to the maximal torus subgroup $\res^{K_q}_T : \OKq \to \OT$ for each $q \in \qrange$. Explicitly,
\[
  \res^{K_q}_T : c^V_{f^i,v_j} \mapsto \delta_{ij} c^V_{f^i,v_i},
\]
where $\{v_i\}$ is a basis of weight vectors, $\{f^i\}$ a dual basis, and the right-hand side is understood as a matrix coefficient for $T\subset K$.
Composing this with the multiplication representation $\mult:\OT\to\cB(L^2(T))$ yields the $*$-representation
\begin{equation}
\label{eq:pi_T}
  \repT = \mult \circ\res^{K_q}_T : \OKq \to \cB(L^2(T)),
\end{equation}
as well as its partially defined analogue on $\OqK$.

\begin{definition}[Soibelman \cite{Soibelman}]
\label{def:big_cell_rep}
  Fix a decomposition $w_0 = s_{i_1} \cdots s_{i_l}$ of the longest word of $W$. We put $\SoibH = \ell^2(\NN)^{\otimes l} \otimes L^2(T)$.  
  For $q \in \qrange$, we define the \emph{Soibelman representation} to be
  \begin{align*}
   \SoibT_q := (\Soib_{i_1, q} \otimes \Soib_{i_2, q} \otimes \cdots \otimes \Soib_{i_l, q} \otimes \repT) \circ \Delta^{(l)} : \OKq &\to \cB(\SoibH).
  \end{align*}
  Again, we use the same notation for the partially defined representation of $\OqK$.
\end{definition}

\begin{remark}
Soibelman \cite{Soibelman} defined a family of irreducible $*$-represent\-ations indexed by the symplectic leaves of $K$.  The representation in \cref{def:big_cell_rep} corresponds to the direct integral over all symplectic leaves of maximal dimension.   It is faithful but not irreducible.  One can obtain an irreducible representation by replacing the multiplication representation of the torus $\repT$ by any character $\rept = \ev_{t}\circ\res^{K_q}_T$ where $\ev_t$ denotes evaluation at a point $t\in T$.  For a nice account of Soibelman's work, and generalizations, see \cite{NesTus:functions}.  
\end{remark}

\begin{theorem}
\label{thm:G_limit}
For any $u \in \OqAOG$, the one-parameter family of operators $\SoibT_q(u)$ admits a norm-limit as $q \to 0$, which we denote by $\SoibT_0(u)$.
\end{theorem}

\begin{proof}
  Firstly, observe that $\Delta:\OqAOG\to\OqAOG\otimes\OqAOG$.
  Next, note that a crystal lattice $\cL$ for a $\Uqk$-module $V$ is also a $\cU_q(\lie{su}_2)$-crystal lattice for the restriction of the module $V$ to $S_i$ for every $i = 1, \cdots, \rank$.  Therefore the maps $\res^{K_q}_{S_i}:\OqK \to \cO_{q_i}[\SU(2)]$ restrict to morphisms $\OqAOG \to \cO_{q_i}^{\bfA_0}[\SL(2)]$.  The result now follows from Theorem \ref{thm:SL2_limit} and the definition of the Soibelman representation.  
\end{proof}

By the same argument as in Corollary \ref{cor:SU2_limit}, we obtain limits for our compact $\pid$-form $\OqAOK$.

\begin{corollary}
\label{cor:K_limit}
For any $u \in \OqAOK$, the one-parameter family of operators $\SoibT_q(u)$ admits a norm-limit as $q \to 0$, which we denote by $\SoibT_0(u)$.
\end{corollary}

\begin{definition}
\label{def:OK0}
We put $\OKO:=\SoibT_0(\OqAOK)$ and refer to this $*$-algebra as the crystal limit of $\OKq$.
\end{definition}

We make a similar definition for the $q=0$ limit of the quantized coordinate rings of flag varieties, following \cref{thm:flag_generators2}.  Note that $\OKO$ inherits a gauge action of the torus $T\cong\TT^\rank$ from that of $\OqK$.

\begin{definition}
\label{def:OXO}
With notation as in \cref{sec:flag_manifolds}, let $K$ be a compact connected semisimple Lie group, not necessarily simply connected, let $\rootset\subset\simpleroots$ be a set of simple roots, and let $Y_\rootset$ be the corresponding principal torus bundle over the flag variety $X_\rootset$.  We define $\OYSO$ to be the subalgebra of $\OKO$ generated by the elements $\Soib_0(\genf^\lambda_i)$ and $\Soib_0(\genv^\lambda_i)$ with $\lambda\in\dominant_{K,\rootset}$.  The gauge-invariant subalgebra is denoted by $\OXSO$.
\end{definition}

The remainder of the paper will be dedicated to describing the structure of the crystal limit $\OKO$, as well as the subalgebras $\OYSO \subset \OKO$ as above.


\section{The Cartan braiding and the crystal limit}
\label{sec:Cartan-braiding}

\subsection{Braiding and commutation relations}

At this point, we will be \linebreak obliged to enlarge our base field slightly in order to work with expressions of the form $q^{(\lambda,\lambda')}$ where $\lambda,\lambda'\in\bP$ are any integral weights.  Let $L\in\NN$ be the smallest positive integer such that $(\bP,\bP)\subseteq \frac{1}{L}\ZZ$.  Then we work over $\CC(\param)$ and put $q=s^L$.  Likewise we redefine the rings $\pid$, $\pidinfty$ and $\laurent$ with the parameter $\param$ in place of $q$.  This has essentially no effect on the crystal theory.  Rather than rewrite expressions in terms of $s$, we will write $q^k$ where $k\in\frac{1}{L}\ZZ$.

We denote by $\Rhat_{V, W}: V \otimes W \to W \otimes V$ the braiding corresponding to the finite-dimensional integrable $\Uqg$-modules $V$ and $W$.
It is a $\Uqg$-module isomorphism which generalizes the classical flip map.
The braiding is not quite unique; the choice we make can be characterized uniquely by
\begin{equation}
\label{eq:R-matrix_convention}
\begin{split}
\Rhat_{V, W}(v \otimes w) &= q^{-(\wt(v), \wt(w))} w \otimes v + \sum_i w_i \otimes v_i,\\
&
\text{
where $\wt(w_i) <\wt(w)$ and $\wt(v_i) > \wt(v)$ for every $i$.
}
\end{split}
\end{equation}
Here, $\wt$ denotes the weight of a weight vector.
In particular, in the case $V = V(\lambda)$ and with $v_\lambda$ a highest weight vector, we get $\Rhat_{V, V}(v_\lambda \otimes v_\lambda) = q^{-(\lambda, \lambda)} v_\lambda \otimes v_\lambda$.

The braiding can be used to give commutation relations for the elements of $\OqG$.
Let $V$ and $W$ be $\Uqg$-modules with bases $\{v_i\}_i$ and $\{w_i\}_i$ and denote by $\{f^i\}_i$ and $\{g^i\}_i$ the dual bases. Let us introduce the coefficients of $\Rhat_{V,W}$ by
\[
\Rhat_{V, W}(v_i \otimes w_j) = \sum_{k, l} (\Rhat_{V, W})_{i j}^{k l} w_k \otimes v_l.
\]
Then, using the fact that the braiding satisfies $\Rhat_{V, W} \Delta(X) = \Delta(X) \Rhat_{V, W}$ for any $X \in \Uqg$, it is easy to derive the relations
\begin{equation}
\label{eq:OqG-relation}
\begin{split}
c^V_{f^i, v_k} c^W_{g^j, w_l} & = \sum_{a, b, c, d} (\Rhat_{V, W}^{-1})^{i j}_{a b} (\Rhat_{V, W})^{c d}_{k l} c^W_{g^a, w_c} c^V_{f^b, v_d} \\
& = \sum_{a, b, c, d} (\Rhat_{W, V})^{i j}_{a b} (\Rhat_{W, V}^{-1})^{c d}_{k l} c^W_{g^a, w_c} c^V_{f^b, v_d}.
\end{split}
\end{equation}

\subsection{The Cartan braiding}

Before discussing the crystal limit of the braiding operators, let us make some observations about tensor products of irreducible crystals.

\begin{definition}
\label{def:Cartan_component}
A finite-dimensional integrable $\Uqg$-module $V$ will be called a \emph{product of irreducibles} if it is of the form $V(\lambda_1)\otimes\cdots\otimes V(\lambda_n)$ for some $\lambda_1, \cdots, \lambda_n \in \dominant$.  Such a module contains a unique irreducible submodule of highest weight $\lambda = \sum_i \lambda_i$, which we call the \emph{Cartan component}.  We refer to $\lambda$ as the \emph{highest weight} of $V$ (although strictly speaking it is the largest among the highest weights of all irreducible submodules of $V$).  

The same terminology will be used to refer to the associated \emph{Cartan component} of a product of irreducible crystals.
If $\cB = \cB(\lambda_1) \otimes \cdots \otimes \cB(\lambda_n)$ is a product of irreducible crystals, we denote by $\eta:\cB\sqcup\{0\}\to\{0,1\}$ the indicator function of its Cartan component: 
\[
 \eta(b)= 
 \begin{cases}
  1 & \text{if $b$ is in the Cartan component}, \\
  0 &\text{otherwise}.
 \end{cases}
\]
\end{definition}

The following well-known property of the Cartan component will occasionally be useful (we include the easy proof for completeness).

\begin{lemma}
\label{lem:tensor_fact}
Let $b_{\lambda}$ and $b_{w_0\lambda}$ denote the highest and lowest weight elements, respectively, of the irreducible crystal $\cB(\lambda)$. Then the tensor $c\otimes b_{\lambda'}$ belongs to the Cartan component of $\cB(\lambda)\otimes\cB(\lambda')$ for every $c\in\cB(\lambda)$.  Similarly, $b_{w_0\lambda}\otimes c'$ belongs to the Cartan component for every $c'\in\cB(\lambda')$.
\end{lemma}

\begin{proof}
The first statement follows by repeatedly applying the Kashiwara operators $\kasF_i$ to the highest weight element $b_\lambda \otimes b_{\lambda'}$ and using the tensor product formula \eqref{eq:tensor_F}. Similarly, the second follows by applying the Kashiwara operators $\kasE_i$ to the lowest weight element $b_{w_0\lambda} \otimes b_{w_0\lambda'}$.
\end{proof}

We now return to the braiding.
Consider two simple modules $V(\lambda)$ and $V(\lambda')$.
We have crystal lattices of $V(\lambda)\otimes V(\lambda')$ and $V(\lambda')\otimes V(\lambda)$ given by
\[
\lattice_{V(\lambda) \otimes V(\lambda')} := \lattice(\lambda) \otimes_{\pid} \lattice(\lambda^\prime), \quad
\lattice_{V(\lambda') \otimes V(\lambda)} := \lattice(\lambda^\prime) \otimes_{\pid} \lattice(\lambda).
\]
The braiding $\Rhat_{V(\lambda), V(\lambda')}$ does not typically map $\lattice_{V(\lambda)\otimes V(\lambda')}$ into $\lattice_{V(\lambda')\otimes V(\lambda)}$ because of the negative powers of $q$ in \eqref{eq:R-matrix_convention}.
This deficiency can be remedied by multiplying by an appropriate power of $q$, as we show in the next theorem.

\begin{theorem}
\label{thm:braiding_limit}
Given any simple modules $V(\lambda)$ and $V(\lambda')$ we have
\[
q^{(\lambda, \lambda')} \Rhat_{V(\lambda), V(\lambda')}(\cL_{V(\lambda) \otimes V(\lambda')}) \subseteq \cL_{V(\lambda') \otimes V(\lambda)}.
\]
Moreover, the map $q^{(\lambda, \lambda')} \Rhat_{V(\lambda), V(\lambda')}$ induces a morphism of crystals
\[
\braid_{\cB(\lambda),\cB(\lambda')}:\cB(\lambda)\otimes\cB(\lambda') \to \cB(\lambda') \otimes \cB(\lambda).
\]
It is an isomorphism between the Cartan components of the two crystals, and is zero on all other components.
\end{theorem}

\begin{proof}
It is a consequence of the tensor product rule \eqref{eq:tensor_E} for $\kasE_i$ that every highest weight element of the crystal $\cB(\lambda) \otimes \cB(\lambda')$ is of the form $b_\lambda \otimes b'$, where $b_\lambda$ is the highest weight element of $\cB(\lambda)$, and $b'$ is some element of $\cB(\lambda')$.  Let $v_\lambda \in \cL(\lambda)$ and $w \in \cL(\lambda')$ be lifts of $b_\lambda$ and $b'$, respectively, to weight vectors.  Write $\mu$ for the weight of $w$.
  
Let us write $V = V(\lambda)$ and $W = V(\lambda')$ for convenience.
By the nature of the braiding described in \eqref{eq:R-matrix_convention}, we have
\begin{equation}
\label{eq:rescaled_braiding1}
q^{(\lambda, \lambda')} \Rhat_{V,W}(v_\lambda \otimes w) = q^{(\lambda, \lambda') - (\lambda, \mu)} w \otimes v_\lambda.
\end{equation}
Since $\lambda \in \dominant$ and $\lambda' - \mu \in \bQ^+ = \NN \cdot \simpleroots$, we have $(\lambda,\lambda'-\mu)\geq0$.  Therefore, \eqref{eq:rescaled_braiding1} is in $\cL_{W\otimes V}$, and moreover is in $q\cL_{W\otimes V}$ unless $w$ is the highest weight vector of $W$.
  
The images of the various highest weight elements $v_\lambda\otimes w$ above, under repeated action of the Kashiwara operators $\kasF_i$, generate the crystal lattice $\cL_{V,W}$ as an $\pid$-module. Since $\Rhat_{V,W}$ is a morphism of $\Uqg$-modules, we get
\[
q^{(\lambda, \lambda')} \Rhat_{V, W}(\kasF(v_\lambda \otimes w)) =
q^{(\lambda, \lambda') - (\lambda, \mu)} \kasF (w \otimes v_\lambda)
\in \cL_{W \otimes V}
\]
for any word $\kasF = \kasF_{i_1} \cdots \kasF_{i_n}$ in the Kashiwara operators $\kasF_i$.  
This proves that the map $q^{(\lambda,\lambda')} \Rhat_{V,W}$ respects the crystal lattices and induces a map of crystals.

For the final claim, note that crystal element $b'\otimes b_\lambda$ associated to the right hand side of \eqref{eq:rescaled_braiding1} always belongs to the Cartan component by \linebreak \cref{lem:tensor_fact}.  Since the left-hand side is a highest-weight element by assumption, it belongs to the Cartan component only if $w$ is of weight $\lambda'$.  It follows that the crystal morphism $\sigma_{\cB(\lambda),\cB(\lambda')}$ must vanish on all non-Cartan components.  On the Cartan component, \eqref{eq:rescaled_braiding1} gives \linebreak $q^{(\lambda,\lambda')}\Rhat_{V,W}(v_\lambda\otimes v_{\lambda'}) = v_{\lambda'}\otimes v_\lambda$, from which we conclude that $\sigma_{\cB(\lambda),\cB(\lambda')}$ is an isomorphism here.
\end{proof}

\begin{definition}
\label{def:Cartan-braiding}
Let $V$ and $V'$ be products of irreducible modules with highest weights $\lambda$ and $\lambda'$, respectively, and let $(\cL, \cB)$ and $(\cL', \cB')$ be their crystal bases.  The \emph{Cartan braiding} is the morphism of crystals
\[
\braid = \braid_{\cB, \cB'}: \cB\otimes \cB' \to \cB' \otimes \cB
\]
induced by the morphism of crystal lattices $q^{(\lambda,\lambda')} \Rhat_{V, V'}: \cL\otimes_\pid\cL' \to \cL'\otimes_\pid\cL$. 
\end{definition}

The use of the word braiding here is a bit misleading.  Firstly, braidings are supposed to be isomorphisms, while our Cartan braiding is not, see \cref{thm:braiding_limit}.  More importantly, the Cartan braiding $\braid_{\cB,\cB'}$ is only defined when $\cB$ and $\cB'$ products of irreducible modules, not on general direct sums of irreducibles.  This is because we need a well-defined notion of highest weights $\lambda$ and $\lambda'$ for $\cB$ and $\cB'$ in order to define it, see below.

If $V$ and $V'$ are irreducible modules, then the Cartan braiding is the crystal morphism described in \cref{thm:braiding_limit}, that is, it is an isomorphism between the Cartan components of $\cB \otimes \cB'$ and $\cB' \otimes \cB$, and zero on all other irreducible components. 

If $V$ and $V'$ are merely products of irreducibles, then we need a few remarks to justify that the Cartan braiding is well-defined.  If $\mu, \mu' \in \dominant$ are the highest weights of any components of $V$ and $V'$ respectively, then we have $\lambda - \mu$ and $\lambda - \mu'$ in $\NN \cdot \simpleroots$.  Therefore
\[
(\lambda, \lambda') = (\lambda - \mu, \lambda') + (\mu, \lambda' - \mu') + (\mu, \mu') \geq (\mu, \mu'),
\]
so $q^{(\lambda, \lambda')} \Rhat_{V(\mu),V(\mu')} = q^k q^{(\mu, \mu')} \Rhat_{V(\mu), V(\mu')}$ for some $k \geq 0$.
Now \cref{thm:braiding_limit} shows that $q^{(\lambda, \lambda')} \Rhat_{V, V'}: \cL \otimes_\pid \cL' \to \cL'\otimes_\pid \cL$ is well defined and descends to a morphism of crystals.
Note though that in this case, it is possible for $\braid_{\cB, \cB'}$ to be non-zero on components other than the Cartan component, as the following example shows.

\begin{example}
  Let $\lie{g} = \lie{sl}_4$, and consider $\cB = \cB(\varpi_1)$ and $\cB' = \cB(\varpi_2)\otimes\cB(\varpi_3)$, where $\varpi_1,\varpi_2,\varpi_3$ are the fundamental weights.  Then $\cB'\cong \cB(\varpi_2+\varpi_3)\oplus\cB(\varpi_1)$.  One can calculate
  \[
    (\varpi_1,\varpi_2+\varpi_3) = {\frac34} = (\varpi_1,\varpi_1).
  \]
  It follows that $\braid_{\cB,\cB'}$ is non-zero on two distinct components in the triple product $\cB(\varpi_1)\otimes\cB(\varpi_2) \otimes\cB(\varpi_3)$.
  
  On the other hand, if we consider the map $\id_{\cB(\varpi_1)}\otimes\braid_{\cB(\varpi_2),\cB(\varpi_3)}$, this will kill $\cB(\varpi_1) \otimes \cB(\varpi_1) \subset \cB(\varpi_1) \otimes (\cB(\varpi_2)\otimes\cB(\varpi_3))$.  In fact, we will show that the Cartan component of any tensor product of irreducibles is characterized by the non-vanishing of every possible action of the Cartan braidings. See \cref{thm:longest_word} below for a precise statement.
\end{example}

The Cartan braiding is very easy to describe in one particular situation.

\begin{example}
When $V = V' = V(\lambda)$, the Cartan braiding is the ``projection'' onto the component of highest weight $2\lambda$, \emph{i.e.},
\[
 \braid_{\cB(\lambda),\cB(\lambda)} (b\otimes b') =
 \begin{cases}
  b\otimes b' & \text{if } \eta(b\otimes b')=1,\\
  0 & \text{otherwise.}
 \end{cases}
\]
\end{example}

Throughout this paper we will consider the running example of $K = \SU(3)$, with the aim of rederiving the result of Giselsson mentioned in the introduction. We begin by determining the Cartan braiding in this case.

\begin{example}
\label{ex:braiding-sl3}
For $\lie{g} = \mathfrak{sl}_3$, we denote the crystals graphs of $\cB(\varpi_1)$ and $\cB(\varpi_2)$ by
\begin{center}
\begin{tikzcd}[column sep=2em]
a_1 \arrow[r, "1"] & a_2 \arrow[r, "2"] & a_3 & & b_1 \arrow[r, "2"] & b_2 \arrow[r, "1"] & b_3
\end{tikzcd}
\end{center}
Using the tensor product rule for the Kashiwara operators, we find that the crystal graphs of $\cB(\varpi_1) \otimes \cB(\varpi_1)$ and $\cB(\varpi_2) \otimes \cB(\varpi_2)$ are given by
\begin{center}
\begin{tikzcd}[column sep=1.5em]
a_1 \otimes a_1 \arrow[r, "1"] & a_2 \otimes a_1 \arrow[r, "2"] \arrow[d, "1"] & a_3 \otimes a_1 \arrow[d, "1"] & & b_1 \otimes b_1 \arrow[r, "2"] & b_2 \otimes b_1 \arrow[r, "1"] \arrow[d, "2"] & b_3 \otimes b_1 \arrow[d, "2"] \\
a_1 \otimes a_2 \arrow[d, "2"] & a_2 \otimes a_2 \arrow[r, "2"] & a_3 \otimes a_2 \arrow[d, "2"] & & b_1 \otimes b_2 \arrow[d, "1"] & b_2 \otimes b_2 \arrow[r, "1"] & b_3 \otimes b_2 \arrow[d, "1"] \\
a_1 \otimes a_3 \arrow[r, "1"] & a_2 \otimes a_3 & a_3 \otimes a_3 & & b_1 \otimes b_3 \arrow[r, "2"] & b_2 \otimes b_3 & b_3 \otimes b_3
\end{tikzcd}
\end{center}
Similarly, for the products $\cB(\varpi_1) \otimes \cB(\varpi_2)$ and $\cB(\varpi_2) \otimes \cB(\varpi_1)$ we have
\begin{center}
\begin{tikzcd}[column sep=1.5em]
a_1 \otimes b_1 \arrow[r, "1"] \arrow[d, "2"] & a_2 \otimes b_1 \arrow[r, "2"] & a_3 \otimes b_1 \arrow[d, "2"] & & b_1 \otimes a_1 \arrow[r, "2"] \arrow[d, "1"] & b_2 \otimes a_1 \arrow[r, "1"] & b_3 \otimes a_1 \arrow[d, "1"] \\
a_1 \otimes b_2 \arrow[r, "1"] & a_2 \otimes b_2 \arrow[d, "1"] & a_3 \otimes b_2 \arrow[d, "1"] & & b_1 \otimes a_2 \arrow[r, "2"] & b_2 \otimes a_2 \arrow[d, "2"] & b_3 \otimes a_2 \arrow[d, "2"] \\
a_1 \otimes b_3 & a_2 \otimes b_3 \arrow[r, "2"] & a_3 \otimes b_3 & & b_1 \otimes a_3 & b_2 \otimes a_3 \arrow[r, "1"] & b_3 \otimes a_3
\end{tikzcd}
\end{center}
Using \cref{thm:braiding_limit}, we deduce from the latter two diagrams that the Cartan braiding $\braid = \braid_{\cB(\varpi_1), \cB(\varpi_2)}$ is given by
\begin{equation}
\label{eq:sl3-braiding}
\begin{array}{lll}
\braid(a_1 \otimes b_1)  = b_1 \otimes a_1, &
\braid(a_2 \otimes b_1)  = b_1 \otimes a_2, &
\braid(a_3 \otimes b_1)  = b_2 \otimes a_2, \\
\braid(a_1 \otimes b_2)  = b_2 \otimes a_1, &
\braid(a_2 \otimes b_2)  = b_3 \otimes a_1, &
\braid(a_3 \otimes b_2)  = b_2 \otimes a_3, \\
\braid(a_1 \otimes b_3)  = 0, &
\braid(a_2 \otimes b_3)  = b_3 \otimes a_2, &
\braid(a_3 \otimes b_3)  = b_3 \otimes a_3.
\end{array}
\end{equation}
Note that this is different from the flip map, even on the Cartan component.
\end{example}

\subsection{Crystal limit of the commutation relations}

As previously, we fix a weight basis $\{v^\lambda_i\}_i\subset \cL(\lambda)$ lifting the crystal basis $\{b^\lambda_i\}_i$ of $\cB(\lambda)$, and let $\{f^i_\lambda\}_i \subset V(\lambda)^*$ be the dual basis.  We will usually suppress the highest weight $\lambda$ in the notation and write $v_i$ and $f^i$.
We make the convention that $v^\lambda_1$ is the highest weight vector, and hence $f_\lambda^1$ is the lowest weight vector.
We also denote by $\{\tilde{v}^\lambda_i\}_i \subset V(\lambda)^{**}$ the dual basis to $\{f_\lambda^i\}_i$.

We recall the notation for the generators of $\OqAOK$ from \cref{sec:compact_form},
\begin{equation}
\label{eq:generators}
\genf^\lambda_i := c^{V(\lambda)}_{f^i, v_1}, \quad
\genv^\lambda_i := c^{V(\lambda)^*}_{\tilde{v}_i, f^1}.
\end{equation}

\begin{proposition}
\label{prop:fxf}
 Let $\lambda, \lambda' \in \dominant$ and let $\projection_{\lambda,\lambda'}:\cB(\lambda)\otimes\cB(\lambda') \to \cB(\lambda+\lambda')$ be the unique non-trivial crystal morphism.  If $b^\lambda_i \otimes b^{\lambda'}_j$ is in the Cartan component, with $\projection_{\lambda,\lambda'}(b^\lambda_i \otimes b^{\lambda'}_j) = b^{\lambda+\lambda'}_m $, then we have
  \begin{align}
  \label{eq:fxf}
    \genf^\lambda_i \genf^{\lambda'}_j &\equiv \genf^{\lambda+\lambda'}_m \mod{q \OqAOK}, \\
  \label{eq:vxv}
    \genv^{\lambda'}_j \genv^\lambda_i  &\equiv \genv^{\lambda+\lambda'}_m \mod{q \OqAOK}.
  \end{align}
Otherwise $\genf^\lambda_i \genf^{\lambda'}_j\equiv 0$ and $\genv^{\lambda'}_j \genv^\lambda_i \equiv 0 \pmod {q \OqAOK}$.
\end{proposition}

\begin{proof}
  We have $\genf^\lambda_i \genf^{\lambda'}_j = c^{V(\lambda)\otimes V(\lambda')}_{f^i_\lambda \otimes f^j_{\lambda'}, v^\lambda_1\otimes v^{\lambda'}_1}$.  
  Let us write $\cB(\lambda)\otimes\cB(\lambda') = \bigsqcup_n \cB_n$ for the decomposition of the tensor product crystal into irreducible components, with $\cB_1$ being the Cartan component.  Let $V_n$ and $\cL_n$ be the corresponding irreducible components of $V(\lambda)\otimes V(\lambda')$ and $\cL(\lambda)\otimes_\pid \cL(\lambda')$.  
  
  The vector $v^\lambda_1\otimes v^{\lambda'}_1$ is the highest weight vector for the Cartan component $\cB_1$.
  If $b^\lambda_i \otimes b^{\lambda'}_j$ is contained in some component $\cB_n$ other than the Cartan component, then  $v^\lambda_i\otimes v^{\lambda'}_j \in \cL_n + q (\cL(\lambda)\otimes_{\pid}\cL(\lambda'))$, and likewise $f^i_\lambda\otimes f^j_{\lambda'}\in \cL_n^* + q (\cL(\lambda)^*\otimes_{\pid}\cL(\lambda')^*) $.  Therefore $\genf^\lambda_i \genf^{\lambda'}_j \in q \OqAOK$.  
  
  On the other hand, if $b^\lambda_i \otimes b^{\lambda'}_j$ is contained in the Cartan component, then $f^\lambda_i\otimes f^{\lambda'}_j \equiv f^{\lambda+\lambda'}_m \pmod{q(\cL(\lambda)\otimes_{\pid}\cL(\lambda'))}$, where we are identifying $f^{\lambda+\lambda'}_m$ with its image in $\cL_1^*\cong\cL(\lambda+\lambda')^*$.  The relation \eqref{eq:fxf} follows, and by applying the involution $*$ and \cref{prop:v_f_adjoints} we obtain \eqref{eq:vxv}.
\end{proof}

We note that \cref{prop:fxf} implies the exchange relation $\genf^\lambda_i \genf^{\lambda'}_j \equiv \genf^{\lambda'}_k \genf^\lambda_l$ modulo $q \OqAOK$ when $b^\lambda_i\otimes b^{\lambda'}_j$ is in the Cartan component and \linebreak $\braid_{\cB(\lambda), \cB(\lambda')}(b^\lambda_i\otimes b^{\lambda'}_j) = b^{\lambda'}_k \otimes b^{\lambda}_l$.  
The following proposition gives more precision, as well as giving exchange relations for the $\genf^\lambda_i$ with the $\genv^{\lambda'}_j$.

\begin{proposition}
\label{prop:exchange-relations}
We have the commutation relations
\[
\begin{split}
\genf^\lambda_i \genf^{\lambda^\prime}_j & = \sum_{k, l} q^{(\lambda, \lambda^\prime)} \left( \Rhat_{V(\lambda^\prime), V(\lambda)} \right)^{i j}_{k l} \genf^{\lambda^\prime}_k \genf^\lambda_l, \\
\genv^\lambda_i \genv^{\lambda^\prime}_j & = \sum_{k, l} q^{(\lambda, \lambda^\prime)} \left( \Rhat_{V(\lambda^\prime), V(\lambda)} \right)^{l k}_{j i} \genv^{\lambda^\prime}_k \genv^\lambda_l, \\
\genf^\lambda_i \genv^{\lambda^\prime}_j & = \sum_{k, l} q^{(\lambda, \lambda^\prime)} \left( \Rhat_{V(\lambda), V(\lambda^\prime)} \right)^{k i}_{l j} \genv^{\lambda^\prime}_k \genf^\lambda_l.
\end{split}
\]
Moreover we have $\sum_i \genv^\lambda_i \genf^\lambda_i = 1$.
\end{proposition}

\begin{proof}
These can be derived from the commutation relations \eqref{eq:OqG-relation} which hold in $\OqG$.
Since $v_1$ is the highest weight vector we have \linebreak $(\Rhat_{V(\lambda^\prime), V(\lambda)}^{-1})^{c d}_{1 1} = q^{(\lambda, \lambda^\prime)} \delta^c_1 \delta^d_1$, which gives the relation for $\genf^\lambda_i \genf^{\lambda'}_j$ in the form
\[
c^{V(\lambda)}_{f^i, v_1} c^{V(\lambda^\prime)}_{f^j, v_1} = \sum_{a, b} q^{(\lambda, \lambda^\prime)} (\Rhat_{V(\lambda^\prime), V(\lambda)})^{i j}_{a b} c^{V(\lambda^\prime)}_{f^a, v_1} c^{V(\lambda)}_{f^b, v_1}.
\]
Similarly, since $f^1$ is the lowest weight vector we have $(\Rhat_{V(\lambda^\prime)^*, V(\lambda)^*}^{-1})^{c d}_{1 1} = q^{(\lambda, \lambda^\prime)} \delta^c_1 \delta^d_1$, which gives the relation for $\genv^\lambda_i \genv^{\lambda'}_j$ in the form
\[
c^{V(\lambda)^*}_{\tilde{v}_i, f^1} c^{V(\lambda^\prime)^*}_{\tilde{v}_j, f^1} = \sum_{a, b} q^{(\lambda, \lambda^\prime)} (\Rhat_{V(\lambda^\prime)^*, V(\lambda)^*})^{i j}_{a b} c^{V(\lambda^\prime)^*}_{\tilde{v}_a, f^1} c^{V(\lambda)^*}_{\tilde{v}_b, f^1}.
\]
It can be rewritten in terms of $\Rhat_{V(\lambda^\prime), V(\lambda)}$ by using the identities
\begin{equation}
\label{eq:braid-coefficients-identities}
(\Rhat_{V^*, W})_{i j}^{k l} = (\Rhat_{V, W}^{-1})_{j l}^{i k}, \quad
(\Rhat_{V, W^*}^{-1})_{i j}^{k l} = (\Rhat_{V, W})_{j l}^{i k}.
\end{equation}
These are valid for any $\Uqg$-modules $V, W$ (with dual bases for $V^*, W^*$) and follow from the fact that the finite-dimensional $\Uqg$-modules form a braided monoidal category (see \cite[Lemma A.1]{Matassa:kahler} for a proof). They imply that
\[
(\Rhat_{V(\lambda^\prime)^*, V(\lambda)^*})^{i j}_{a b} = (\Rhat_{V(\lambda^\prime), V(\lambda)})^{b a}_{j i},
\]
which gives the relation for $\genv^\lambda_i \genv^{\lambda'}_j$ as stated.

Proceeding as above, we have $(\Rhat_{V(\lambda), V(\lambda^\prime)})^{c d}_{1 1} = q^{(\lambda, \lambda^\prime)} \delta^c_1 \delta^d_1$ and this gives the relation for $\genf^\lambda_i \genv^{\lambda^\prime}_j$ in the form
\[
c^{V(\lambda)}_{f^i, v_1} c^{V(\lambda^\prime)^*}_{\tilde{v}_j, f^1} = \sum_{a, b} q^{(\lambda, \lambda^\prime)} (\Rhat_{V(\lambda), V(\lambda^\prime)^*}^{-1})^{i j}_{a b} c^{V(\lambda^\prime)^*}_{\tilde{v}_a, f^1} c^{V(\lambda)}_{f^b, v_1}.
\]
It can be rewritten using the identity $(\Rhat_{V(\lambda), V(\lambda^\prime)^*}^{-1})^{i j}_{a b} = (\Rhat_{V(\lambda), V(\lambda^\prime)})^{a i}_{b j}$ from \eqref{eq:braid-coefficients-identities}, which gives the relation as stated.

Finally we want to show that $\sum_i \genv^\lambda_i \genf^\lambda_i = 1$.
It is not difficult show that $\sum_i \genv^\lambda_i \genf^\lambda_i$ is invariant under the right action of $\Uqg$, which means that $\sum_i \genv^\lambda_i \genf^\lambda_i = c \cdot 1$ for some $c \in \CC$.
Applying the counit to the left-hand side we have
\[
\sum_i \varepsilon(\genv^\lambda_i \genf^\lambda_i) = \sum_i \varepsilon(c^{V(\lambda)^*}_{\tilde{v}_i, f^1}) \varepsilon(c^{V(\lambda)}_{f^i, v_1}) = \sum_i \tilde{v}_i(f^1) f^i(v_1) = 1.
\]
Therefore $c = 1$, which gives the claimed relation.
\end{proof}

Note that the third exchange relations from \cref{prop:exchange-relations} yields the following factorization for our compact $\pid$-form $\OqKAO$.

\begin{corollary}
We have $\OqKAO = \OqGAOM \cdot \OqGAOP$.
\end{corollary}


The key consequence of \cref{prop:exchange-relations} is that, thanks to \cref{thm:braiding_limit}, the stated commutation relations also make sense when we specialize at $q = 0$. This allows us to deduce a set of commutation relations for the analytic limits $\SoibT_0(\genf^\lambda_i)$ and $\SoibT_0(\genv^\lambda_i)$ in $\OKO$.  
Combining all the above results, we get relations at $q=0$ as follows.

\begin{theorem}
\label{thm:pi0_relations}
The following relations hold in $\cO[K_0] = \SoibT_0(\OqAOK)$. 
\begin{enumerate}
 \item[1)] For any $\lambda, \lambda' \in \dominant$ and any $i,j$ we have
\begin{align}
\label{eq:pi0_relations1}
  \SoibT_0(\genf^\lambda_i) \SoibT_0(\genf^{\lambda^\prime}_j) 
  &= \eta(b^\lambda_i\otimes b^{\lambda'}_j) \SoibT_0(\genf^{\lambda+\lambda'}_m) ,
\\
\label{eq:pi0_relations2}
  \SoibT_0(\genv^{\lambda'}_j) \SoibT_0(\genv^{\lambda}_i) 
  &= \eta(b^\lambda_i\otimes b^{\lambda'}_j) \SoibT_0(\genv^{\lambda+\lambda'}_m) ,
\end{align}
where we write $b^\lambda_i\otimes b^{\lambda'}_j \mapsto b^{\lambda+\lambda'}_m$ under the unique surjective crystal morphism $\cB(\lambda)\otimes\cB(\lambda')\to\cB(\lambda+\lambda')$.

\item[2)] For any $\lambda, \lambda' \in \dominant$ and any $i, j$ we have 
\begin{align}
\label{eq:pi0_relations3}
\SoibT_0(\genf^\lambda_i) \SoibT_0(\genv^{\lambda^\prime}_j) &= 
\sum_{k, l}
\SoibT_0(\genv^{\lambda^\prime}_k) \SoibT_0(\genf^\lambda_l),
\end{align}
where the sum is over all pairs $(k,l)$ such that $\braid_{\cB(\lambda),\cB(\lambda')}(b^\lambda_l \otimes b^{\lambda^\prime}_j) = b^{\lambda^\prime}_k \otimes b^\lambda_i$.

\item[3)] For any $\lambda \in \dominant$ we have
 \begin{equation}
 \sum_i \SoibT_0(\genv^\lambda_i) \SoibT_0(\genf^\lambda_i) = 1.
 \end{equation}
 
 \item[4)]
 For any $\lambda \in \dominant$ and any $i$ we have $\SoibT_0(\genf^\lambda_i)^* = \SoibT_0(\genv^\lambda_i)$.
 
\end{enumerate}
\end{theorem}

\begin{proof}
1) and 4) follow from \cref{prop:fxf} and \cref{prop:v_f_adjoints}, while 3) follows from \cref{prop:exchange-relations}.
The relation 2) follows from the third relation of \cref{prop:exchange-relations} using the following consequence of \cref{thm:braiding_limit}:
\[
\lim_{q \to 0} q^{(\lambda, \lambda')} \left( \Rhat_{V(\lambda), V(\lambda')} \right)_{i j}^{k l} = 
\begin{cases}
 1 & \text{if }\braid_{\cB(\lambda),\cB(\lambda')}(b^\lambda_i \otimes b^{\lambda'}_j) = b^{\lambda'}_k \otimes b^{\lambda'}_l, \\
 0 & \text{if } \eta(b^\lambda_i\otimes b^{\lambda'}_j)=0.
\end{cases} \qedhere
\]
\end{proof}

Ultimately, we will show that \cref{thm:pi0_relations} gives a complete set of generators and relations for $\cO[K_0]$.

\begin{remark}
\label{rmk:exchange_relations}
Note that relations \eqref{eq:pi0_relations1} and \eqref{eq:pi0_relations2} imply the exchange relations
\begin{align}
\label{eq:pi0_relations1b}
  \SoibT_0(\genf^\lambda_i) \SoibT_0(\genf^{\lambda^\prime}_j) &= \SoibT_0(\genf^{\lambda^\prime}_k) \SoibT_0(\genf^\lambda_l), 
\\
\label{eq:pi0_relations2b}
  \SoibT_0(\genv^{\lambda'}_j) \SoibT_0(\genv^{\lambda}_i) 
  &= \SoibT_0(\genv^{\lambda}_l) \SoibT_0(\genv^{\lambda'}_k),  
\end{align}
whenever $\braid_{\cB(\lambda),\cB(\lambda')}(b^\lambda_i\otimes b^{\lambda'}_j) = b^{\lambda'}_k \otimes b^\lambda_l$.  These also follow from the $q \to 0$ limit of the relations in \cref{prop:exchange-relations}.
\end{remark}

\section{Properties of the Cartan braiding}
\label{sec:properties-Cartan}

\subsection{Hexagon and braid relations}

The Cartan braiding of \cref{def:Cartan-braiding} inherits properties from the braiding operators $\Rhat_{V,V'}$ as follows.

\begin{theorem}
Let $V$, $V'$ and $V''$ be products of irreducibles with crystal bases $(\cL, \cB)$, $(\cL', \cB')$ and $(\cL'', \cB'')$, respectively. Then the crystal braiding maps satisfy the hexagon relations
\begin{align}
  \braid_{\cB,\cB'\otimes\cB''} &= (\id_{\cB'}\otimes\braid_{\cB,\cB''})(\braid_{\cB,\cB'}\otimes\id_{\cB''}),
  \label{eq:crystal_hexagon1} \\
  \braid_{\cB\otimes\cB',\cB''} &= (\braid_{\cB,\cB''}\otimes\id_{\cB'})(\id_{\cB}\otimes\braid_{\cB',\cB''}),  
  \label{eq:crystal_hexagon2} 
\end{align}
and the braid relation
\begin{align}
 (\braid_{\cB',\cB''} \otimes \id_\cB) &(\id_\cB'\otimes \braid_{\cB,\cB''}) (\braid_{\cB,\cB'} \otimes \id_{\cB''}) \nonumber \\
  &= (\id_{\cB''} \otimes \braid_{\cB,\cB'}) (\braid_{\cB,\cB''}\otimes\id_{\cB'})(\id_\cB\otimes\braid_{\cB',\cB''}).
  \label{eq:Cartan_braiding} 
\end{align}
\end{theorem}

\begin{proof}
Let $\lambda$, $\lambda'$ and $\lambda''$ be the highest weights of $V$, $V'$ and $V''$, with crystals $\cB$, $\cB'$ and $\cB''$, respectively. Then the highest weight of $\cB' \otimes \cB''$ is $\lambda' + \lambda''$.  Using the hexagon relation for $\hat{R}$, we get
  \begin{align*}
    q^{(\lambda, \lambda' + \lambda'')} \Rhat_{V, V' \otimes V''} 
     & = q^{(\lambda, \lambda'')} q^{(\lambda, \lambda')} (\id_{V'} \otimes \Rhat_{V, V''}) (\Rhat_{V, V'} \otimes \id_{V''}).
  \end{align*}
This specializes at $q = 0$ by \cref{thm:braiding_limit}, and the specialization induces the hexagon relation \eqref{eq:crystal_hexagon1}. The other hexagon relation is proven similarly, and the braid relation follows by an analogous calculation.
\end{proof}

\subsection{Partial action of the symmetric group}

Let $\cB_1, \cdots, \cB_n$ be products of irreducible crystals, see \cref{def:Cartan_component}.  For long tensor products of the form $\cB_1\otimes\cdots\otimes\cB_n$, we will denote the action of the crystal braiding on two successive terms by $\braid_i = \id_{\cB_1} \otimes \cdots \otimes \braid_{\cB_i,\cB_{i+1}} \otimes \cdots \otimes \id_{\cB_n}$, so that
\[
  \braid_i : \cB_1\otimes\cdots\cB_i\otimes\cB_{i+1}\otimes \cdots \otimes \cB_n \to \cB_1\otimes\cdots\cB_{i+1}\otimes\cB_{i}\otimes \cdots \otimes \cB_n.
\]

\begin{proposition}
  Let $\cB_1, \cdots, \cB_n$ be irreducible crystals.  Let $s\in S_n$ be a permutation and let $s = s_{i_1} \cdots s_{i_k}$ be a realization of $s$ as a reduced word in the transpositions $s_i=(i\;\;i\!+\!1)$. Then the crystal morphism 
  \[
    \braid_s \defeq \braid_{i_1} \circ \cdots \circ \braid_{i_n} : \cB_1\otimes \cdots \cB_n \to \cB_{s(1)} \otimes \cdots \otimes \cB_{s(n)}
  \]
  is independent of the choice of reduced word representing $s$. 
\end{proposition}

\begin{proof}
 Thanks to \cref{thm:braiding_limit}, this follows from the well-known fact that any two reduced words can be linked by repeated application of the braid relations $s_is_{i+1}s_i = s_{i+1}s_is_{i+1}$.  
 \end{proof} 

Note that the relation $s_i^2=\id$ is not necessary in the above proof.  This is important since $\braid_i^2\neq\id$, although $\braid_i^2$ is the identity on the Cartan component.  As a consequence, the map $s\mapsto \braid_s$ doesn't define an action of $S_n$, but it does define an action between the Cartan components, thanks to the following theorem.

\begin{theorem}
\label{thm:longest_word}
Let $\cB_1, \cdots, \cB_n$ be irreducible crystals and let $b_i \in \cB_i$ for each $i$.  The following are equivalent:
\begin{enumerate}
 \item[(i)] $b_1\otimes\cdots\otimes b_n$ is in the Cartan component of $\cB_1\otimes\cdots\otimes\cB_n$,
 \item[(ii)] $\braid_s(b_1\otimes \cdots\otimes b_n) \neq 0$ for every $s\in S_n$,
 \item[(iii)] $\braid_{s_0}(b_1\otimes\cdots\otimes b_n) \neq 0$, where $s_0 = \scriptstyle \begin{pmatrix} 1 & 2 & \cdots & n \\ n & n - 1 & \cdots & 1 \end{pmatrix}$ is the longest permutation.
\end{enumerate}
If $\cB_1, \cdots, \cB_n$ are products of irreducibles, the equivalence holds if we add to (ii) and (iii) the condition that each $b_i$ be in the Cartan component of $\cB_i$.
\end{theorem}

\begin{proof}
  To begin with, assume that all the $\cB_i$ are irreducible.
  
  For (i) $\thus$ (ii), note that if $b_1\otimes \cdots \otimes b_n$ is in the Cartan component of $\cB_1\otimes\cdots\otimes\cB_n$, then necessarily $b_i\otimes b_{i+1}$ is in the Cartan component of $\cB_i\otimes\cB_{i+1}$ for every $i$.  Therefore $\braid_{\cB_i, \cB_{i + 1}}(b_i \otimes b_{i + 1})$ is non-zero and, since it corresponds to the same element of the Cartan component as $b_i\otimes b_{i+1}$, we have that $\braid_i(b_1 \otimes \cdots \otimes b_n)$ again belongs to the Cartan component.
  Inductively, any repeated composition of the $\braid_i$ is non-zero.
  
  The implication (ii) $\thus$ (iii) is obvious.  To prove (iii) $\thus$ (i), we work inductively on $n$.  When $n=1$ or $2$, the claim is immediate.  Suppose it is true for $n-1$.   If $\braid_{s_0}(b_1\otimes\cdots\otimes b_n) \neq 0$, write 
  \begin{equation}
  \label{eq:s0_words}
  s_0 = s_0' (s_{n-1} s_{n-2} \cdots s_1) = (s_{n-1} s_{n-2} \cdots s_1) s_0'',
  \end{equation} where
  \[
   s_0' = \scriptstyle \begin{pmatrix} 1&2&\cdots& n-1&n\\n-1&n-2&\cdots&1&n \end{pmatrix}, \quad
   s_0'' = \scriptstyle \begin{pmatrix} 1&2&\cdots& n-1&n\\1 & n & \cdots&3&2 \end{pmatrix}
  \] 
  are the permutations reversing the first $n-1$ and last $n-1$ entries, respectively.  Upon writing $s_0'$ and $s_0''$ as reduced words, both expressions in \eqref{eq:s0_words} give a reduced expression for $s_0$, so we have
 \begin{multline*}
   \braid_{s_0'}(\braid_{n-1} \circ \cdots \circ \braid_1(b_1 \otimes \cdots \otimes b_n)) \\
   = \braid_{n-1} \circ \cdots \circ \braid_1(\braid_{s_0''}(b_1 \otimes \cdots \otimes b_n)) 
   = \braid_{s_0}(b_1 \otimes \cdots \otimes b_n) \neq 0.
 \end{multline*}
 In particular, $\braid_{s_0''}(b_1 \otimes \cdots \otimes b_n) \neq 0$, so by the inductive hypothesis, $b_2 \otimes \cdots \otimes b_n$ is in the Cartan component of $\cB_2 \otimes \cdots \otimes \cB_n$.  But also $\braid_{n-1} \circ \cdots \circ \braid_1(b_1 \otimes \cdots \otimes b_n) \neq 0$, and repeated application of the hexagon relation \eqref{eq:crystal_hexagon1} shows that
 \[
   \braid_{n-1} \circ \braid_{n-2} \circ \cdots \circ \braid_1 = \braid_{\cB_1,(\cB_2 \otimes \cdots \otimes \cB_n)},
 \]
 so $b_1\otimes(b_2 \otimes \cdots \otimes b_n)$ is necessarily in the Cartan component of $\cB_1 \otimes (\cB_2 \otimes \cdots \otimes \cB_n)$.  
 
 This completes the proof for $\cB_i$ irreducible.  The general case where the $\cB_i$ are products of irreducibles follows readily from this by naturality of the braiding, which implies that we can restrict our attention to the Cartan components in each $\cB_i$.
\end{proof}

\subsection{Left and right ends}

We now make a definition that is going to play a major role in the rest of this work, namely the notion of right end of a crystal element.

\begin{definition}
Let $\lambda, \mu \in \dominant$ with $\lambda \geq \mu$.
Considering the crystal $\cB(\lambda)$, the \emph{left end} and \emph{right end} with respect to $\mu$ are the set-theoretic maps
\begin{align*}
&\leftend_\mu: \cB(\lambda) \to \cB(\mu), &
&\rightend_\mu: \cB(\lambda) \to \cB(\mu),
\end{align*}
defined as follows. There exist unique injective morphisms of crystals
\begin{align*}
&\cB(\lambda) \to \cB(\mu) \otimes \cB(\lambda - \mu), &
&\cB(\lambda) \to \cB(\lambda - \mu) \otimes \cB(\mu),
\end{align*}
and we define $\leftend_\mu(b)$ and $\rightend_\mu(b)$ to be the unique elements of $\cB(\mu)$ such that
\begin{align*}
&b \mapsto \leftend_\mu(b) \otimes b^\prime, &
&b \mapsto b^{\prime \prime} \otimes \rightend_\mu(b),
\end{align*}
for some $b',b'' \in \cB(\lambda - \mu)$.

More generally, let $\cB = \cB(\lambda_1) \otimes \cdots \otimes \cB(\lambda_n)$ be a product of irreducibles of highest weight $\lambda=\sum_i\lambda_i$.
If $b \in \cB$ is in the Cartan component we define $\leftend_\mu(b)$ and $\rightend_\mu(b)$ as above, under the identification of the Cartan component of $\cB$ with $\cB(\lambda)$. 
Otherwise $\leftend_\mu(b) = \rightend_\mu(b) = 0$.
\end{definition}

We will also consider the right ends with respect to a \emph{family} of dominant weights, and in particular for the family of fundamental weights $\fundweights=(\varpi_1, \cdots, \varpi_\rank)$.

\begin{definition}
Let $\Sset = (\mu_1, \cdots, \mu_n)$ be an $n$-tuple of dominant weights.
Let $\cB$ be a product of irreducible crystals of highest weight $\lambda$ with $\lambda \geq \mu_i$ for all $i$.
Then the \emph{$\Sset$-left end} and \emph{$\Sset$-right end} maps are the set-theoretic maps
\[
\leftend_\Sset, \rightend_\Sset: \cB \to 
\cB(\mu_1) \times \cdots \times \cB(\mu_n) \sqcup \{0\}
\]
defined by
\begin{align*}
&\leftend_\Sset(b) := \left( \leftend_{\mu_1}(b), \cdots, \leftend_{\mu_n}(b) \right), &
&\rightend_\Sset(b) := \left( \rightend_{\mu_1}(b), \cdots, \rightend_{\mu_n}(b) \right).
\end{align*}
\end{definition}

\begin{remark}
According to the usual conventions on crystal bases, the notations $\cB(\mu_1) \otimes \cdots \otimes \cB(\mu_n)$ and $\cB(\mu_1) \times \cdots \times \cB(\mu_n)$ denote the same object.
We use the former when we take into account the additional crystal structure, but the latter when regarded simply as a set.
This is the notation we use for left and right ends, since $\leftend_\Sset$ and $\rightend_\Sset$ are only set-theoretic maps.
\end{remark}

The rest of this section is devoted to the properties of the right end maps.  We focus on the right ends since these play the main role in our setup, but similar results can of course be obtained for the left end maps. 

\begin{proposition}
\label{prop:rightend-sigma}
Let $\lambda_1, \cdots, \lambda_n \in \dominant$ be dominant weights.  An element $b_1 \otimes \cdots \otimes b_n$ belongs to the Cartan component of $\cB(\lambda_1) \otimes \cdots \otimes \cB(\lambda_n)$ if and only if we have
\begin{equation}
\label{eq:rightend-sigma}
\braid_{n - 1} \circ \cdots \circ \braid_k (b_1 \otimes \cdots \otimes b_n) \neq 0
\end{equation}
for every $k\in\{1, \cdots, n - 1\}$.
In this case, for each $k \in \{1, \cdots, n\}$, the right end $\rightend_{\lambda_k}(b_1 \otimes \cdots \otimes b_n)$ is equal to the rightmost tensor factor of \eqref{eq:rightend-sigma}.
\end{proposition}

\begin{proof}
For $n=1$ the statement is trivial, so suppose it is true for tensor products of length $n-1$.  The necessity of \eqref{eq:rightend-sigma} is obvious from \cref{thm:longest_word}, so it remains to check sufficiency.   

By the inductive hypothesis, $b_2\otimes\cdots \otimes b_n$ is in the Cartan component of $\cB(\lambda_2)\otimes\cB(\lambda_n)$, so by \cref{thm:longest_word} we have $s_0'(b_2\otimes\cdots\otimes b_n) \neq0$ where $s_0'$ denotes the longest word of $S_{n-1}$.  But now the longest word $s_0$ of $S_n$ satisfies
\[
 s_0\circ(\id\otimes s_0')(b_1\otimes\cdots\otimes b_n) = \braid_{n - 1} \circ \cdots \circ \braid_1 (b_1\otimes\cdots\otimes b_n) \neq 0
\]
by \eqref{eq:rightend-sigma}.
Therefore, by \cref{thm:longest_word}, $(\id\otimes s_0')(b_1\otimes\cdots\otimes b_n)$ is in the Cartan component of $\cB(\lambda_1)\otimes\cdots\otimes\cB(\lambda_n)$, and so $b_1\otimes\cdots\otimes b_n$ as well.   This proves the first statement. 

Now the map from \eqref{eq:rightend-sigma} gives a morphism of crystals
\[
\braid_{n - 1} \circ \cdots \circ \braid_k: \cB(\lambda_1) \otimes \cdots \otimes \cB(\lambda_n) \to \cB(\lambda_1) \otimes \cdots \otimes \widehat{\cB(\lambda_k)} \otimes \cdots \otimes \cB(\lambda_n) \otimes \cB(\lambda_k)
\]
with the hat denoting omission, which gives the second statement.
\end{proof}

\begin{remark}
We briefly mention how this connects to the theory of set-theoretic solutions to the Yang-Baxter equations, as in \cite{EtiSchSol}.
Writing the Cartan braiding as $\braid(b \otimes b^\prime) = \cpa_b(b^\prime) \otimes \cpb_{b^\prime}(b)$, \cref{prop:rightend-sigma} can be recast in the form
\[
\rightend_{\lambda_k}(b_1\otimes\cdots\otimes b_n) = \cpb_{b_n} \circ \cdots \circ \cpb_{b_{k + 1}}(b_k), \quad
k = 1, \cdots, n.
\]
This gives the components of the map $J_n$ introduced in \cite[Proposition 1.7]{EtiSchSol} (if we work with left ends instead of right ends).
\end{remark}

Next, we discuss how right ends behave with respect to tensor products.

\begin{proposition}
\label{prop:right_end_independence}
Let $\lambda, \lambda', \mu \in \dominant$
and let $b\otimes b'$ be in the Cartan component of $\cB(\lambda)\otimes\cB(\lambda')$. 
\begin{enumerate}
    \item[1)] If $\lambda'\geq\mu$ then $\rightend_\mu(b\otimes b') = \rightend_\mu(b')$.
    \item[2)] If $\lambda\geq\mu$ then $\rightend_\mu(b\otimes b') = \rightend_\mu(\rightend_\mu(b)\otimes b')$.  Consequently, if $\lambda''\geq\mu$ and $c\in\cB(\lambda'')$ has the same $\mu$-right end as $b\in\cB(\lambda)$, then $\rightend_\mu(b\otimes b')=\rightend_\mu(c\otimes b')$.
\end{enumerate}
\end{proposition}

\begin{proof}
1) is almost immediate from the definitions.  For 2), consider the inclusion of the Cartan component $\inclusion: \cB(\lambda) \to \cB(\lambda-\mu)\otimes\cB(\mu)$.  We have
  \[
  \xymatrix{
   b\otimes b' \ar@{|->}[r]^-{\inclusion \otimes \id} &
   b''\otimes \rightend_\mu(b) \otimes b' \ar@{|->}[r]^-{\braid_2} &
   b'' \otimes b''' \otimes \rightend_\mu(\rightend_\mu(b)\otimes b'),
   }
  \]
  for some $b''\in\cB(\lambda-\mu)$ and $b'''\in\cB(\lambda')$. The result then follows from \cref{prop:rightend-sigma}. The final statement is now immediate.
\end{proof}

The next result shows that we obtain the same set of non-zero $\Sset$-right ends from any appropriate crystal.

\begin{proposition}
\label{prop:set_of_ends}
Let $\Sset = (\mu_1, \cdots, \mu_n)$ be a family of dominant weights.  Let $\cB$ be a product of irreducible crystals.  Then the set of (non-zero) $\Sset$-right ends of elements of $\cB$, namely
\begin{equation}
 \label{eq:right_end_set}
 \{ \rightend_\Sset(b) \mid b\in \cB\} \setminus \{0\} 
 \quad 
 \subseteq \prod_{i=1}^\NC\cB(\mu_i),
\end{equation}
is independent of the choice of $\cB$, as long as the highest weight $\lambda$ of $\cB$ satisfies $\lambda\geq\mu_i$ for every $i$.
\end{proposition}

\begin{proof}
  We may assume without loss of generality that $\cB=\cB(\lambda)$ is irredudible, with $\lambda\geq\mu_i$ for all $i$, since if $\cB$ has highest weight $\lambda$ then by definition the right ends of $\cB$ are the same as those of $\cB(\lambda)$.  
  
  Now suppose $\lambda'\geq\lambda$.  \cref{prop:right_end_independence}(1) shows that the $\mu_i$-right ends of $\cB(\lambda)$ are the same as the $\mu_i$-right ends of $\cB(\lambda'-\lambda)\otimes\cB(\lambda)$, which by definition are the same as the $\mu_i$-right ends of $\cB(\lambda')$.
  To complete the proof, just note that for any two $\lambda_1,\lambda_2$ there is $\lambda'$ which is greater than both of them.
\end{proof}


\section{The higher-rank graph algebra}
\label{sec:k-graph_algebra}

\subsection{Higher-rank graphs}
We begin by recalling the definition of a \linebreak $k$-graph from \cite{KumPas}.  For this, we need to regard the abelian monoid $\NN^k$ as a category with one object, in which composition is given by addition.

\begin{definition}
A \emph{higher-rank graph of rank $k$} (or $k$-graph) is a small category $\Lambda$ equipped with a functor $\dmap: \Lambda \to \NN^k$ satisfying the \emph{factorization property}: for every $e \in \Lambda$, and any pair of multi-indices $m, n \in \NN^k$ with $\dmap(e) = m + n$, there are unique elements $e_1, e_2 \in \Lambda$ with $\dmap(e_1) = m$, $\dmap(e_2) = n$ such that $e = e_1 \cdot e_2$.
\end{definition}

The elements of $\Lambda$ are called \emph{paths}, the indices $\{ 1, \cdots, k \}$ are referred to as \emph{colours}, and $\dmap(e)$ is called the \emph{degree} or \emph{coloured length} of $e$.  We write $\Lambda^n$ for the set of paths of length $n\in\NN^k$.  Paths of length $0$ are called \emph{vertices}, naturally enough.  Paths of length $\delta_i = (0, \cdots, 1, \cdots, 0)$, with $1$ in the $i$th slot, are called \emph{edges of colour $i$}. We use the notation $\Lambda^{\neq0}$ for the set of paths of non-zero length.

The factorization property implies that for every path $e\in\Lambda$, there are unique vertices $\rmap(e)$ and $\smap(e)$ in $\Lambda^0$, with the property that
\[
 e = \rmap(e) \cdot e \cdot \smap(e).
\]
This defines the \emph{range} and \emph{source} maps $\rmap,\smap: \Lambda \to \Lambda^0$.
Given a vertex $v \in \Lambda^0$ and a multi-index $n \in \NN^k$, we write
\begin{align*}
& v \Lambda^n := \{ e \in \Lambda^n: \rmap(e) = v \}, &
& \Lambda^n v := \{ e \in \Lambda^n: \smap(e) = v \}.
\end{align*}

It is common to impose the following additional hypotheses on a higher-rank graph, which will apply in our case, see \cref{sec:graph-properties}.

\begin{definition}
\label{def:row_finite_and_no_sources}
A $k$-graph $\Lambda$ is \emph{row-finite} if $v\Lambda^n$ is finite for every $v\in\Lambda^0$ and  $n\in \NN^k$.  It has \emph{no sources} if $v\Lambda^n\neq\emptyset$ for every $v\in\Lambda^0$ and $n \in \NN^k$, and \emph{no sinks} if $\Lambda^nv\neq\emptyset$ for every $v\in\Lambda^0$ and $n \in \NN^k$.
\end{definition}

\subsection{Higher-rank graph algebras}

Higher-rank graph algebras were first defined as $C^*$-algebras, by Kumjian and Pask \cite{KumPas}.  We will work initially with the algebraic version, which is due to Aranda Pino, Clark, an Huef and Raeburn \cite{Aranda:Kumjian-Pask}, although we will modify their definition slightly to take advantage of the fact that we only work over the field $\CC$ rather than an arbitrary commutative ring with unit.

\begin{definition}
\label{def:graph-algebra}
Let $\Lambda$ be a row-finite $k$-graph without sources.
The \linebreak \emph{Kumjian-Pask algebra} $\mathrm{KP}_\CC(\Lambda)$ is the $*$-algebra over $\CC$ generated by the elements $\{ p_v \}_{v \in \Lambda^0}$ and $\{ s_e \}_{e \in \Lambda^{\neq 0}}$ with the following relations:
\begin{enumerate}
\item[(KP1)] $\{ p_v \}_{v \in \Lambda^0}$ is a set of mutually orthogonal projections, that is $p_v^* = p_v = p_v^2$ and $p_v p_{v^\prime} = \delta_{v, v^\prime} p_v$,
\item[(KP2)] for all $e, e^\prime \in \Lambda^{\neq 0}$ with $\smap(e) = \rmap(e^\prime)$ we have
\[
s_e s_{e^\prime} = s_{e \cdot e^\prime}, \quad
p_{\rmap(e)} s_e = s_e = s_e p_{\smap(e)},
\]
\item[(KP3)] for all $e, e^\prime \in \Lambda^{\neq 0}$ with $\dmap(e) = \dmap(e^\prime)$ we have
\[
s_e^* s_{e'} = \delta_{e, e^\prime} p_{\smap(e)},
\]
\item[(KP4)] for all $v \in \Lambda^0$ and $n \in \NN^k \backslash \{0\}$ we have
\[
p_v = \sum_{e \in v \Lambda^n} s_e s_e^*.
\]
\end{enumerate}
\end{definition}

\begin{remark}
The definition in \cite{Aranda:Kumjian-Pask} is given in terms of paths $e \in \Lambda^{\neq 0}$ and ghost paths $e^* \in G(\Lambda^{\neq 0})$, corresponding to generators $s_e$ and $s_{e^*}$.
When working over $\CC$ there is a $*$-structure uniquely defined by $p_v^* = p_v$ and $s_e^* = s_{e^*}$, so we prefer to give the definition in this way.
\end{remark}

For the definition of the $C^*$-algebra of a higher-rank graph, see \cite{KumPas}.

\subsection{Higher-rank graphs associated to quantum groups}
\label{sec:crystal_N-graph}

Let $K$ be a connected, simply connected compact semisimple Lie group of rank $\rank$ with Lie algebra $\lie{k}$ and complexification $\lie{g}=\lie{k}_\CC$.
Our goal in this section is to define a higher-rank graph $\hgraphstd$ of rank $r$ associated to $\lie{g}$.    
In fact, we will make a more general construction, in anticipation of the results required for torus bundles over flag varieties, compare \cref{sec:flag_manifolds} and in particular \cref{thm:flag_generators2}.  

Consider any $N$-tuple $\Cset = (\vartheta_1, \cdots, \vartheta_\NC)$ of linearly independent dominant weights.
We refer to $\Cset$ as the set of \emph{colours} and we blur the distinction between $\Cset$ and the index set $\{1, \cdots, \NC\}$.
Then to any such choice of colours we will associate a higher-rank graph $\hgraph$ of rank $\NC$.
The basic case $\hgraphstd$ is obtained when $\Cset = \fundweights = (\varpi_1, \cdots, \varpi_r)$ is the $\rank$-tuple fundamental weights.

Write $\dominant_\Cset := \NN \cdot \Cset$ for the submonoid of $\dominant$ generated by the colours $\Cset$.  We identify $\dominant_\Cset$ with the monoid $\NN^\NC$ via the map
\begin{equation}
    \label{eq:monoid}
\NN^\NC \to \dominant_\Cset, \qquad (n_1, \cdots, n_\NC) \mapsto \sum_{i = 1}^\NC n_i \vartheta_i.
\end{equation}
In this way, the degree of a path in our higher rank graph will be given by a weight in $\dominant_\Cset$.
We also fix the dominant weight
\( \ds
\rho_\Cset := \sum_{i = 1}^\NC \vartheta_i \in \dominant_\Cset.
\)

\begin{definition}
\label{def:vertices_and_paths}
Let $\lie{g}$ be a complex semisimple Lie algebra and fix a set of colours $\Cset = (\vartheta_1, \cdots, \vartheta_\NC)$ as above.
We define the pair $(\Vset, \hgraph)$ as follows.
\begin{itemize}
\item The \emph{set of vertices} consists of the $\Cset$-right ends of $\cB(\rho_\Cset)$, that is
\[
\Vset := \left\{ \rightend_\Cset(b) \mid b \in \cB(\rho_\Cset) \right\} \subset \cB(\vartheta_1) \times \cdots \times \cB(\vartheta_\NC).
\]
\item The \emph{set of paths} $\hgraph$ is the set of pairs $(v, b)$, where $v = (c_1, \cdots, c_\NC) \in \Vset$ and $b \in \cB(\lambda)$ for some $\lambda \in \dominant_\Cset$ such that $c_i\otimes b$ is in the Cartan component of $\cB(\vartheta_i)\otimes\cB(\lambda)$ for every colour $i$.
\end{itemize}
We define the \emph{degree} of the path $e = (v, b)$ above to be $\dmap(e) = \lambda$, where we use the identification $\dominant_\Cset\cong\NN^\NC$ from \eqref{eq:monoid}.
\end{definition}

Some remarks are in order. Firstly, we can identify the vertex set $\Vset$ with the set $\hgraph^0$ of paths of degree $0$ via the map $v\mapsto(v, b_0)$, where $b_0 \in \cB(0)$ is the unique element of the trivial crystal.  
Secondly, note that in the definition of the vertex set $\Vset$, we could equally well use the set of $\Cset$-right ends of any crystal $\cB(\mu)$ with highest weight $\mu \geq \rho_\Cset$, thanks to \cref{prop:set_of_ends}.

\begin{lemma}
\label{lem:right_end_of_a_path}
Let $v \in \Vset$ and fix $c \in \cB(\mu)$
with $\rightend_\Cset(c) = (c_1,\cdots,c_\NC)=v$, for some $\mu \in\dominant_\Cset$ with $\mu \geq \rho_\Cset$.  Let also $b\in\cB(\lambda)$ for some $\lambda\in\dominant_\Cset$.  Then $(v,b)\in\hgraph$ if and only if $c\otimes b$ is in the Cartan component of $\cB(\mu)\otimes\cB(\lambda)$.  In this case, for every  $\vartheta_i\in\Cset$, we have $\rightend_{\vartheta_i}(c_i\otimes b) = \rightend_{\vartheta_i}(c\otimes b)$.
\end{lemma}

\begin{proof}
  First, suppose that $c\otimes b$ is in the Cartan component of $\cB(\mu) \otimes \cB(\lambda)$.  Then \cref{prop:right_end_independence}(2) shows that $\rightend_{\vartheta_i}(c_i \otimes b) = \rightend_{\vartheta_i}(c \otimes b) \neq 0$ for all $i$, so $(v,b) \in \hgraph$.

  Conversely, suppose that $(v,b) \in \hgraph$, so that $c_i \otimes b$ is in the Cartan component of $\cB(\vartheta_i)\otimes\cB(\lambda)$ for every $i$.
  We can write $\mu = \sum_{i=1}^\NC n_i \vartheta_i$ with $n_i > 0$ for every $i$, and we put $|\mu| = \sum_i n_i$.  We have an inclusion of crystals 
  \[
  \cB(\mu) \into \cB(\vartheta_1)^{\otimes n_1} \otimes \cdots \otimes \cB(\vartheta_\NC)^{\otimes n_{\NC}}
  \]
  as the Cartan component on the right-hand side.  Let us identify $c$ with its image $c\mapsto a_{1}\otimes\cdots\otimes a_{|\mu|}$ in the Cartan component.   We can then identify $c\otimes b$ with its image $a_1 \otimes\cdots\otimes a_{|\mu|} \otimes b$ in $\cB(\vartheta_1)^{\otimes n_1} \otimes \cdots \otimes \cB(\vartheta_\NC)^{\otimes n_{\NC}} \otimes \cB(\lambda)$.  For each $k \in \{1, \cdots, |\mu|\}$, if the $k$-th factor in the tensor product above is $\cB(\vartheta_i)$, then using \cref{prop:rightend-sigma} we have
  \begin{align*}
   \braid_{|\mu|} \circ & \cdots \circ \braid_k (a_1\otimes\cdots\otimes a_{|\mu|} \otimes b) \\
    &= \braid_{|\mu|}(a_1\otimes\cdots a_{k-1} \otimes a_k' \otimes \cdots \otimes a'_{|\mu|-1}\otimes c_i \otimes b) \\
    &= a_1\otimes\cdots a_{k-1} \otimes a_k' \otimes \cdots \otimes a'_{|\mu|-1}\otimes b' \otimes \rightend_{\vartheta_i}(c_i\otimes b) \neq 0
  \end{align*}
  for some $a'_k, \cdots, a'_{|\mu|-1}$ and some $b'$.  By \cref{prop:rightend-sigma} again, we deduce that $c\otimes b$ is in the Cartan component.  Moreover, $\rightend_{\vartheta_i}(c\otimes b) = \rightend_{\vartheta_i}(c_i\otimes b)$, which proves the final statement.
\end{proof}

In particular, \cref{lem:right_end_of_a_path} shows that $\rightend_\Cset(c\otimes b)$ depends only on the right end $v = \rightend_\Cset(c) \in \Vset$ of $c$, and not on the choice of the element $c$ representing it.  Therefore, the following definition makes sense.

\begin{definition}
\label{def:range_and_source}
We define the \emph{source} and \emph{range} maps $\smap, \rmap: \hgraph \to \Vset$ as follows. 
Let $e = (v, b)$ be a path with $v = (c_1, \cdots, c_\NC)$. Then we define
\[
\smap(e) := v, \qquad
\rmap(e) := \left( \rightend_{\vartheta_1}(c_1 \otimes b) , \cdots, \rightend_{\vartheta_\NC}(c_\NC \otimes b) \right).
\]
Equivalently, choosing $c \in \cB(\rho_\Cset)$ such that $\rightend_\Cset(c) = v$, we have $\rmap(e) = \rightend_\Cset(c \otimes b)$.
\end{definition}

As before, in the equality $\rmap(e) = \rightend_\Cset(c \otimes b)$ we can replace $c \in \cB(\rho_\Cset)$ with any crystal element $c \in \cB(\mu)$, where $\mu \in \dominant_\Cset$ and $\mu \geq \rho_\Cset$, such that $\rightend_\Cset(c) = v$.

\begin{theorem}
\label{thm:hgraph}
The set $\hgraph$ becomes a higher-rank graph of rank $\NC$ with:
\begin{itemize}
\item degree map $\dmap$ as in \cref{def:vertices_and_paths},
\item source and range maps as in \cref{def:range_and_source},
\item composition of paths $(v, b), (v', b') \in \hgraph$ with $\rmap(v, b) = \smap(v', b')$ defined by
\begin{equation}
    \label{eq:composition}
    (v',b') \cdot (v,b) = (v, \projection(b \otimes b')),
\end{equation}
where $\projection: \cB(\lambda) \otimes \cB(\lambda') \to \cB(\lambda + \lambda')$ denotes the projection onto the Cartan component.
\end{itemize}
\end{theorem}

\begin{proof}
First we need to show that the composition law is well-defined.  Let $(v,b)$ and $(v',b')\in\hgraph$ with $\rmap(v,b) = \smap(v',b')$.  By \cref{lem:right_end_of_a_path}, this means that if we fix any $c \in \cB(\rho_\Cset)$ with $\rightend_\Cset(c) = v$, then $c\otimes b$ is in the Cartan component of $\cB(\rho_\Cset) \otimes \cB(\lambda)$ and $\rightend_\Cset(c \otimes b) = v'$. Again using \cref{lem:right_end_of_a_path}, we get that $c \otimes b \otimes b'$ is in the Cartan component of $\cB(\rho_\Cset) \otimes \cB(\lambda) \otimes \cB(\lambda')$, and hence $c \otimes \projection(b \otimes b')$ is in the Cartan component of $\cB(\rho_\Cset) \otimes \cB(\lambda + \lambda')$.  This proves that $(v, \projection(b\otimes b')) \in \hgraph$.  

Associativity follows from the associativity of the tensor product of crystals.  The identity arrows are the paths of length $0$, see the remarks immediately after \cref{def:vertices_and_paths}.  Thus $\hgraph$ is a (small) category.

If $\dmap(v, b) = \lambda$ and $\dmap(v',b') = \lambda'$ then clearly $\dmap(v,\projection(b\otimes b')) = \lambda + \lambda'$.  Conversely, if $\dmap(v,b'') = \lambda + \lambda'$, then $b''\in\cB(\lambda\otimes\lambda'')$ and there is a unique element $b\otimes b'$ in the Cartan component of $\cB(\lambda)\otimes\cB(\lambda')$ with $\projection(b\otimes b')=b''$.  Fix any $c\in\cB(\rho_\Cset)$ with $\rightend_\Cset(c)=v$.  Since $(v,b'')\in\hgraph$, we have that $c\otimes b \otimes b'$ is in the Cartan component of $\cB(\rho_\Cset)\otimes\cB(\lambda)\otimes\cB(\lambda')$, hence both $(\rightend_\Cset(c),b)=(v,b)$ and $(\rightend_\Cset(c\otimes b),b')$ are in $\hgraph$.  If we put $v'=\rightend_\Cset(c\otimes b)$ then $\rmap(v,b) = v'$, and we get a factorization $(v,b'') = (v',b') \cdot (v,b)$.  This factorization is unique, by the uniqueness of $b$ and $b'$.
\end{proof}

We describe this structure in some detail in the case of our running example.

\begin{example}
\label{ex:graph-su3}
Consider $\lie{g} = \lie{sl}_3$ with the colours $\Cset = (\varpi_1, \varpi_2)$.
We want to determine the $2$-graph corresponding to this case.

We can identify the crystal $\cB(\rho) = \cB(\varpi_1 + \varpi_2)$ with the Cartan component of the product $\cB(\varpi_1) \otimes \cB(\varpi_2)$, which consists of $8$ elements.
Then using the Cartan braiding \eqref{eq:sl3-braiding} from \cref{ex:braiding-sl3}, \cref{prop:rightend-sigma} gives
\[
\begin{array}{lll}
\rightend_\Cset(a_1 \otimes b_1) = (a_1, b_1), \quad &
\rightend_\Cset(a_2 \otimes b_1) = (a_2, b_1), \quad &
\rightend_\Cset(a_3 \otimes b_1) = (a_2, b_1), \\
\rightend_\Cset(a_1 \otimes b_2) = (a_1, b_2), &
\rightend_\Cset(a_2 \otimes b_2) = (a_1, b_2), &
\rightend_\Cset(a_3 \otimes b_2) = (a_3, b_2), \\
&
\rightend_\Cset(a_2 \otimes b_3) = (a_2, b_3), &
\rightend_\Cset(a_3 \otimes b_3) = (a_3, b_3).
\end{array}
\]
From these $8$ elements we obtain $6$ distinct right ends, namely
\begin{align*}
v_1 & = (a_1, b_1), &
v_2 & = (a_1, b_2), & 
v_3 & = (a_2, b_1), \\
v_4 & = (a_2, b_3), &
v_5 & = (a_3, b_2), & 
v_6 & = (a_3, b_3).
\end{align*}

Next, we determine the edges (paths of length $1$) of colour $\varpi_1$.
Recall that these consists of pairs $((a_i, b_j), a_k)$ such that $a_i \otimes a_k$ and $b_j \otimes a_k$ are in their respective Cartan components. Using the graphs from \cref{ex:braiding-sl3} we find
\begin{align*}
e_1 & = (v_1, a_1), &
e_2 & = (v_2, a_1), &
e_3 & = (v_3, a_1), &
e_4 & = (v_3, a_2), \\
e_5 & = (v_4, a_1), &
e_6 & = (v_4, a_2), &
e_7 & = (v_5, a_1), &
e_8 & = (v_5, a_2), \\
e_9 & = (v_5, a_3), &
e_{10} & = (v_6, a_1), &
e_{11} & = (v_6, a_2), &
e_{12} & = (v_6, a_3).
\end{align*}
The sources and ranges $(\smap(e_i),\rmap(e_i))$ of these edges are given by
\begin{align*}
e_1 & : (v_1, v_1), &
e_2 & : (v_2, v_2), &
e_3 & : (v_3, v_1), &
e_4 & : (v_3, v_3), \\
e_5 & : (v_4, v_2), &
e_6 & : (v_4, v_4), &
e_7 & : (v_5, v_2), &
e_8 & : (v_5, v_3), \\
e_9 & : (v_5, v_5), &
e_{10} & : (v_6, v_2), &
e_{11} & : (v_6, v_4), &
e_{12} & : (v_6, v_6).
\end{align*}
This gives the following portion of the graph, which we depict in red
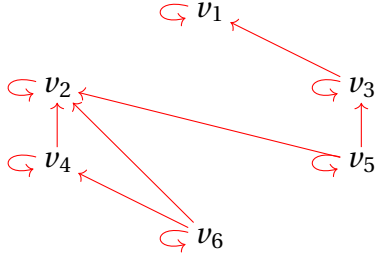
\begin{center}
\begin{tikzpicture}[
vertex/.style = {align=center, inner sep=2pt},
Rarr/.style = {->, red},
Barr/.style = {->, blue, dotted},
shadow/.style = {white, line width=3pt},
Rloop/.style = {->, red, out=165, in=195, loop},
Bloop/.style = {->, blue, out=15, in=-15, loop, dotted}
]
\node (v1) at ( 0, 0) [vertex] {$v_1$};
\node (v2) at (-2,-1) [vertex] {$v_2$};
\node (v3) at ( 2,-1) [vertex] {$v_3$};
\node (v4) at (-2,-2) [vertex] {$v_4$};
\node (v5) at ( 2,-2) [vertex] {$v_5$};
\node (v6) at ( 0,-3) [vertex] {$v_6$};

\draw [Rloop] (v1) edge (v1);
\draw [Rloop] (v2) edge (v2);
\draw [Rarr]  (v3) edge (v1);
\draw [Rloop] (v3) edge (v3);
\draw [Rarr]  (v4) edge (v2);
\draw [Rloop] (v4) edge (v4);
\draw [Rarr]  (v5) edge (v2);
\draw [Rarr]  (v5) edge (v3);
\draw [Rloop] (v5) edge (v5);
\draw [Rarr]  (v6) edge (v2);
\draw [Rarr]  (v6) edge (v4);
\draw [Rloop] (v6) edge (v6);
\end{tikzpicture}
\end{center}
Similar computations can be made for the colour $\varpi_2$, or alternatively we can employ the obvious symmetry coming from the Dynkin diagram automorphism. This leads to the $2$-graph presented in the introduction.
\end{example}

\subsection{Properties of the higher-rank graph}
\label{sec:graph-properties}

We now prove that our graphs satisfy the properties from \cref{def:row_finite_and_no_sources}.

\begin{proposition}
\label{prop:nice_graph}
The $\NC$-graph $\hgraph$ is row-finite and has no sources and sinks.
\end{proposition}

\begin{proof}
Row-finiteness is clear, because $v\hgraph^\lambda \subseteq \{v\}\otimes\cB(\lambda)$, which is finite.

Now let $v\in\Vset$ and $\lambda\in\dominant_\Cset$.  Pick $c\in\cB(\rho_\Cset)$ with $\rightend_\Cset(c)=v$.  Let $b_\lambda$ and $b_{w_0\lambda}$ denote the highest and lowest weight elements of $\cB(\lambda)$, respectively. By \cref{lem:tensor_fact}, $c\otimes b_\lambda$ belongs to the Cartan component of $\cB(\rho_\Cset)\otimes\cB(\lambda)$ and so $(v,b_\lambda)\in \hgraph^\lambda v$.  

Likewise, $b_{w_0\lambda}\otimes c$ is in the Cartan component of $\cB(\lambda)\otimes\cB(\rho_\Cset)$.  Let $c'\otimes b = \braid_{\cB(\lambda),\cB(\rho_\Cset)}(b_{w_0\lambda}\otimes c)$.  Then $\rightend_\Cset(c'\otimes b) = \rightend_\Cset(b_{w_0\lambda}\otimes c) = \rightend_\Cset(c) = v$ by \cref{prop:right_end_independence}, so putting $v'=\rightend_\Cset(c')$ we have $(v',b)\in v\hgraph^\lambda$.
\end{proof}

The higher rank graph $\hgraph$ is compatible with a partial ordering on the vertices, as we now describe. 
Firstly, every irreducible crystal $\cB(\lambda)$ admits a partial ordering by declaring that $b\leq b'$ if and only if $b=\kasF b'$ where $\kasF = \kasF_{i_1} \cdots \kasF_{i_m}$ is some product of the Kashiwara operators $\kasF_i$.  Then for $N$-tuples $v = ( b_1, \cdots, b_\NC )$ and $v^\prime = (b_1^\prime, \cdots, b_\NC^\prime)$ in $\Vset$, we write $v\leq v'$ if and only if $b_i\leq b_i'$ for every $i$.

\begin{proposition}
Consider the partial order $\leq$ on $\Vset$ defined above. Then:
\begin{enumerate}
\item[1)] for any $e \in \hgraph$ we have $\smap(e) \leq \rmap(e)$,
\item[2)] there are unique maximal and minimal elements in $\Vset$, namely
\begin{align*}
   & v_{\mathrm{max}} =(b_{\vartheta_1}, \cdots, b_{\vartheta_\NC}), && v_{\mathrm{min}} = (b_{w_0\vartheta_1}, \cdots, b_{w_0\vartheta_\NC}),
\end{align*}
where $b_\lambda$ and $b_{w_0\lambda}$ denote the highest and lowest weight elements of $\cB(\lambda)$, respectively. 
\end{enumerate}
\end{proposition}

\begin{proof}
1) Consider a path $e = (v, b)$ with $b \in \cB(\lambda)$.
Write $v = ( b_1, \cdots, b_\NC )$ for $\smap(e)$ and $v'=(b_1', \cdots, b_\NC')$ for $\rmap(e)$.  Fixing $i$, we have $b'_i = \rightend_{\vartheta_i}(b_i\otimes b)$, so by Proposition \ref{prop:rightend-sigma} we have $\sigma_{\cB(\vartheta_i),\cB(\lambda)} : b_i\otimes b \mapsto c \otimes b'_i$ for some $c\in\cB(\lambda)$.  Considering the structure of the braiding operators $\Rhat_{\cB(\vartheta_i),\cB(\lambda)}$ fixed in \eqref{eq:R-matrix_convention}, we obtain that $b_i\leq b_i'$. The result follows.

2) Consider any $\lambda, \mu \in \dominant$ with $\lambda \geq \mu$.
The inclusion $\cB(\lambda) \into \cB(\lambda - \mu) \otimes \cB(\mu)$ maps $b_\lambda$ to $b_{\lambda-\mu}\otimes b_\mu$, from which we see that $\rightend_\mu(b_\lambda)=b_\mu$.  It follows that $\rightend_\Cset(b_{\rho_\Cset}) = (b_{\vartheta_1}, \cdots, b_{\vartheta_\NC}) \in \Vset$, and this element is clearly maximal.  The result for $v_{\mathrm{min}}$ is obtained in a similar way.
\end{proof}


\section{The crystal algebra}
\label{sec:crystal-algebra}

\subsection{Definition of the algebra}

In this section we introduce an abstract $*$-algebra $\CalgKstd$ which is universal for the generators and relations for the crystal limit $\OKO$ given in \cref{thm:pi0_relations}, so that we get a surjective $*$-homomorphism $\CalgKstd \to \OKO$.
It provides a convenient bridge between $\OKO$ and the higher-rank graph algebra $\KPalgstd$.
Ultimately, we are going to prove that these three algebras are all $*$-isomorphic.

As in \cref{sec:crystal_N-graph}, we will anticipate the case of flag varieties by fixing a linearly independent family of dominant weights $\Cset = (\vartheta_1, \cdots, \vartheta_\NC)$ and writing $\dominant_\Cset = \NN \cdot \Cset$ for the monoid generated by the colours.

\begin{definition}
\label{def:crystal_algebra}
We define the \emph{crystal algebra} $\CalgG$ (with respect to the Lie algebra $\lie{g}$ and the colours $\Cset$) to be the unital $\CC$-algebra generated by elements $\{\genf_b, \genv_b \mid \lambda\in\dominant_\Cset, b \in \cB(\lambda)\}$ with the following relations:
\begin{enumerate}
\item for any $\lambda, \lambda' \in \dominant_\Cset$ and $b \in \cB(\lambda)$, $b' \in \cB(\lambda')$, 
\begin{align}
\label{eq:multiplication-rules}
&\genf_b \genf_{b^\prime} = \Cartan(b \otimes b^\prime) \genf_{b^{\prime \prime}}, &
&\genv_{b^\prime} \genv_b = \Cartan(b \otimes b^\prime) \genv_{b^{\prime \prime}},
\end{align}
where we write $b \otimes b^\prime \mapsto b^{\prime \prime}$ under the unique surjective crystal morphism $\cB(\lambda) \otimes \cB(\lambda^\prime) \to \cB(\lambda + \lambda^\prime)$;
\item for any $\lambda, \lambda' \in \dominant_\Cset$ and $b \in \cB(\lambda)$, $b' \in \cB(\lambda')$, 
\begin{equation}
\label{eq:cross-relations}
\genf_b \genv_{b^\prime} = \sum_{(c, c')} \genv_{c^\prime} \genf_c,
\end{equation}
where the sum is over all pairs $(c, c') \in \cB(\lambda) \times \cB(\lambda')$ such that the condition $\braid(c \otimes b^\prime) = c^{\prime} \otimes b$ holds;
\item for any $\lambda \in \dominant_\Cset$ we have
\begin{equation}
\label{eq:unitality}
\sum_{b \in \cB(\lambda)} \genv_b \genf_b = 1.
\end{equation}
\end{enumerate}
For the fundamental weights $\Cset = \fundweights$ we simply write $\CalgGstd$ instead of $\CalgGpi$.
\end{definition}

\begin{remark}
Note that, by the first relation, we could reduce the set of generators to those $\genf_b$, $\genv_b$ with $b\in\cB(\vartheta_i)$ for $\vartheta_i\in\Cset$.
Hence the algebra is finitely generated. 
\end{remark}

As in Remark \ref{rmk:exchange_relations}, the relations \eqref{eq:multiplication-rules} imply the exchange relations
\begin{align}
\label{eq:exchange_relations}
& \genf_b \genf_{b'} = \genf_{c'} \genf_c, &
& \genv_{b'} \genv_b = \genv_c \genv_{c'},
\end{align}
whenever $b\otimes b'$ is in the Cartan component of $\cB(\lambda) \otimes \cB(\lambda')$ and $\braid(b \otimes b') = c'\otimes c$.  On the other hand, $\genf_b \genf_{b'} = \genv_{b'} \genv_b = 0$ if $b\otimes b'$ is not in the Cartan component.

In the case $\lambda = \lambda'$, the relations \eqref{eq:exchange_relations} and \eqref{eq:cross-relations} simplify further: for any elements $b, b' \in \cB(\lambda)$ we have
\begin{equation}
\label{eq:lambda-lambda_relations}
\begin{gathered}
\genf_b \genf_{b^\prime} = \Cartan(b \otimes b^\prime) \genf_b \genf_{b^\prime}, \qquad
\genv_{b^\prime} \genv_b = \Cartan(b \otimes b^\prime) \genv_{b^\prime} \genv_b, \\
\genf_b \genv_{b^\prime} = \delta_{b, b^\prime} \sum_{c \in \cB(\lambda)} \Cartan(c \otimes b) \genv_c \genf_c.
\end{gathered}
\end{equation}

\begin{proposition}
We have a $*$-structure on $\CalgG$ given by $\genf^*_b = \genv_b$.
\end{proposition}

\begin{proof}
The only relation of $\CalgG$ which needs some checking is \eqref{eq:cross-relations}.
By the properties of the Cartan braiding $\braid$, the relation $\braid(b \otimes b') = c' \otimes c$ is equivalent to $\braid(c' \otimes c) = b \otimes b'$. Then we have 
\[
\left(\sum_{(c, c')} \genv_{c'} \genf_{c}\right)^* = \sum_{(c, c')} \genv_c \genf_{c'},
\]
where on both sides the sum is over pairs $(c,c')$ as in \eqref{eq:cross-relations}.  The result follows.
\end{proof}

We write $\CalgK$ for the crystal algebra $\CalgG$ equipped with this $*$-structure.
When $\Cset = \fundweights$ we simply write $\CalgKstd$.

The following statement now follows immediately from \cref{thm:pi0_relations}.

\begin{proposition}
\label{prop:AC_universal_map}
There is a unique surjective morphism of $*$-algebras $\CalgKstd \to \OKO$ defined on generators as follows: for any $\lambda \in \dominant$ and any $b_i \in \cB(\lambda)$ we have
\begin{equation*}
    \genf_{b_i} \mapsto \SoibT_0(\genf_i^\lambda), \quad
    \genv_{b_i} \mapsto \SoibT_0(\genv_i^\lambda),
\end{equation*}
where $\genf_i^\lambda$ and $\genv_i^\lambda$ are as in Equation \eqref{eq:genf_genv}.

With notation as in \cref{sec:flag_manifolds},
if $K$ is a compact connected semisimple Lie group, not necessarily simply connected, with complexified Lie algebra $\lie{k}_\CC = \lie{g}$, $\rootset \subset \simpleroots$ is a set of simple roots, and $Y_\rootset = K/K^0_\rootset$ is the associated torus bundle over the flag manifold $X_\rootset = K / K_\rootset$, then the above morphism restricts to a surjective $*$-morphism $\CalgK \to \cO[Y_{\rootset,0}]$ where $\Cset = (\vartheta_1, \cdots, \vartheta_\NC)$ is the family of linearly independent dominant weights which generate $\dominant_{K,S}$.
\end{proposition}

The following notation is useful.

\begin{definition}
\label{def:f_f_b}
Given $b = b_1 \otimes \cdots \otimes b_n \in \cB(\lambda_1) \otimes \cdots \otimes \cB(\lambda_n)$, we write
\[
\genf_b := \genf_{b_1} \cdots \genf_{b_n}, \quad
\genv_b := \genv_{b_n} \cdots \genv_{b_1}.
\]
Note that $\genf_b^* = \genv_b$. We also adopt the useful convention $\genf_0 = \genv_0 := 0$.
\end{definition}

\begin{proposition}
\label{prop:commutation-relations}
Let $b \in \cB(\lambda_1) \otimes \cdots \otimes \cB(\lambda_n)$, and let $\projection:\cB(\lambda_1)\otimes\cdots\otimes\cB(\lambda_n) \to \cB(\lambda_1+\cdots+\lambda_n)$ denote the projection onto the Cartan component. Then
\[
\genf_b = \genf_{\projection(b)}, \quad
\genv_b = \genv_{\projection(b)}.
\]
In particular $\genf_b = \genv_b = 0$ if $b$ is not in the Cartan component.
\end{proposition}

\begin{proof}
This follows by an inductive argument using the relations \eqref{eq:multiplication-rules}.
\end{proof}

\begin{remark}
\cref{prop:commutation-relations} implies the following generalization of \eqref{eq:exchange_relations}: with $b$ as above, we have $\genf_b = \genf_{\braid_s(b)}$ and $\genv_b = \genv_{\braid_s(b)}$ for every $s\in S_n$.
\end{remark}


\subsection{Projections associated to crystal elements}

Fix $\lambda\in\dominant_\Cset$.  For any $b \in \cB(\lambda)$ we define
\[P_b := \genv_b \genf_b \in \CalgK.
\]
As we now show, for each fixed $\lambda \in \dominant_\Cset$,  $\{P_b\}_{b \in \cB(\lambda)}$ is a set of mutually orthogonal projections summing to $1$.

\begin{proposition}
\label{prop:projections-sameweight}
For any $b, b^\prime \in \cB(\lambda)$ we have
\[
P_b^* = P_b, \quad
P_b P_{b^\prime} = \delta_{b, b^\prime} P_b, \quad
\sum_{b \in \cB(\lambda)} P_b = 1.
\]
\end{proposition}

\begin{proof}
The first relation follows from the $*$-structure $\genf^*_b = \genv_b$ and the third relation is just the unitality relation \eqref{eq:unitality}. 
For the second we use the relations \eqref{eq:lambda-lambda_relations}.
They immediately imply $P_b P_b^\prime = 0$ for $b \neq b^\prime$, while for $b = b^\prime$ we compute
\[
\begin{split}
P_b^2 & = \genv_b \genf_b \genv_b \genf_b
= \sum_{c \in \cB(\lambda)} \Cartan(c \otimes b) \genv_b \genv_c \genf_c \genf_b \\
& = \sum_{c \in \cB(\lambda)} \genv_b \genv_c \genf_c \genf_b = \genv_b \genf_b = P_b.
\qedhere
\end{split}
\]
\end{proof}

Now consider arbitrary projections $P_b$ and $P_{b^\prime}$ with $b \in \cB(\lambda)$, $b^\prime \in \cB(\lambda^\prime)$. Our next goal is to show that they commute even when $\lambda \neq \lambda^\prime$.
First we need the following result, which will also be useful elsewhere.

\begin{lemma}
\label{lem:relationPv}
Let $b \in \cB(\lambda)$ and $b^\prime \in \cB(\lambda^\prime)$. Then we have
\[
P_b \genv_{b^\prime} = \sum_{\substack{c \in \cB(\lambda) \\ \rightend_\lambda(c \otimes b^\prime) = b}} \genv_{b^\prime} P_c.
\]
\end{lemma}

\begin{proof}
Using the cross-relations \eqref{eq:cross-relations} we have
\[
P_b \genv_{b^\prime} = \genv_b \genf_b \genv_{b^\prime} = \sum_{(c, c')} \genv_b \genv_{c^\prime} \genf_c,
\]
where the sum is over all pairs $(c, c') \in \cB(\lambda) \times \cB(\lambda')$ such that $\braid(c \otimes b^\prime) = c^\prime \otimes b$.
For such pairs we have the relation $\genv_b \genv_{c^\prime} = \genv_{b^\prime} \genv_c$. This gives
\[
P_b \genv_{b^\prime} = \sum_c \genv_{b^\prime} \genv_c \genf_c = \sum_c \genv_{b^\prime} P_c.
\]
Finally note that the condition $\braid(c \otimes b^\prime) = c^\prime \otimes b$ is equivalent to $\rightend_\lambda(c \otimes b^\prime) = b$.
\end{proof}

Now we can show that the various projections commute.

\begin{proposition}
\label{prop:projections-commute}
Let $b \in \cB(\lambda)$ and $b^\prime \in \cB(\lambda^\prime)$. Then $P_b P_{b^\prime} = P_{b^\prime} P_b$.
\end{proposition}

\begin{proof}
We need to show that $P_b P_{b^\prime} = P_b \genv_{b^\prime} \genf_{b^\prime}$ and $P_{b^\prime} P_b = \genv_{b^\prime} \genf_{b^\prime} P_b$ are equal.
Using the relation from \cref{lem:relationPv} and its adjoint we have
\[
P_b \genv_{b^\prime} \genf_{b^\prime} = \sum_{\substack{c \in \cB(\lambda) \\ \rightend_\lambda(c \otimes b^\prime) = b}}  \genv_{b^\prime} P_c \genf_{b^\prime} = \genv_{b^\prime} \genf_{b^\prime} P_b. \qedhere
\]
\end{proof}


\subsection{Projections associated to sets of crystal elements}

Knowing that the projections $P_b$ all commute, we can associate projections to $n$-tuples of crystal elements, as in the following definition.

\begin{definition}
\label{def:P_B}
Let $\Sset = (\mu_1, \cdots, \mu_n)$ be any family of dominant weights and let $B = (b_1, \cdots, b_n)$ be an $n$-tuple with $b_i \in \cB(\mu_i)$ for each $i$. We define the projection
\[
P_B := P_{b_1} \cdots P_{b_n}.
\]
\end{definition}

We will be particularly interested in the case where $\Sset = \Cset = (\vartheta_1, \cdots, \vartheta_\NC)$ is a linearly independent family of dominant weights as in \cref{sec:flag_manifolds}, and $B=v\in\Vset$ is a vertex of the higher-rank graph of \cref{sec:crystal_N-graph}.

The next result generalizes \cref{lem:relationPv}.

\begin{lemma}
\label{lem:commutation-projection-v}
Let $\Sset$ be as above and let $B = (b_1, \cdots, b_n)$ with $b_i \in \cB(\mu_i)$.  Fix any $\lambda \in \dominant$.  
Then for any $b \in \cB(\lambda)$ we have
\[
P_B \genv_{b} = \sum_{B'} \genv_{b} P_{B^\prime},
\]
where the sum is over all families $B' = ( b_1^\prime, \cdots, b_n^\prime )$ with $b_i^\prime \in \cB(\mu_i)$ satisfying
\[
 \rightend_{\mu_i}(b_i'\otimes b) = b_i
\]
for every $i = 1, \cdots, n$.
\end{lemma}

\begin{proof}
We have $P_B = P_{b_1} \cdots P_{b_n}$ and using \cref{lem:relationPv} we compute
\[
P_B \genv_{b} = \sum_{b_1^\prime: \rightend_{\mu_1}(b_1^\prime \otimes b) = b_1} \cdots \sum_{b_n^\prime: \rightend_{\mu_n}(b_n^\prime \otimes b) = b_n} \genv_{b} P_{b_1^\prime} \cdots P_{b_n^\prime}.
\]
This gives the result.
\end{proof}

We now show how to rewrite the projections $P_B$ in a normal form in terms of the generators $\genv_b$ and $\genf_b$.

\begin{proposition}
\label{prop:P_B_sum}
Let $\Sset = (\mu_1, \cdots, \mu_n)$ and write $\cB = \cB(\mu_1) \otimes \cdots \otimes \cB(\mu_n)$.  For any  $B = (b_1, \cdots, b_n)$ with $b_i\in\cB(\mu_i)$ for each $i$, we have
\[
  P_B = \sum_{\substack{b\in\cB \\ \rightend_\Sset(b) = B}} \genv_b \genf_b.
\]
\end{proposition}

\begin{proof}
We work by induction on the size $n$ of $\Sset$.  For $n=1$, the statement degenerates to the definition of $P_{b_1}$. 

Now suppose it holds for all sets with $n - 1$ elements.
Write \linebreak $\Sset_* = (\mu_1, \cdots, \mu_{n - 1} )$, $\cB_* = \cB(\mu_1) \otimes \cdots \otimes \cB(\mu_{n-1})$ and $B_* = ( b_1, \cdots, b_{n - 1} )$.
Using this notation and \cref{lem:commutation-projection-v} we obtain
\[
P_B = P_{B_*} \genv_{b_n} \genf_{b_n} = \sum_{B'_*} \genv_{b_n} P_{B'_*} \genf_{b_n},
\]
where the sum is over all $B'_* = (b'_1,\cdots,b'_{n-1})$ with $b'_i\in\cB(\mu_i)$ satisfying $\rightend_{\mu_i}(b_i'\otimes b_n) = b_i$ for all $i=1,\cdots,n-1$.  
By the inductive assumption we know that 
\[
P_{B'_*} = \sum_{\substack{b'\in\cB_* \\ \rightend_{\Sset_*}(b')=B'_*}}  \genv_{b'} \genf_{b'},
\]
and hence 
\[
P_B = \sum_{b'} \genv_{b'\otimes b_n}  \genf_{b'\otimes b_n},
\]
where now the sum is over all $b' \in \cB_*$ such that $\rightend_{\mu_i}(\rightend_{\mu_i}(b') \otimes b_n) = b_i$ for all $i=1,\cdots, n-1$.  But by \cref{prop:right_end_independence} this is equivalent to $\rightend_\Sset(b'\otimes b_n) = B$, so we are done.
\end{proof}

\begin{corollary}
\label{cor:projection-rightend}
Let $\Cset = (\vartheta_1,\cdots,\vartheta_\NC)$ be a family of linearly independent dominant weights.  Fix $\lambda\in\dominant_\Cset$ with $\lambda\geq\rho_\Cset$. Then for every $v = (b_1,\cdots,b_\NC)$ with $b_i\in\cB(\vartheta_i)$, we have
\[
  P_v = \sum_{\substack{b\in \cB(\lambda) \\ \rightend_\Cset(b)=v}} \genv_b \genf_b.
\]
In particular, we have $P_v = 0$  if $v\notin\Vset$.
\end{corollary}

\begin{proof}
  Fix $v = (b_1, \cdots, b_\NC)$ as in the statement.  Write $\lambda = \sum_i n_i \vartheta_i$ and put $\Sset = (\vartheta_{i_1}, \cdots, \vartheta_{i_{|\lambda|}})$, where each $\vartheta_i \in \Cset$ appears with multiplicity $n_i$.  Put also $\cB = \cB(\vartheta_{i_1}) \otimes \cdots \otimes \cB(\vartheta_{i_{|\lambda|}})$ and $B=(b_{i_1},\cdots,b_{i_{|\lambda|}})$ with the same multiplicities.  

  Applying \cref{prop:P_B_sum}, and using the fact that $(P_{b_i})^{n_i} = P_{b_i}$, we get
  \[
   P_v = P_B = \sum_{\substack{\tilde{b}\in\cB \\ \rightend_\Sset(\tilde{b})=B}} \genv_{\tilde{b}} \genf_{\tilde{b}}.
  \]
  But $\rightend_\Sset(\tilde{b}) \neq 0$ if and only if $\tilde{b}$ is in the Cartan component, in which case its image $b$ under the projection to $\cB(\lambda)$ satisfies $\rightend_\Sset(b) = B$ if and only if $\rightend_\Cset(b) = v$.  The result follows.
\end{proof}

In \cref{cor:projections_nonzero} we will show that all $P_v$ with $v\in\Vset$ are non-zero.


\subsection{Graph algebra relations}

Let $\Cset = (\vartheta_1, \cdots, \vartheta_\NC)$ be a family of linearly independent dominant weights and consider the corresponding higher-rank graph $\hgraph$, as in \cref{def:vertices_and_paths}.
For any vertex $v = (b_1, \cdots, b_\NC) \in \Vset$ with $b_i \in \cB(\vartheta_i)$, we have defined the projection $P_v = P_{b_1} \cdots P_{b_\NC}$.
Next, given a path $e \in \hgraph$ of the form $e = (v, b)$, we define
\[
\edge_e := \genv_b P_v, \quad
\edge_e^* := P_v \genf_b.
\]

Our goal is to show that the elements $\{P_v, \edge_e\}$ satisfy the relations of a Kumjian-Pask algebra, as in \cref{def:graph-algebra}.

\begin{remark}
\label{rmk:not_a_path}
If $(v,b) \in \Vset \times \cB(\lambda)$ does not define an element of $\hgraph$, then we get $v_b P_v = 0$.  To see this, note that if $v = (c_1, \cdots, c_\NC)$ then by hypothesis $c_i \otimes b$ is not in the Cartan component of $\cB(\lambda) \otimes \cB(\vartheta_i)$ for some $i$.  Therefore $\genv_b P_{c_i} = \genv_b \genv_{c_i} \genf_{c_i} = 0$, and the claim follows.
\end{remark}

We begin with the condition (KP1) concerning the elements $P_v$.

\begin{lemma}
The set $\{P_v\}_{\Vset}$ consists of mutually orthogonal projections.
Moreover we have $\sum_{v \in \Vset} P_v = 1$.
\end{lemma}

\begin{proof}
Given a vertex $v = (b_1, \cdots, b_\NC) \in \Vset$, we have the corresponding projection $P_v = P_{b_1} \cdots P_{b_\NC}$.
It easily follows from commutativity of the projections and \cref{prop:projections-sameweight} that $P_v$ is an orthogonal projection. In the same way one shows that they are mutually orthogonal, that is $P_v P_{v^\prime} = \delta_{v, v^\prime} P_v$.

It remains to show that $\sum_{v \in \Vset} P_v = 1$. It follows from \cref{prop:projections-sameweight} that
\[
\sum_{b_1 \in \cB(\vartheta_1)} \cdots \sum_{b_\NC \in \cB(\vartheta_\NC)} P_{b_1} \cdots P_{b_\NC} = 1.
\]
By \cref{cor:projection-rightend} we have $P_{b_1} \cdots P_{b_\NC} = 0$ unless $(b_1, \cdots, b_\NC) = \rightend_\Cset(b_1^\prime \otimes \cdots \otimes b_N^\prime)$ for some $b_1^\prime \otimes \cdots \otimes b_N^\prime$ in the Cartan component of $\cB(\vartheta_1) \otimes \cdots \otimes \cB(\vartheta_\NC)$.
Since the latter can be identified with $\cB(\rho_\Cset)$ we obtain the result.
\end{proof}

Next we look at the condition (KP2), which we divide into two parts for convenience.

\begin{lemma}
\label{lem:vertex-edge-relations}
For any $e \in \hgraph$ we have
\[
P_{\rmap(e)} \edge_e = \edge_e = \edge_e P_{\smap(e)}.
\]
\end{lemma}

\begin{proof}
Consider a path $e = (v, b)$ with range $\rmap(e) = (c_1, \cdots, c_\NC)$ where $c_i \in \cB(\vartheta_i)$.
Then using \cref{lem:commutation-projection-v} we obtain
\[
P_{\rmap(e)} \edge_e = P_{\rmap(e)} v_b P_v = \sum_{v'} v_b P_{v^\prime} P_v,
\]
where the sum is over all $v' = (c_1', \cdots, c_\NC') \in \cB(\vartheta_1) \times \cdots \times \cB(\vartheta_\NC)$ satisfying the condition $\rightend_{\vartheta_i}(c_i' \otimes b) = c_i$ for $i = 1, \cdots, \NC$.
Note that $v' = v$ satisfies this condition, by definition of the range map.
Then using the relation $P_{v'} P_v = \delta_{v', v} P_v$ we obtain $P_{\rmap(e)} \edge_e = v_b P_v = \edge_e$.
The second identity follows immediately from the fact that $\smap(e) = v$, since we have $\edge_e P_{\smap(e)} = v_b P_v P_v = \edge_e$.
\end{proof}

\begin{lemma}
Let $e_1, e_2 \in \hgraph$ be composable paths. Then $\edge_{e_1 \cdot e_2} = \edge_{e_1} \edge_{e_2}$.
\end{lemma}

\begin{proof}
Write $e_i = (v_i, b_i)$ for $i = 1, 2$.
We have $\smap(e_1) = v_1 = \rmap(e_2)$, since $e_1$ and $e_2$ are composable paths. Then we compute
\[
\edge_{e_1} \edge_{e_2} = \genv_{b_1} P_{v_1} \edge_{e_2} = \genv_{b_1} P_{\rmap(e_2)} \edge_{e_2} = \genv_{b_1} \edge_{e_2},
\]
where we have used \cref{lem:vertex-edge-relations}. Then we obtain $\edge_{e_1} \edge_{e_2} = \genv_{b_1} \genv_{b_2} P_{v_2}$.
Recall that the composition of paths is given by $e_1 \cdot e_2 = (v_2, \projection(b_2 \otimes b_1))$, where $\projection: \cB(\lambda) \otimes \cB(\lambda') \to \cB(\lambda + \lambda')$ denotes the projection onto the Cartan component.
Then $\genv_{b_1} \genv_{b_2} = \genv_{b_2 \otimes b_1} = \genv_{\projection(b_2\otimes b_1)}$ and we get $\edge_{e_1} \edge_{e_2} = \edge_{e_1 \cdot e_2}$.
\end{proof}

Next we consider the condition (KP3), concerning source projections.

\begin{lemma}
For any $e, e' \in \hgraph$ with $\dmap(e) = \dmap(e')$ we have
\[
\edge_e^* \edge_{e'} = \delta_{e, e'} P_{\smap(e)}.
\]
\end{lemma}

\begin{proof}
Write $e = (v, b)$ and $e' = (v', b')$ with $\dmap(e) = \dmap(e') = \lambda$, that is $b, b' \in \cB(\lambda)$.
By \eqref{eq:lambda-lambda_relations} we have $f_b v_{b'} = \delta_{b, b'} \sum_{c \in \cB(\lambda)} \Cartan(c \otimes b') P_c$ and so
\begin{equation}
\label{eq:pathstar-path-relation}
\edge_e^* \edge_{e'} = P_v f_b v_{b'} P_{v'} = \delta_{b, b'} \sum_{c \in \cB(\lambda)} \Cartan(c \otimes b') P_v P_c P_{v'}.
\end{equation}
Recall that the projections commute and we have $P_v P_{v'} = \delta_{v, v'} P_v$ for $v, v' \in \Vset$.
Then we find that $\edge_e^* \edge_{e'} = \delta_{e, e'} \sum_{c \in \cB(\lambda)} \Cartan(c \otimes b) P_c P_v$.
If $e\neq e'$ we are done, so it remains to deal with the case $e = e'$.  

We work by induction on the total length $|\lambda| = \sum_i n_i$ where $\lambda = \sum_{i = 1}^\NC n_i \vartheta_i$.  If $|\lambda| = 0$, then we have $\edge^*_e \edge_e = P_v^*P_v = P_v$ so the lemma holds.
Suppose the lemma holds for any $e'$ of degree $\lambda'$ with $|\lambda'|< n$.  Consider $e = e'' \cdot e'$ with $e''$ of length $1$ and $e'$ of length $n - 1$.
Note that $\smap(e'') = \rmap(e')$ since they are composable. Then we compute
\[
\begin{split}
\edge_e^* \edge_e & = \edge_{e'}^* \edge_{e''}^* \edge_{e''} \edge_{e'} = \edge_{e'}^* P_{\smap(e'')} \edge_{e'} = \edge_{e'}^* P_{\rmap(e')} \edge_{e'} \\
& = \edge_{e'}^* \edge_{e'} = P_{\smap(e')} = P_{\smap(e)},
\end{split}
\]
where we use the relation $P_{\rmap(e)} \edge_e = \edge_e$ from \cref{lem:vertex-edge-relations}.
\end{proof}

Finally we look at the condition (KP4), concerning range projections.
Recall that $v \hgraph^n$ denotes the set of paths of degree $n$ and range $v$.

\begin{lemma}
For any $v \in \Vset$ and $n \in \NN^\NC$ we have
\[
P_v = \sum_{e \in v \hgraph^n} \edge_e \edge_e^*.
\]
\end{lemma}

\begin{proof}
We identify $n \in \NN^\NC$ with the corresponding dominant weight $\lambda = \sum_i n_i\vartheta_i$.
Then we consider the sum $\sum_e \edge_e \edge_e^*$ over all paths of degree $\lambda$ and range $v$.

Let $e = (v', b) \in v \hgraph^n$.
From \cref{cor:projection-rightend} we have
\[
\edge_e \edge_e^* = \genv_b P_{v'} \genf_b = 
  \sum_{\substack{b' \in \cB(\rho_\Cset) \\ \rightend_\Cset(b') = v'}} \genv_b \genv_{b'} \genf_{b'} \genf_b.
\]
Summing this over all such $e = (v', b) \in v \hgraph^n$, we get
\[
\sum_{e = (v', b)\in v \hgraph^n} \edge_e \edge_e^*
 = \sum_{\substack{b \in \cB(\lambda),~ b' \in \cB(\rho_\Cset) \\ \rightend_\Cset(b' \otimes b) = v}}
 \genv_b \genv_{b'} \genf_{b'} \genf_b
\]
Recall that $\genf_{b'} \genf_b =0$ unless $b' \otimes b$ is in the Cartan component of $\cB(\rho_\Cset) \otimes \cB(\lambda)$, in which case it is equal to $\genf_c$, where $c$ is the image of $b' \otimes b$ under the projection $\cB(\rho_\Cset) \otimes \cB(\lambda) \to \cB(\lambda + \rho_\Cset)$.  Thus the above sum becomes
\[
 \sum_{e = (v', b)\in v \hgraph^n} \edge_e \edge_e^*
  = \sum_{\substack{c \in \cB(\lambda + \rho_\Cset) \\ \rightend_\Cset(c) = v}} \genv_c \genf_c = P_v,
\]
again by \cref{cor:projection-rightend}.
\end{proof}

Collecting all the results above, we obtain the following.

\begin{proposition}
\label{prop:KP_universal_map}
There is a surjective $*$-homomorphism $\KPalg \to \CalgK$, which on the generators is given by
\[
p_v \mapsto P_v, \quad
s_e \mapsto \edge_e.
\]
\end{proposition}

\begin{proof}
The existence of the $*$-homomorphism follows from the fact that the elements $P_v, \edge_e \in \CalgK$ satisfy the relations (KP1)--(KP4), as we have shown above.

To prove surjectivity, we must show that $\CalgK$ is generated by the elements $\edge_e$ with $e \in \hgraph$ and their adjoints.
Let $b \in \cB(\lambda)$ for any $\lambda \in \dominant_\Cset$. From \cref{lem:tensor_fact}, we have that $b_{w_0 \rho_\Cset} \otimes b$ is in the Cartan component of $\cB(\rho_\Cset) \otimes \cB(\lambda)$, so putting $v' = \rightend_\Cset(b_{w_0 \rho_\Cset})$, we have that $e = (v',b)$ defines a path in $\hgraph$. Therefore $S_e = \genv_b P_{v'}$ is non-zero.  Using \cref{prop:projections-sameweight} and \cref{rmk:not_a_path}, we get
\[
 v_b = \sum_{v \in \Vset} v_b P_v 
 = \sum_{\substack{v\in\Vset \\ (v, b) \in \hgraph}} S_{(v, b)}.
\]
Since $\CalgK$ is generated by the $\genv_b$'s and their adjoints, the result follows.
\end{proof}

In the next section we will show that this map is an isomorphism.


\section{Crystal limits are higher-rank graph algebras}
\label{sec:crystal-limit}

The universal maps from \cref{prop:KP_universal_map} and \cref{prop:AC_universal_map} yield a pair of surjective $*$-homomorphisms
\begin{equation}
\label{eq:KPiso}
  \KPalgstd \to \CalgKstd \to \OKO.
\end{equation}
In this section, we will show that these maps are isomorphisms, as well as their restrictions to the subalgebras
\begin{equation}
\label{eq:KPCiso}
  \KPalg \to \CalgK \to \cO[Y_{\rootset,0}]
\end{equation}
corresponding to the principal torus bundle over any flag variety of a connected compact semisimple group $K$.

We use the same notation as in \cref{sec:compact_form}, namely, for any $\lambda\in\dominant$ we fix a weight basis $\{v^\lambda_i\}_i$, sometimes denoted simply $\{v_i\}_i$, lifting the crystal $\cB(\lambda)=\{b^\lambda_i\}_i$, with $v^\lambda_1$ being the highest weight vector, and we let $\{f^\lambda_i\}_i$ be the dual basis.  We define the generating matrix coefficients $\genf^\lambda_i$ and $\genv^\lambda_i$ for $\OqAOK$ as in \eqref{eq:genf_genv}.  

\begin{lemma}
\label{lem:f_v_nonzero}
The image of every one of the generators $\genf^\lambda_i$ and $\genv^\lambda_i$ of $\OqAOK$ from \cref{def:OqAOK} under the representation $\SoibT_0$ is non-zero.
\end{lemma}

\begin{proof}
Let $w_0 = s_{i_1} \cdots s_{i_l}$ be the chosen reduced decomposition of the \linebreak longest word of the Weyl group. The (partially defined) Soibelman representation 
\[
   \SoibT_q = (\Soib_{i_1, q} \otimes \Soib_{i_2, q} \otimes \cdots \otimes \Soib_{i_l, q} \otimes \repT) \circ \Delta^{(l)}: \OqK \to \cB(\SoibH).
\]
maps $\OqAOK$ to $\cB(\ell^2\NN)^{\otimes l}\otimes C(T)$. 
Since $\genf^\lambda_j = c^{V(\lambda)}_{f^j, v_1}$ we have
\[
  \Delta^{(l)}(\genf^\lambda_j) = \sum_{k_1, \cdots, k_l}
    c^{V(\lambda)}_{f^j, v_{k_1}} \otimes c^{V(\lambda)}_{f^{k_1}, v_{k_2}} \otimes \cdots \otimes c^{V(\lambda)}_{f^{k_l}, v_1}.
\]
Let $\counit_T: C(T) \to \CC$ denote the evaluation at the identity in $T$. Then
\begin{equation}
    \label{eq:Soib_nonvanishing}
  (\id^{\otimes l} \otimes \counit_T) \circ \SoibT_0 (\genf_{b_j})  = \sum_{k_1, \cdots, k_{l - 1}}
    \Soib_{i_1, 0}(c^{V(\lambda)}_{f^j, v_{k_1}}) \otimes \Soib_{i_2, 0}(c^{V(\lambda)}_{f^{k_1}, v_{k_2}}) \otimes \cdots \otimes \Soib_{i_l, 0}(c^{V(\lambda)}_{f^{k_{l-1}}, v_1}).
\end{equation}
With respect to the standard basis for $\ell^2(\NN)$, the operators $\Soib_0(c^{V(\lambda)}_{f^i, v_j})$ have all coefficients equal to $0$ or $1$, see \cref{thm:SL2_limit}, so there can be no cancellation in the sum \eqref{eq:Soib_nonvanishing}. Therefore it suffices to check that at least one of the terms in the sum is non-zero.  

For this, we use the notion of \emph{string patterns}, see for instance \cite[Chapter 11]{BumpSchilling}, which shows that if $b_\lambda\in\cB(\lambda)$ is the highest weight element and $b_j\in\cB(\lambda)$ is any other element, then
there exists an $l$-tuple $(a_1, \cdots, a_l)$ of non-negative integers such that $b_\lambda = \kasE_{i_l}^{a_l} \cdots \kasE_{i_1}^{a_1} b_j$. 
Explicitly, these are given by $a_1 = \varepsilon_{i_1}(b)$, $a_2 = \varepsilon_{i_2}(\kasE_{i_1}^{a_1} b)$, \emph{etc}.
Therefore $b_{j_m} := \kasE_{i_m}^{a_l} \cdots \kasE_{i_1}^{a_1} b$
for $m=1,\cdots,l$ is a highest weight crystal element for the restriction to $\cU_{q_{i_m}}(\lie{su}(2))$ associated to the simple root $\alpha_{i_m}$.  Therefore, by \cref{thm:SL2_limit} we have
\begin{equation*}
    \Soib_{i_1,0}(c^{V(\lambda)}_{f^j,v_{j_1}}) \otimes \Soib_{i_2,0}(c^{V(\lambda)}_{f^{j_1},v_{j_2}}) \otimes \cdots \otimes \Soib_{i_l,0}(c^{V(\lambda)}_{f^{j_{l-1}},v_{j_l}}) \neq 0,
\end{equation*}
and so $\SoibT_0(f^\lambda_j)\neq0$.

The result for $\SoibT_0(\genv_b)$ follows by taking adjoints.
\end{proof}

Now let $\KPiso:\KPalgstd\to\OKO$ denote the $*$-homomorphism of \eqref{eq:KPiso}.

\begin{corollary}
\label{cor:projections_nonzero}
For any $v \in \Vset$ we have $\KPiso(p_v) \neq 0$.
\end{corollary}

\begin{proof}
  By \cref{cor:projection-rightend}, and \cref{prop:AC_universal_map}, we have
  \[
   \KPiso(p_v)=\sum_i \SoibT_0(\genf^\lambda_{i})^* \SoibT_0(\genf^\lambda_{i}),
  \]
  where the sum is over all $i\in\{1,\ldots,\dim V(\lambda)$ such that $\rightend_\Cset(b_i) = v$.  The result now follows from \cref{lem:f_v_nonzero}.
\end{proof}

We are now in a position to apply the gauge invariant uniqueness theorem \cite{KumPas}, or more precisely the algebraic version due to Aranda Pino, Clark, an Huef and Raeburn \cite{Aranda:Kumjian-Pask}.
Recall that the Kumjian-Pask algebra $\KPalg$ is equipped with a natural $\ZZ^\rank$-grading, see \cite[Theorem 3.4]{Aranda:Kumjian-Pask}.  Explicitly, each generator $s_e$ is given degree $d(e)$, $s_e^*$ is given degree $-d(e)$ and the vertex projections $p_v$ are given degree $0$.  

Meanwhile, the $\bP$-grading on $\OqK$ from \cref{def:Zr-grading} is such that $\genf^\lambda_i = c^{V(\lambda)}_{f^i,v_1}$ has degree $\lambda$ and $\genv^\lambda_i$ has degree $-\lambda$.  Here, as previously, we are identifying $\ZZ^\rank$ with $\bP$ via $(n_1,\cdots,n_\rank) \mapsto \sum_i n_i\varpi_i$.

Moreover, the Soibelman representations $\SoibT_q:\OqK \to \cB(\ell^2(\NN))^{\otimes l}) \otimes C(T) \subset \cB(\SoibH)$ are all grading-preserving, in the following sense.  By Pontrjagin duality, we have $\cO[T] \cong \CC[\bP]$, which is a $\bP$-graded algebra in the obvious way.  Examining the formula for $\SoibT_q$ in \cref{def:big_cell_rep}, we see that $\SoibT_q:\OqK_\mu \mapsto \cB(\ell^2(\NN)^{\otimes l}) \otimes \cO[T]_\mu$.
Considering the limit as $q\to0$, this makes $\OKO$ into a $\bP$-graded algebra.

Now let $e=(v,b)\in\hgraphstd$ with degree $\dmap(e)=\lambda$, so that $b=b^\lambda_i\in\cB(\lambda)$ for some $i$.  The image of $s_e$ in $\CalgKstd$ is $S_e=\genv^\lambda_{b_i} P_v$.  Under the universal map $\CalgKstd\to\OKO$, the element $\genv^\lambda_{b_i}$ maps to $\SoibT_0(\genv^\lambda_i)$, which has degree $-\lambda$, while the projection $P_v$ maps to degree $0$.  It follows that the $*$-homomorphism $\KPiso:\KPalgstd\to\OKO$ is grading-reversing.

Thus, equipping $\OKO$ with the opposite grading and taking into account \cref{cor:projections_nonzero}, we can immediately apply the graded-uniqueness theorem as appearing in \cite[Theorem 4.1]{Aranda:Kumjian-Pask}.  We get the following result.

\begin{theorem}
\label{thm:Soibelman_faithful}
Let $K$ be a compact, connected, simply connected semisimple Lie group.
The map $\KPiso: \KPalgstd \to\cO[K_0]$ is an isomorphism of $*$-algebras.  Consequently, we have a $C^*$-isomorphism $C^*(\hgraphstd) \cong C(K_0)$.
\end{theorem}

By restricting to the subalgebras spanned by elements of degree $\mu\in\hgraph$ for the $\NC$-tuple of dominant weights $\Cset = (\vartheta_1,\cdots,\vartheta_\NC)$ described in \cref{sec:flag_manifolds}, we deduce the following more general result.

\begin{theorem}
\label{thm:Soibelman_faithful_flag}
Let $K$ be a compact, connected, semisimple Lie group, not necessarily simply connected.
Let $\rootset \subseteq \simpleroots$ be a set of simple roots, with $X_\rootset$ be the associated flag variety and $Y_\rootset$ the torus bundle over $X_\rootset$, as described in \cref{sec:flag_manifolds}.  Let $\Cset$ be the $\NC$-tuple of dominant weights generating the submonoid $\dominant_{K,\rootset}$ from  \cref{thm:flag_generators2}.

The map $\KPiso$ above restricts to an isomorphism of $*$-algebras $\KPalg \to \cO[Y_{\rootset,0}]$.  The degree $0$ part $\KPalg_0$ is isomorphic to $\cO[X_{\rootset,0}]$.

Taking $C^*$-completions gives an isomorphism $C^*(\hgraph) \cong C(Y_{\rootset,0})$ with gauge-invariant subalgebra $(C^*(\hgraph))_0 \cong C(X_{\rootset,0})$.
\end{theorem}

It is well-known that the family of $C^*$-algebras $C(K_q)$, with $q\in(0,\infty)$ form a continuous field of $C^*$-algebras.  The matrix coefficients $c^V_{f,v}$, with $v\in V(\lambda)$, $f\in V(\lambda)^*$,
form a generating family of continuous sections.  It follows immediately from \cref{thm:Soibelman_faithful} that this continuous field can be extended to a continuous field over $[0,\infty)$, with fibre at $0$ being $\OKO\cong\KPalgstd$.  A generating set of continuous sections is given locally near $q=0$ by the matrix coefficients belonging to the compact $\pid$-form $\OqAOK$.

Moreover, it is well-known that we have an isomorphism \linebreak  $\OKq\cong\cO[K_{q^{-1}}]$, see for instance \cite[Lemma 2.4.2]{NesTus:book}.  Therefore, we can also extend the continuous field to $q=\infty$, with fibre $\cO[K_\infty]\cong\OKO\cong\KPalgstd$. Restricting these continuous fields to the fields subalgebras generated by the matrix coefficients of appropriate simple modules, we obtain \cref{thm:continuous_field}.


\section{Further properties and examples}
\label{sec:further-properties}

We conclude with some additional remarks on the structure of the \linebreak higher-rank graphs $\hgraph$, in particular in relation to the role of the Weyl group.
First we need the following property of the Cartan braiding, which is a consequence of a similar property for the ordinary braiding $\Rhat$.

\begin{lemma}
\label{lem:commutation-weyl}
Let $\cB(\lambda)$ and $\cB(\lambda')$ be irreducible crystals.
Then the condition
\[
\braid_{\cB(\lambda), \cB(\lambda')}(b \otimes b') = b' \otimes b
\]
is equivalent to $b \otimes b'$ being in the Cartan component and $(\wt(b), \wt(b')) = (\lambda, \lambda')$.
\end{lemma}

\begin{proof}
Consider $V = V(\lambda)$ and $W = V(\lambda')$.
Using the definition of the braiding from \eqref{eq:R-matrix_convention} we have $(\Rhat_{V, W})_{i j}^{j i} = q^{-(\wt(v_i), \wt(w_j))}$ for any $v_i \in V$ and $w_j \in W$.
Taking into account that $(\wt(v_i), \wt(w_j)) \leq (\lambda, \lambda')$, we get
\begin{equation}
\label{eq:rescaled-weyl}
\lim_{q \to 0} q^{(\lambda, \lambda')} (\Rhat_{V, W})_{i j}^{j i} =
\begin{cases}
1 & (\wt(v_i), \wt(w_j)) = (\lambda, \lambda'), \\
0 & \mathrm{otherwise}.
\end{cases}
\end{equation}
Now denote by $b_i \in \cB(\lambda)$ and $c_j \in \cB(\lambda')$ the crystal elements corresponding to $v_i$ and $w_j$.
According to \cref{thm:braiding_limit}, we have $\braid(b_i \otimes c_j) = 0$ if $b_i \otimes c_j$ is not in the Cartan component, otherwise $\braid(b_i \otimes c_j) = c_k \otimes b_l$ for some $k, l$.
Then \eqref{eq:rescaled-weyl} shows that, in the case $(\wt(v_i), \wt(w_j)) = (\lambda, \lambda')$, we have $\braid(b_i \otimes c_j) = c_j \otimes b_i$.
The converse also follows in a similar way from \eqref{eq:rescaled-weyl}.
\end{proof}

\begin{remark}
Suppose $b$ and $b'$ are such that $\wt(b) = w \lambda$ and $\wt(b') = w \lambda'$ for some $w \in W$, where $W$ denotes the Weyl group. Then $W$-invariance of the inner product gives $(\wt(b), \wt(b')) = (\lambda, \lambda')$.
\end{remark}

As usual, let $\Cset=(\vartheta_1,\ldots,\vartheta_\NC)$ be a family of dominant weights.  For $w\in W$, let us write  $b_w$, for the unique element of $\cB(\rho_\Cset)$ with extremal weight $w\rho_\Cset$.

\begin{proposition}
Consider the map
\[
W \to \Vset, \quad
w \mapsto \rightend_\Cset(b_w).
\]
Write $W_{\rho_\Cset} \subset W$ for the stabilizer of the dominant weight $\rho_\Cset$.
\begin{enumerate}
\item[(1)] We have an embedding of $W / W_{\rho_\Cset}$ into $\Vset$.
\item[(2)] When $\Cset = \fundweights$ this is an embedding of $W$.
\end{enumerate}
\end{proposition}

\begin{proof}
(1) We have a map $W / W_{\rho_\Cset} \to \Vset$, since $w \rho_\Cset = w' \rho_\Cset$ if and only if $w^{-1} w' \in W_{\rho_\Cset}$.
Hence it suffices to show that $\rightend_\Cset(b_w) \neq \rightend_\Cset(b_{w'})$ whenever $b_w \neq b_{w'}$.

Identify $\cB(\rho_\Cset) = \cB(\vartheta_1 + \cdots \vartheta_\NC)$ with the Cartan component of $\cB(\vartheta_1) \otimes \cdots \otimes \cB(\vartheta_\NC)$. Then $b_w \in \cB(\rho_\Cset)$ corresponds to a unique element $b_1 \otimes \cdots \otimes b_\NC \in \cB(\vartheta_1) \otimes \cdots \otimes \cB(\vartheta_\NC)$.
Since $w \rho_\Cset = w \vartheta_1 + \cdots + w \vartheta_\NC$ has multiplicity one, being in the Weyl group orbit of $\rho_\Cset$, we find that $\wt(b_i) = w \vartheta_i$ for $i = 1, \cdots, \NC$.
It follows that $(\wt(b_i), \wt(b_j)) = (\vartheta_i, \vartheta_j)$ for any $i$ and $j$.

According to \cref{prop:rightend-sigma}, the right end $\rightend_{\vartheta_k}(b_w)$ is equal to the rightmost factor of $\braid_{\NC - 1} \circ \cdots \circ \braid_k(b_1 \otimes \cdots \otimes b_\NC)$.
Then using \cref{lem:commutation-weyl} we obtain
\[
\braid_{\NC - 1} \circ \cdots \circ \braid_k(b_1 \otimes \cdots \otimes b_\NC) = b_1 \otimes \cdots \otimes \widehat{b_k} \otimes \cdots \otimes b_\NC \otimes b_k.
\]
It follows that $\rightend_{\vartheta_k}(b_w) = b_k$ for any $k = 1, \cdots, \NC$.
Since $b_1 \otimes \cdots \otimes b_\NC$ is the unique element corresponding to $b_w$, we find that $\rightend_\Cset(b_w) \neq \rightend_\Cset(b_{w'})$ when $b_w \neq b_{w'}$.

(2) In this case the stabilizer of $\rho = \varpi_1 + \cdots + \varpi_r$ is trivial.
\end{proof}

\begin{remark}
A bit less formally we could write $\rightend_\Cset(b_w) = b_w$, where on the left-hand side $b_w$ is identified with $b_1 \otimes \cdots \otimes b_\NC \in \cB(\vartheta_1) \otimes \cdots \otimes \cB(\vartheta_\NC)$ and on the right-hand side with the $\NC$-tuple $(b_1, \cdots, b_\NC)$.
\end{remark}

In the case of the $2$-graph for $K = SU(3)$ given in \cref{ex:graph-su3}, the number of vertices coincides with the number of elements of the Weyl group.
This can be seen to hold more generally for $SU(n)$.  But it is not true in general.

\begin{example}
In \cref{ex:graph-su3} we saw that the vertices for $\lie{g} = \mathfrak{sl}_3$ are given by
\begin{align*}
v_1 & = (a_1, b_1), &
v_2 & = (a_1, b_2), & 
v_3 & = (a_2, b_1), \\
v_4 & = (a_2, b_3), &
v_5 & = (a_3, b_2), & 
v_6 & = (a_3, b_3).
\end{align*}
These correspond to the $6$ elements of the Weyl group of $\mathfrak{sl}_3$, which is the symmetric group $S_3$.
More precisely, they correspond to the elements
\[
w_1 = 1, \quad
w_2 = s_2, \quad
w_3 = s_1, \quad
w_4 = s_1 s_2, \quad
w_5 = s_2 s_1, \quad
w_6 = s_1 s_2 s_1.
\]
This can be checked by computing their action on the weights of $a_i$ and $b_j$.
\end{example}

What is special about $\lie{g} = \mathfrak{sl}_n$ is that every fundamental representation is \emph{minuscule}, that is the Weyl group acts transitively on the weights of the representation.
To show that this does not always hold, we consider the example of $\lie{g} = C_2$, the symplectic Lie algebra of rank two.

\begin{example}
We consider the Lie algebra $\lie{g} = C_2$ of rank two.
The crystal graphs of the fundamental representations are
\[
\begin{split}
\cB(\varpi_1) : \quad & a_1 \xrightarrow{1} a_2 \xrightarrow{2} a_3 \xrightarrow{1} a_4, \\
\cB(\varpi_2) : \quad & b_1 \xrightarrow{2} b_2 \xrightarrow{1} b_3 \xrightarrow{1} b_4 \xrightarrow{2} b_5.
\end{split}
\]
Here $\cB(\varpi_1)$ corresponds to a minuscule representation, while $\cB(\varpi_2)$ does not.
Indeed, the weight $\wt(b_3) = 0$ is not in the orbit of $\varpi_2$ under the Weyl group.

The crystal graph of $\cB(\varpi_1) \otimes \cB(\varpi_1)$ is given by
\begin{center}
\begin{tikzcd}[column sep=1.5em, row sep=1em]
a_1 \otimes a_1 \arrow[r, "1"] & a_2 \otimes a_1 \arrow[r, "2"] \arrow[d, "1"] & a_3 \otimes a_1 \arrow[r, "1"] & a_4 \otimes a_1 \arrow[d, "1"] \\
a_1 \otimes a_2 \arrow[d, "2"] & a_2 \otimes a_2 \arrow[r, "2"] & a_3 \otimes a_2 \arrow[d, "2"] & a_4 \otimes a_2 \arrow[d, "2"] \\
a_1 \otimes a_3 \arrow[r, "1"] & a_2 \otimes a_3 \arrow[d, "1"] & a_3 \otimes a_3 \arrow[r, "1"] & a_4 \otimes a_3 \arrow[d, "1"] \\
a_1 \otimes a_4 & a_2 \otimes a_4 \arrow[r, "2"] & a_3 \otimes a_4 & a_4 \otimes a_4
\end{tikzcd}
\end{center}
with connected components corresponding to $\cB(2 \varpi_1)$, $\cB(\varpi_2)$ and $\cB(0)$.
The crystal graph of $\cB(\varpi_2) \otimes \cB(\varpi_2)$ is given by
\begin{center}
\begin{tikzcd}[column sep=1.5em, row sep=1em]
b_1 \otimes b_1 \arrow[r, "2"] & b_2 \otimes b_1 \arrow[r, "1"] \arrow[d, "2"] & b_3 \otimes b_1 \arrow[r, "1"] \arrow[d, "2"] & b_4 \otimes b_1 \arrow[r, "2"] & b_5 \otimes b_1 \arrow[d, "2"] \\
b_1 \otimes b_2 \arrow[d, "1"] & b_2 \otimes b_2 \arrow[r, "1"] & b_3 \otimes b_2 \arrow[r, "1"] & b_4 \otimes b_2 \arrow[d, "1"] & b_5 \otimes b_2 \arrow[d, "1"] \\
b_1 \otimes b_3 \arrow[r, "2"] \arrow[d, "1"] & b_2 \otimes b_3 \arrow[r, "1"] & b_3 \otimes b_3 \arrow[d, "1"] & b_4 \otimes b_3 \arrow[r, "2"] \arrow[d, "1"] & b_5 \otimes b_3 \arrow[d, "1"] \\
b_1 \otimes b_4 \arrow[r, "2"] & b_2 \otimes b_4 \arrow[d, "2"] & b_3 \otimes b_4 \arrow[d, "2"] & b_4 \otimes b_4 \arrow[r, "2"] & b_5 \otimes b_4 \arrow[d, "2"] \\
b_1 \otimes b_5 & b_2 \otimes b_5 \arrow[r, "1"] & b_3 \otimes b_5 \arrow[r, "1"] & b_4 \otimes b_5 & b_5 \otimes b_5
\end{tikzcd}
\end{center}
with connected components $\cB(2 \varpi_2)$, $\cB(2 \varpi_1)$ and $\cB(0)$.

Next, the crystal graph of $\cB(\varpi_1) \otimes \cB(\varpi_2)$ is given by
\begin{center}
\begin{tikzcd}[column sep=1.5em, row sep=1em]
a_1 \otimes b_1 \arrow[r, "1"] \arrow[d, "2"] & a_2 \otimes b_1 \arrow[r, "2"] & a_3 \otimes b_1 \arrow[r, "1"] \arrow[d, "2"] & a_4 \otimes b_1 \arrow[d, "2"] \\
a_1 \otimes b_2 \arrow[r, "1"] & a_2 \otimes b_2 \arrow[d, "1"] & a_3 \otimes b_2 \arrow[r, "1"] & a_4 \otimes b_2 \arrow[d, "1"] \\
a_1 \otimes b_3 \arrow[d, "1"] & a_2 \otimes b_3 \arrow[r, "2"] \arrow[d, "1"] & a_3 \otimes b_3 \arrow[d, "1"] & a_4 \otimes b_3 \arrow[d, "1"] \\
a_1 \otimes b_4 \arrow[d, "2"] & a_2 \otimes b_4 \arrow[r, "2"] & a_3 \otimes b_4 \arrow[d, "2"] & a_4 \otimes b_4 \arrow[d, "2"] \\
a_1 \otimes b_5 \arrow[r, "1"] & a_2 \otimes b_5 & a_3 \otimes b_5 \arrow[r, "1"] & a_4 \otimes b_5 \\
\end{tikzcd}
\end{center}
and the crystal graph of $\cB(\varpi_2) \otimes \cB(\varpi_1)$ is given by
\begin{center}
\begin{tikzcd}[column sep=1.5em, row sep=1em]
b_1 \otimes a_1 \arrow[r, "2"] \arrow[d, "1"] & b_2 \otimes a_1 \arrow[r, "1"] & b_3 \otimes a_1 \arrow[r, "1"] & 4_1 \otimes a_1 \arrow[r, "2"] \arrow[d, "1"] & b_5 \otimes a_1 \arrow[d, "1"] \\
b_1 \otimes a_2 \arrow[r, "2"] & b_2 \otimes a_2 \arrow[r, "1"] \arrow[d, "2"] & b_3 \otimes a_2 \arrow[d, "2"] & b_4 \otimes a_2 \arrow[r, "2"] & b_5 \otimes a_2 \arrow[d, "2"] \\
b_1 \otimes a_3 \arrow[d, "1"] & b_2 \otimes a_3 \arrow[r, "1"] & b_3 \otimes a_3 \arrow[r, "1"] & b_4 \otimes a_3 \arrow[d, "1"] & b_5 \otimes a_3 \arrow[d, "1"] \\
b_1 \otimes a_4 \arrow[r, "2"] & b_2 \otimes a_4 \arrow[r, "1"] & b_3 \otimes a_4 & b_4 \otimes a_4 \arrow[r, "2"] & b_5 \otimes a_4
\end{tikzcd}
\end{center}
They have connected components $\cB(\varpi_1 + \varpi_2)$ and $\cB(\varpi_1)$.

Examining the Cartan component of the latter two diagrams, we see that the Cartan braiding $\braid: \cB(\varpi_1) \otimes \cB(\varpi_2) \to \cB(\varpi_2) \otimes \cB(\varpi_1)$ restricted to the Cartan component is given by
\begin{small}
\begin{align*}
\braid(a_1 \otimes b_1) & = b_1 \otimes a_1, &
\braid(a_1 \otimes b_2) & = b_2 \otimes a_1, &
\braid(a_2 \otimes b_1) & = b_1 \otimes a_2, &
\braid(a_2 \otimes b_2) & = b_3 \otimes a_1, \\
\braid(a_2 \otimes b_3) & = b_4 \otimes a_1, &
\braid(a_2 \otimes b_4) & = b_4 \otimes a_2, &
\braid(a_3 \otimes b_1) & = b_2 \otimes a_2, &
\braid(a_3 \otimes b_2) & = b_2 \otimes a_3, \\
\braid(a_3 \otimes b_3) & = b_5 \otimes a_1, &
\braid(a_3 \otimes b_4) & = b_5 \otimes a_2, &
\braid(a_3 \otimes b_5) & = b_5 \otimes a_3, &
\braid(a_4 \otimes b_1) & = b_3 \otimes a_2, \\
\braid(a_4 \otimes b_2) & = b_3 \otimes a_3, &
\braid(a_4 \otimes b_3) & = b_4 \otimes a_3, &
\braid(a_4 \otimes b_4) & = b_4 \otimes a_4, &
\braid(a_4 \otimes b_5) & = b_5 \otimes a_4.
\end{align*}
\end{small}
Choosing the colours $\Cset = (\varpi_1, \varpi_2)$, we compute the right ends
\begin{align*}
v_1 & = (a_1, b_1), &
v_2 & = (a_1, b_2), &
v_3 & = (a_2, b_1), &
v_4 & = (a_1, b_3), &
v_5 & = (a_2, b_4), \\
v_6 & = (a_3, b_2), &
v_7 & = (a_3, b_3), &
v_8 & = (a_3, b_5), &
v_9 & = (a_4, b_4), &
v_{10} & = (a_4, b_5), 
\end{align*}
We see that $2$ of these vertices do not come from the Weyl group of $C_2$, since the Weyl group consists of $8$ elements.
The additional vertices are $v_4 = (a_1, b_3)$ and $v_7 = (a_3, b_3)$, both of which feature the non-extremal element $b_3 \in \cB(\varpi_2)$.
Note that $a_1 \otimes b_3$ does not even belong to the Cartan component of $\cB(\varpi_1) \otimes \cB(\varpi_2)$.

We omit the computations for the edges and simply report the results in \cref{fig:graph-C2}, where we split the $2$-graph with respect to the two colours for readability.
Notice that the two vertices which do not come from the Weyl group, namely $v_4$ and $v_7$, do not have any loops.

\begin{figure}[h]
\centering

\begin{tikzpicture}[
vertex/.style = {align=center, inner sep=2pt},
Rarr/.style = {->, red},
Barr/.style = {->, blue, dashed},
Rloop/.style = {->, red, out=165, in=195, loop},
Bloop/.style = {->, blue, out=15, in=-15, loop, dashed}
]
\node (v1)  at ( 0, 0) [vertex] {$v_1$};
\node (v2)  at (-2,-1) [vertex] {$v_2$};
\node (v3)  at ( 2,-1) [vertex] {$v_3$};
\node (v4)  at ( 0,-2) [vertex] {$v_4$};
\node (v5)  at (-2,-3) [vertex] {$v_5$};
\node (v6)  at ( 2,-3) [vertex] {$v_6$};
\node (v7)  at ( 0,-4) [vertex] {$v_7$};
\node (v8)  at (-2,-5) [vertex] {$v_8$};
\node (v9)  at ( 2,-5) [vertex] {$v_9$};
\node (v10) at ( 0,-6) [vertex] {$v_{10}$};

\draw [Rloop] (v1) edge (v1);
\draw [Rloop] (v2) edge (v2);
\draw [Rarr]  (v3) edge (v1);
\draw [Rloop] (v3) edge (v3);
\draw [Rarr]  (v4) edge (v2);
\draw [Rarr]  (v5) edge (v4);
\draw [Rloop] (v5) edge (v5);
\draw [Rarr]  (v6) edge[bend right=20] (v2);
\draw [Rarr]  (v6) edge (v3);
\draw [Rloop] (v6) edge (v6);
\draw [Rarr]  (v7) edge (v2);
\draw [Rarr]  (v7) edge (v3);
\draw [Rarr]  (v7) edge (v6);
\draw [Rarr]  (v8) edge (v4);
\draw [Rarr]  (v8) edge (v5);
\draw [Rloop] (v8) edge (v8);
\draw [Rarr]  (v9) edge (v4);
\draw [Rarr]  (v9) edge[bend left=20] (v5);
\draw [Rarr]  (v9) edge (v7);
\draw [Rloop] (v9) edge (v9);
\draw [Rarr]  (v10) edge[bend left] (v4);
\draw [Rarr]  (v10) edge (v5);
\draw [Rarr]  (v10) edge (v8);
\draw [Rloop] (v10) edge (v10);

\node (vb1)  at (6, 0) [vertex] {$v_1$};
\node (vb2)  at (4,-1) [vertex] {$v_2$};
\node (vb3)  at (8,-1) [vertex] {$v_3$};
\node (vb4)  at (6,-2) [vertex] {$v_4$};
\node (vb5)  at (4,-3) [vertex] {$v_5$};
\node (vb6)  at (8,-3) [vertex] {$v_6$};
\node (vb7)  at (6,-4) [vertex] {$v_7$};
\node (vb8)  at (4,-5) [vertex] {$v_8$};
\node (vb9)  at (8,-5) [vertex] {$v_9$};
\node (vb10) at (6,-6) [vertex] {$v_{10}$};

\draw [Bloop] (vb1) edge (vb1);
\draw [Barr]  (vb2) edge (vb1);
\draw [Bloop] (vb2) edge (vb2);
\draw [Bloop] (vb3) edge (vb3);
\draw [Barr]  (vb4) edge (vb1);
\draw [Barr]  (vb4) edge (vb2);
\draw [Barr]  (vb5) edge[bend left=20] (vb3);
\draw [Barr]  (vb5) edge (vb2);
\draw [Barr]  (vb5) edge (vb4);
\draw [Bloop] (vb5) edge (vb5);
\draw [Barr]  (vb6) edge (vb3);
\draw [Bloop] (vb6) edge (vb6);
\draw [Barr]  (vb7) edge (vb3);
\draw [Barr]  (vb7) edge (vb6);
\draw [Barr]  (vb8) edge (vb3);
\draw [Barr]  (vb8) edge[bend left=20] (vb6);
\draw [Barr]  (vb8) edge (vb4);
\draw [Barr]  (vb8) edge (vb5);
\draw [Bloop] (vb8) edge (vb8);
\draw [Barr]  (vb9) edge[bend right=20] (vb3);
\draw [Barr]  (vb9) edge (vb6);
\draw [Barr]  (vb9) edge (vb7);
\draw [Bloop] (vb9) edge (vb9);
\draw [Barr]  (vb10) edge (vb3);
\draw [Barr]  (vb10) edge (vb6);
\draw [Barr]  (vb10) edge (vb7);
\draw [Barr]  (vb10) edge (vb9);
\draw [Bloop] (vb10) edge (vb10);
\end{tikzpicture}

\caption{The $2$-graph for $\lie{g} = C_2$. On the left are the edges for the colour $\varpi_1$ and on the right for the colour $\varpi_2$.}
\label{fig:graph-C2}
\end{figure}
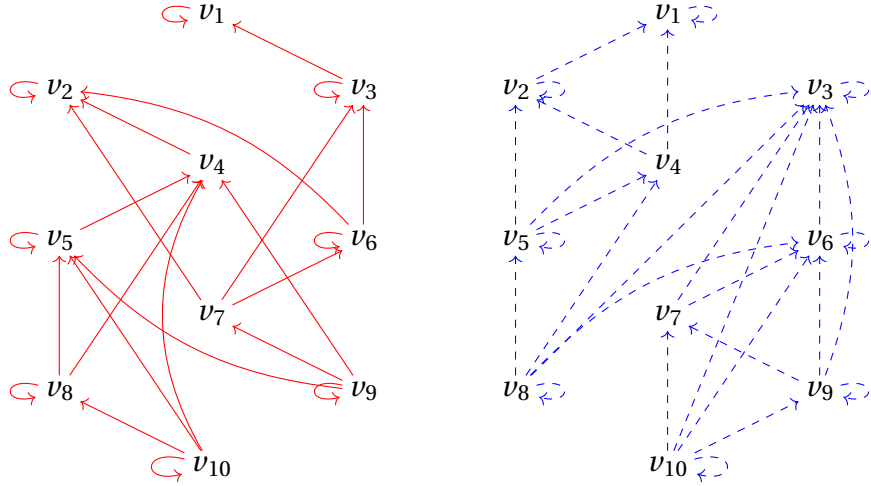

\end{example}

We conclude with one more example corresponding to a quantum homogeneous spaces, to connect with some existing literature.

\begin{example}
Consider the Lie algebra $\lie{g} = \mathfrak{sl}_n$ and the single colour $\Cset = (\varpi_1)$.
Geometrically this example corresponds to $Y_S = S^{2 n - 1}$ being an odd-dimensional sphere and $X_S = \CC P^{n - 1}$, with notation as in \cref{sec:flag_manifolds}.

In this case we have $\rho_\Cset = \varpi_1$, which obviously implies that $\rightend_\Cset(b) = b$ for every $b \in \cB(\rho_\Cset)$.
Therefore the vertex set of $\hgraph$ can be identified with the weights of the minuscule representation $V(\varpi_1)$.

Let us label the crystal basis by $\{b_1, \cdots, b_n\}$, where $b_{i + 1} = \kasF_i b_i$.
It is easy to see that $b_i \otimes b_j$ is in the Cartan component if and only if $i \geq j$ (in fact, one can generalize this to any minuscule representation).
Therefore there is a single edge from $b_i$ to $b_j$ for every $i\geq j$.
This graph coincides with the one given in \cite[Theorem 4.4]{HonSzy:spheres} (up to switching sources and ranges of edges, to match up with our conventions).

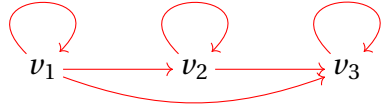
\begin{figure}[h]
\centering

\begin{tikzpicture}[
vertex/.style = {align=center, inner sep=2pt},
Rarr/.style = {->, red},
Rloop/.style = {->, red, out=135, in=45, loop},
]
\node (v1)  at (0,0) [vertex] {$v_1$};
\node (v2)  at (2,0) [vertex] {$v_2$};
\node (v3)  at (4,0) [vertex] {$v_3$};

\draw [Rloop] (v1) edge (v1);
\draw [Rloop] (v2) edge (v2);
\draw [Rloop] (v3) edge (v3);
\draw [Rarr]  (v1) edge (v2);
\draw [Rarr]  (v2) edge (v3);
\draw [Rarr]  (v1) edge[bend right=20] (v3);

\end{tikzpicture}

\caption{The graph for the case $n = 3$, corresponding to the sphere $S^5$.}
\label{fig:graph-projective}
\end{figure}
\end{example}

\bibliographystyle{alpha}
\bibliography{refs}

\end{document}